\documentclass{article}
\usepackage[utf8]{inputenc}
\usepackage{amsmath,amsthm,amsfonts,amssymb,bm, mathdots}
\usepackage{xcolor}
\usepackage{geometry}
\usepackage{tikz}
\usepackage{ytableau}
\usepackage{authblk}
\usepackage{enumitem}
\usepackage{indentfirst}
\usepackage{url}
\usepackage{fdsymbol}
\usepackage{adjustbox}

\definecolor{brightpink}{rgb}{1.0, 0.0, 0.5}

\theoremstyle{plain}
\newtheorem{thm}{Theorem}[section]
\newtheorem{prop}[thm]{Proposition}
\newtheorem{cor}[thm]{Corollary}
\newtheorem{lem}[thm]{Lemma}
\newtheorem{example}[thm]{Example}
\newtheorem{definition}[thm]{Definition}
\newtheorem{remark}[thm]{Remark}

\def\diag{ \begin{tikzpicture} \draw[dashed] (-.12,-.12) -- (.42, .42); \end{tikzpicture} }

\newcommand*\circled[2]{\tikz[baseline=(char.base)]{
            \node[shape=circle,draw=#2,inner sep=2pt] (char) {#1};}}

\newcommand{\speciallattice}{
\resizebox{4cm}{!}{
\begin{tikzpicture}[baseline=(current bounding box.center)]
\draw (0,0) -- (0,3); \draw (1,0) -- (1,3); \draw (2,0) -- (2,3); \draw (3,0) -- (3,3); 
\draw (0,0) -- (3,0); \draw (0,1) -- (3,1); \draw (0,2) -- (3,2); \draw (0,3) -- (3,3);
\node[left] at (0,0.5) {$\0$}; \node[left] at (0,1.5) {$\vdots$}; \node[left] at (0,2.5) {$\0$}; 
\node[right] at (3,0.5) {$\0$}; \node[right] at (3,1.5) {$\vdots$}; \node[right] at (3,2.5) {$\0$}; 
\node[above] at (0.5,3) {$\beta(r)$}; \node[above] at (1.5,3) {$\hdots$}; \node[above] at (2.5,3) {$\beta(s)$};
\node[below] at (0.5,0) {$\gamma(r)$}; \node[below] at (1.5,0) {$\hdots$}; \node[below] at (2.5,0) {$\gamma(s)$};
\node at (0.5,0.5) {$x_1$}; \node at (1.5,0.5) {$\hdots$}; \node at (2.5,0.5) {$x_1$};
\node at (0.5,2.5) {$x_n$}; \node at (1.5,2.5) {$\hdots$}; \node at (2.5,2.5) {$x_n$};
\node at (0.5,1.5) {$\vdots$}; \node at (2.5,1.5) {$\vdots$}; 
\end{tikzpicture}}
}

\newcommand{\exone}[7]{
	\begin{tabular}{c}
	\resizebox{4cm}{!}{
	\ytableausetup{nosmalltableaux}
	\ytableausetup{nobaseline}
	\begin{ytableau}
	\none & \none & \none & \none & \none & \none[-1] & \none[0] & \none[1] & \none[2] \\	
	\none & \none &\none & \none & \none[\diag] &\none[\diag] &\none[\diag] &\none[\diag]  \\
	\none & \none & \none & #6 &\none[\diag] &\none[\diag] &\none[\diag] &\none \\
	\none & \none & \none[\diag] & *(lightgray) &#4 &#5 &\none &\none \\
	\none & \none[\diag] & \none[\diag] &\none[\diag] &\none[\diag]&\none &\none & \none \\
	*(lightgray) &#2 &#3 &\none[\diag] &\none &\none & \none & \none \\
	*(lightgray) & *(lightgray) &#1 &\none &\none & \none & \none & \none
	\end{ytableau}}
	\\\\
	\resizebox{3cm}{!}{
	\begin{tikzpicture}
	\draw (0,0) -- (0,2); \draw (1,0) -- (1,2); \draw (2,0) -- (2,2); \draw (3,0) -- (3,2); \draw (4,0) -- (4,2); \draw (5,0) -- (5,2);
	\draw (0,0) -- (5,0); \draw (0,1) -- (5,1); \draw (0,2) -- (5,2);
	\draw[blue] (3.5,0)--(3.5,0.5+#1-1)--(4.5,0.5+#1-1)--(4.5,2);
	\draw[blue] (1.4,0)--(1.4,0.6+#2-1)--(2.4,0.6+#2-1)--(2.4,0.6+#3-1)--(3.4,0.6+#3-1)--(3.4,2);
	\draw[red] (2.7,0)--(2.7,0.3+#4-1)--(3.7,0.3+#4-1)--(3.7,0.3+#5-1)--(4.7,0.3+#5-1)--(4.7,2);
	\draw[red] (0.6,0)--(0.6,0.4+#6-1)--(1.6,0.4+#6-1)--(1.6,2);
	\end{tikzpicture}}
	\\
	#7
	\end{tabular}
}

\newcommand{\modexone}[6]{
	\begin{tabular}{ccc}
	\resizebox{4cm}{!}{
	\ytableausetup{nosmalltableaux}
	\ytableausetup{nobaseline}
	\begin{ytableau}
	\none & \none & \none & \none & \none & \none[-1] & \none[0] & \none[1] & \none[2] \\	
	\none & \none &\none & \none & \none[\diag] &\none[\diag] &\none[\diag] &\none[\diag]  \\
	\none & \none & \none & #6 &\none[\diag] &\none[\diag] &\none[\diag] &\none \\
	\none & \none & \none[\diag] & *(lightgray) &#4 &#5 &\none &\none \\
	\none & \none[\diag] & \none[\diag] &\none[\diag] &\none[\diag]&\none &\none & \none \\
	*(lightgray) &#2 &#3 &\none[\diag] &\none &\none & \none & \none \\
	*(lightgray) & *(lightgray) &#1 &\none &\none & \none & \none & \none
	\end{ytableau}} & $\leftrightarrow$ &
	\resizebox{4cm}{!}{
	\begin{tikzpicture}[baseline=(current bounding box.center)]
	\draw (0,0) -- (0,2); \draw (1,0) -- (1,2); \draw (2,0) -- (2,2); \draw (3,0) -- (3,2); \draw (4,0) -- (4,2); \draw (5,0) -- (5,2);
	\draw (0,0) -- (5,0); \draw (0,1) -- (5,1); \draw (0,2) -- (5,2);
	\node[above] at (0.5,2) {$-1$}; \node[above] at (1.5,2) {$0$}; \node[above] at (2.5,2) {$1$}; \node[above] at (3.5,2) {$2$}; \node[above] at (4.5,2) {$3$}; 
	\draw[blue] (3.5,0)--(3.5,0.5+#1-1)--(4.5,0.5+#1-1)--(4.5,2);
	\draw[blue] (1.4,0)--(1.4,0.6+#2-1)--(2.4,0.6+#2-1)--(2.4,0.6+#3-1)--(3.4,0.6+#3-1)--(3.4,2);
	\draw[red] (2.7,0)--(2.7,0.3+#4-1)--(3.7,0.3+#4-1)--(3.7,0.3+#5-1)--(4.7,0.3+#5-1)--(4.7,2);
	\draw[red] (0.6,0)--(0.6,0.4+#6-1)--(1.6,0.4+#6-1)--(1.6,2);
	\end{tikzpicture}}
	\end{tabular}
}

\newcommand{\Lemptyexone}{
	\resizebox{4cm}{!}{
	\ytableausetup{nosmalltableaux}
	\ytableausetup{nobaseline}
	\begin{ytableau}
	\none & \none & \none & \none & \none & \none[-1] & \none[0] & \none[1] & \none[2] \\	
	\none & \none &\none & \none & \none[\diag] &\none[\diag] &\none[\diag] &\none[\diag]  \\
	\none & \none & \none & &\none[\diag] &\none[\diag] &\none[\diag] &\none \\
	\none & \none & \none[\diag] & *(lightgray) & & &\none &\none \\
	\none & \none[\diag] & \none[\diag] &\none[\diag] &\none[\diag]&\none &\none & \none \\
	*(lightgray) & & &\none[\diag] &\none &\none & \none & \none \\
	*(lightgray) & *(lightgray) & &\none &\none & \none & \none & \none
	\end{ytableau}} 
}

\newcommand{\Remptyexone}{
	\resizebox{6cm}{!}{
	\begin{tikzpicture}[baseline=(current bounding box.center)]
	\draw (0,0) -- (0,2); \draw (1,0) -- (1,2); \draw (2,0) -- (2,2); \draw (3,0) -- (3,2); \draw (4,0) -- (4,2); \draw (5,0) -- (5,2);
	\draw (0,0) -- (5,0); \draw (0,1) -- (5,1); \draw (0,2) -- (5,2);
	\node[left] at (0,0.5) {$(0,0)$}; \node[left] at (0,1.5) {$(0,0)$};
	\node[right] at (5,0.5) {$(0,0)$}; \node[right] at (5,1.5) {$(0,0)$};
	\node[above] at (0.5,2.5) {$-1$}; \node[above] at (1.5,2.5) {$0$}; \node[above] at (2.5,2.5) {$1$}; \node[above] at (3.5,2.5) {$2$}; \node[above] at (4.5,2.5) {$3$}; 
	\node[above] at (0.5,2) {$(0,0)$}; \node[above] at (1.5,2) {$(0,1)$}; \node[above] at (2.5,2) {$(0,0)$}; \node[above] at (3.5,2) {$(1,0)$}; \node[above] at (4.5,2) {$(1,1)$};
	\node[below] at (0.5,0) {$(0,1)$}; \node[below] at (1.5,0) {$(1,0)$}; \node[below] at (2.5,0) {$(0,1)$}; \node[below] at (3.5,0) {$(1,0)$}; \node[below] at (4.5,0) {$(0,0)$};
	\node at (0.5,0.5) {$x_1$}; \node at (1.5,0.5) {$x_1$}; \node at (2.5,0.5) {$x_1$}; \node at (3.5,0.5) {$x_1$}; \node at (4.5,0.5) {$x_1$};
	\node at (0.5,1.5) {$x_2$}; \node at (1.5,1.5) {$x_2$}; \node at (2.5,1.5) {$x_2$}; \node at (3.5,1.5) {$x_2$}; \node at (4.5,1.5) {$x_2$};
	\end{tikzpicture}}
}

\newcommand{\extwo}[6]{
	\begin{tabular}{c}
	\resizebox{4cm}{!}{
	\ytableausetup{nosmalltableaux}
	\ytableausetup{nobaseline}
	\begin{ytableau}
	\none & \none & \none & \none & \none[0] & \none[1] & \none[2] \\	
	\none & \none & \none &\none[\diag] &\none[\diag] &\none[\diag] &\none \\
	\none & \none & #4 &#5 &\none[\diag]] &\none \\
    \none & \none[\diag] & \none[\diag] &\none[\diag] &\none &\none & \none & \none \\
	#1 & #2 & #3 &\none &\none & \none & \none & \none
	\end{ytableau}}
	\\\\
	\resizebox{3cm}{!}{
	\begin{tikzpicture}
	\draw (0,0) -- (0,2); \draw (1,0) -- (1,2); \draw (2,0) -- (2,2); \draw (3,0) -- (3,2); \draw (4,0) -- (4,2);
	\draw (0,0) -- (4,0); \draw (0,1) -- (4,1); \draw (0,2) -- (4,2);
	\draw[blue] (0.4,0)--(0.4,0.6+#1-1)--(1.4,0.6+#1-1)--(1.4,0.6+#2-1)--(2.4,0.6+#2-1)--(2.4,0.6+#3-1)--(3.4,0.6+#3-1)--(3.4,2);
	\draw[red] (0.7,0)--(0.7,0.3+#4-1)--(1.7,0.3+#4-1)--(1.7,0.3+#5-1)--(2.7,0.3+#5-1)--(2.7,2);
	\end{tikzpicture}}
	\\
	#6
	\end{tabular}
}

\newcommand{\twobox}{
\draw (0,0)--(0,2)--(1,2)--(1,0)--(0,0); \draw (0,1)--(1,1);
}

\newcommand{\gauche}{
\draw (0,0.5) -- (1,1.5); \draw (0,1.5) -- (1,0.5); \draw (1,0) grid (2,2);
}

\newcommand{\droite}{
\draw (2,0.5) -- (1,1.5); \draw (2,1.5) -- (1,0.5); \draw (0,0) grid (1,2);
}

\newcommand\qbin[3]{\left[\begin{matrix} #1 \\ #2 \end{matrix} \right]_{#3}}

\definecolor{brightpink}{rgb}{1.0, 0.0, 0.5}

\newcommand{\bbullet}{\textcolor{brown}{\bullet}}
\newcommand{\gbullet}{\textcolor{green}{\bullet}}
\newcommand{\pbullet}{\textcolor{brightpink}{\bullet}}
\newcommand{\bdiamond}{\textcolor{brown}{\blacktriangle}}
\newcommand{\gdiamond}{\textcolor{green}{\blacktriangle}}
\newcommand{\pdiamond}{\textcolor{brightpink}{\blacktriangle}}

\newcommand{\bg}{{\bm{\beta}/\bm{\gamma}}}
\newcommand{\N}{\mathbb{N}}
\newcommand{\0}{\mathbf{0}}
\newcommand{\1}{\mathbf{1}}
\newcommand{\I}{\bm{I}}
\newcommand{\J}{\bm{J}}
\newcommand{\K}{\bm{K}}
\renewcommand{\L}{\bm{L}}
\newcommand{\Z}{\mathbb{Z}}

\newcommand{\mc}[1]{{\mathcal #1}}

\DeclareMathOperator{\inv}{inv}
\DeclareMathOperator{\coinv}{coinv}
\DeclareMathOperator{\rev}{rev}

\DeclareMathOperator{\SSYT}{SSYT}
\DeclareMathOperator{\weight}{weight}
\DeclareMathOperator{\band}{band}

\title{A Vertex Model for LLT Polynomials}
\author{Sylvie Corteel}
\author{Andrew Gitlin}
\author{David Keating}
\author{Jeremy Meza}
\affil{Department of Mathematics, UC Berkeley USA\\
{\small\ttfamily\{corteel, andrew\_gitlin, dkeating, jdmeza\}@berkeley.edu}}

\date{\today}

\begin{document}

\maketitle

\begin{center}
    {\it Covid and shelter\\ California's awesome sky\\
A vertex model}
\end{center}

\abstract{We describe a novel Yang-Baxter integrable vertex model. From this vertex model we construct a certain class of partition functions that we show are essentially equal to the LLT polynomials of Lascoux, Leclerc, and Thibon. Using the vertex model formalism, we give alternate proofs of many properties of these polynomials, including symmetry and a Cauchy identity.}

\section{Introduction}

LLT polynomials were originally introduced by Lascoux, Leclerc and Thibon (and for whom the polynomials are eponymously named) in \cite{LLT}. They are a class of symmetric polynomials, and can be seen as a $t$-deformation of products of (skew) Schur functions. The original motivation for LLT polynomials was to study certain plethysm coefficients, and they were defined via a relationship with the Fock space representation of a quantum affine Lie algebra. The original definition expresses the LLT polynomials as a sum over $k$-ribbon tableaux, weighted with a spin statistic which arises naturally in this representation \cite{LLT,iij13}. Bylund and Haiman discovered an alternative way to model LLT polynomials, instead indexed by $k$-tuples of skew Young diagrams, weighted with an inversion statistic. The Bylund–Haiman model is described in \cite{HHLRU}. The relationship between the two definitions uses the Stanton-White correspondence \cite{stanton1985schensted}, which sends ribbon tableaux to tuples of semistandard Young tableaux. For more information on LLT polynomials, see the website of Alexandersson \cite{Alexandersson}.

While not apparent from either definition, LLT polynomials possess many astonishing properties, to name a few:
\begin{enumerate}[label=\Roman*.]
    \item \label{item-symmetric} they are symmetric in the variables $X=\{x_1,\ldots ,x_n\}$ \cite{LLT};
    \item \label{item-HL} when their indexing tuple of partitions consist of single rows, they are equal to the modified Hall-Littlewood polynomials $\widetilde{H}_\mu(X;t)$ \cite{LLT};
    \item \label{item-Cauchy} they satisfy a Cauchy-like identity \cite{lam2005ribbon}.
    \item \label{item-SP} they are Schur-positive, that is they expand as a $\N[t]$-linear combination of Schur polynomials \cite{leclerc2000, GrojnowskiHaiman}.
\end{enumerate}

The goal of the present work is to study the defining relation of the LLT polynomials, from the perspective of integrable vertex models. The study of integrable systems is a classical subject, see e.g. \cite{baxter2016exactly, reshetikhin2010lectures}. Also known as vertex models, ice models, or multiline queues, they have recently enjoyed an advent into the world of (non)symmetric polynomials \cite{brubaker2011schur, WHEELER2016543, BorodinWheeler}. Recently these models were generalized to colored vertex models \cite{borodin2018coloured,brubaker2019colored,brubaker2019colored2, buciumas2020colored, corteel2018multiline, garbali2020modified} and polyqueue tableaux \cite{corteel2018multiline, ayyer2020stationary}. We extend this catalog of vertex models to the case of LLT polynomials. 

On one hand the model itself is quite simple and elegant: it is the superposition of $k$ five vertex models and the parameter $t$ follows certain natural interactions of the $k$ models. On the other hand, the question of integrability proved to be deceptively cumbersome (see the proof of the Yang-Baxter equation in Appendix \ref{YBE-proof}). Nevertheless, we prove several new results on LLT polynomials and give new (combinatorial) proofs of several of the known listed properties above. Our motivations for doing so are fourfold:
\begin{enumerate}
    \item When $k=1$, LLT polynomials are Schur polynomials.
    Their combinatorial definition can be in terms of semistandard young tableaux or equivalently non intersecting paths. The corresponding vertex model is the five vertex model which has been widely studied. See for example \cite{burenev2020determinant, de2018limit, gulacsi1993phase, kapitonov2008five}.
    \item When the $k$-tuple of partitions are single rows, our model gives a vertex model for 
    the Hall-Littlewood polynomials. This model is simpler than the one given in \cite{garbali2020modified}.
    \item We are able to define some non symmetric analogue of LLT polynomials and to generalize LLT polynomials with an extra parameter $q$. A special case of this new family are the modified Macdonald polynomials. These results will appear in the second paper of this series \cite{CGKM2}.
    \item Last but not least we expect to generalize asymptotic results on lozenge tilings and Schur polynomials \cite{gorin2015asymptotics, novak2015lozenge, petrov2014asymptotics, petrov2015asymptotics} by computing the limit shape of our vertex models. These results will appear in the third paper of this series \cite{CGKM3}.
\end{enumerate}

Our main result can be summarized as follows.
\begin{thm}
Let $\bg$ be a tuple of skew partitions. There is an integrable vertex model whose partition function $\mc{Z}_{\bg}(x_1, \ldots, x_n;t)$ is precisely the coinversion LLT polynomial $\mc{L}_{\bg}(x_1, \ldots, x_n; t)$.
\end{thm}
Here, we present a new formulation of LLT polynomials, which we call a \textit{coinversion LLT polynomial}. They serve as a generating function for $k$-tuples of semistandard Young tableaux, weighted with a \textit{coinversion} statistic. The definition is easily seen to be equivalent to the inversion definition after inverting $t$ and multiplying by a suitable power of $t$. This new formulation was detailed to the fourth author in personal correspondence with M. Haiman, and can be reviewed in the first of a recent series of publications \cite{shuffle}. In the sections below, we strive to follow the notation therein, but deviate slightly in order to give a self-contained treatment of the material.

We take the time now to expound on the details of properties \ref{item-symmetric}-\ref{item-SP} and outline the organization of the paper. 

\noindent \textbf{\ref{item-symmetric} Symmetric.}
The definition of LLT polynomials as spin-generating functions arises naturally in the study of the representation theory of the Fock space for $\mathcal{U}_q(\widehat{\mathfrak{sl}}_n)$. There are natural vertex operators on this space whose action on basis elements is captured by the LLT polynomials. The fact that LLT polynomials are symmetric follows from the commutativity of these vertex operators. Later, a purely combinatorial proof was given using the inverison variant LLT polynomials \cite{HHLsym}.

As innocuous as it is, the fact that LLT polynomials are symmetric has led to several important uses in combinatorics. In \cite{HHLRU}, the authors conjectured a combinatorial formula for the Frobenius character of the ring of diagonal coinvariants (known later as the shuffle conjecture). This Frobenius character is inherently symmetric, owing to a natural $S_n$-module structure for the diagonal coinvariant ring. The combinatorial formula was shown to expand into LLT polynomials, thus witnessing its symmetricity. A similar argument was used later in \cite{HHLsym} to show that a proposed monomial expansion for Macdonald polynomials was indeed symmetric.

We use the integrability of our vertex model to provide another proof that LLT polynomials are symmetric. Our proof is combinatorial, modulo the underlying representation theory governing the Yang-Baxter equation.

\noindent \textbf{\ref{item-HL} Single Rows.}
From the definition of coinversion LLT polynomials (Definition \ref{coinv-LLT} below), one can easily see that at $t=1$, the definition devolves into a product of Schur functions. Hence, the LLT polynomial $\mc{L}_{\bm{\mu}}(X;t)$ indexed by the tuple of partitions $\bm{\mu} = (\mu^{(1)}, \ldots, \mu^{(k)})$ gives a $t$-analog $c^{\lambda}_{\mu^{(1)}, \ldots, \mu^{(k)}}(t)$ of the classical Littlewood-Richardson coefficients.

When the partitions $\mu^{(j)}$ have only one part $\mu_j$, then $c^{\lambda}_{\mu_1, \ldots, \mu_k}(t)$ coincides (up to a power of $t$) with the Kostka-Foulkes polynomial $K_{\lambda, \mu}(t)$. The Kostka-Foulkes polynomials have seemingly endless appearances in representation theory and combinatorics; in particular they are equal to the coefficients of a Schur function on the basis of transformed Hall-Littlewood polynomials $H_\mu(X;t)$, themselves defined as the dual basis to the classical Hall-Littlewood polynomials $\{P_\mu(X;t)\}$ with respect to the standard Hall inner product on the ring of symmetric functions. Hence, the identity
\begin{equation} \label{eq:L-HL}
    \mc{L}_{\bm{\mu}}(X;t) = t^d H_\mu(X;t)
\end{equation}
holds, for some integer $d$. The original proof of (\ref{eq:L-HL}) relies on the geometry of an underlying flag variety. We apply our vertex model to provide an alternate, combinatorial proof.

\noindent \textbf{\ref{item-Cauchy} Cauchy identity.}
A Cauchy identity was given in \cite{lam2005ribbon} for the original spin-generating LLT polynomials. We prove the following Cauchy identity for coinversion LLT polynomials
\begin{equation} \label{intro-cauchy}
 \sum_{\bm{\lambda}} \mathcal{L}_{\bm{\lambda}}(X_n;t)\mathcal{L}_{\bm{\lambda}^{\text{rot}}}(Y_n;t) = \prod_{i,j=1}^n\prod_{m=0}^{k-1} \left(1-x_iy_j t^m \right)^{-1}
\end{equation}
where $\bm{\lambda}^{\text{rot}}$ denotes the tuple of partitions gotten by rotating each partition 180 degrees and then reversing the order. As remarked in \cite{lam2005ribbon}, the reader is warned that (\ref{intro-cauchy}) does not imply that the LLT polynomials form an orthogonal basis under some inner product, as they are not linearly independent.

\noindent \textbf{\ref{item-SP} Schur positivity.}
It was shown in \cite{leclerc2000} that when the LLT polynomials are indexed by tuples of partitions, then their coefficients in the Schur basis are certain affine Kazhdan-Lusztig polynomials. As it is known that these Kazhdan-Lusztig polynomials have non-negative coefficients, the result implies that LLT polynomials are Schur-positive. This argument was extended in \cite{GrojnowskiHaiman} to arbitrary skew partitions, and moreover generalized to any complex reductive Lie group. Unfortunately, it is not clear how the vertex model formalism can be used to tackle the notion of positivity.

Our paper is organized as follows. In Section \ref{sec:LLT} we define the coinversion LLT polynomials and review the necessary combinatorial preliminaries. In Section \ref{sec:latticemodel} we construct our vertex model and show its partition function is a coinversion LLT polynomial. The integrability of our model is stated in Section \ref{sec:YBE}, with the proof postponed to Appendix \ref{YBE-proof}. In Section \ref{sec:one-row} we apply our model to show that our LLT polynomials coincide with Hall-Littlewood polynomials, and in Section \ref{sec:Cauchy} we use the model to show they satisfy the above Cauchy identity. Finally, as our coinversion LLT polynomials are a new variant of the familiar LLT polynomials in the current literature, we opt to provide several examples for comparison in Appendix \ref{AppendixExamples}.

\subsection*{Acknowledgements.} JM wants to thank Mark Haiman for constructive comments and
explanations during the elaboration of the work.
All authors want to thank the participants of the reading seminar of the Spring 2020 at UC Berkeley
on "Combinatorial and Integrable Models for Macdonald polynomials".
SC, AG and DK are partially funded by the UC Berkeley start-up funds.

\section{LLT Polynomials \label{sec:LLT}}

In this section we review the theory of LLT polynomials and set notation.

Fix $n$ and let $\lambda = (\lambda_1 \geq \cdots \geq \lambda_m \geq 0)$ be a partition with $m$ parts. Note that we consider our partitions to have a fixed number of parts, but allow for the possibility of parts of zero. We let the length $\ell(\lambda)$ be the number of non-zero parts of $\lambda$. We associate to $\lambda$ its Young (or Ferrers) diagram $D(\lambda) \subseteq \Z \times \Z$, given as
\[ D(\lambda) = \{(i,j) \mid 1 \leq i \leq \ell(\lambda), \; 1 \leq j \leq \lambda_i \} \]
We draw our diagrams in French notation, in the first quadrant, such as below
\[ \lambda = (4,2,1), \qquad D(\lambda) = 
\ytableausetup{aligntableaux=center}
\begin{ytableau} \\ & \\ & & \bullet & \end{ytableau} \]
We refer to the elements in $D(\lambda)$ as \textbf{cells} or \textbf{boxes}. The cell labelled above has coordinates (1,3). 

In what follows we will use $\lambda$ and $D(\lambda)$ interchangeably, when it will not cause confusion. We will also make frequent use of the staircase partition $\rho_m := (m-1, \ldots, 1, 0)$, which has the property that $\lambda + \rho_m$ has distinct parts, for any partition $\lambda$ with $m$ non-negative parts. In what follows we will also use the finite set of variables $X_n = \{x_1, \ldots, x_n\}$. For ease of notation, we drop the subscripts $n, m$ when it is clear from context.

The \textbf{content} of a cell $u = (i,j)$ in row $i$ and column $j$ of any Young diagram is $c(u) = j-i$. Given a tuple $\bg = (\beta^{(1)}/\gamma^{(1)}, \ldots, \beta^{(k)}/\gamma^{(k)})$ of skew partitions, define a semistandard Young tableau $T$ of shape $\bg$ to be a semistandard Young tableau on each $\beta^{(j)}/\gamma^{(j)}$, that is,
\[ \SSYT(\bg) = \SSYT(\beta^{(1)}/\gamma^{(1)}) \times \cdots \times \SSYT(\beta^{(k)}/\gamma^{(k)}) \]
We can picture this as placing the Young diagrams diagonally ``on content lines" with the first shape in the South-West direction and the last shape in the North-East direction. See Example~\ref{exampleLLT} below.

\begin{example}
	Let $\bg = ((3,1), (2,2,2)/(1,1,1), (1), (2,1)/(2))$. The top row labels the contents of each line.
	
	\ytableausetup{nosmalltableaux}
	\ytableausetup{nobaseline}
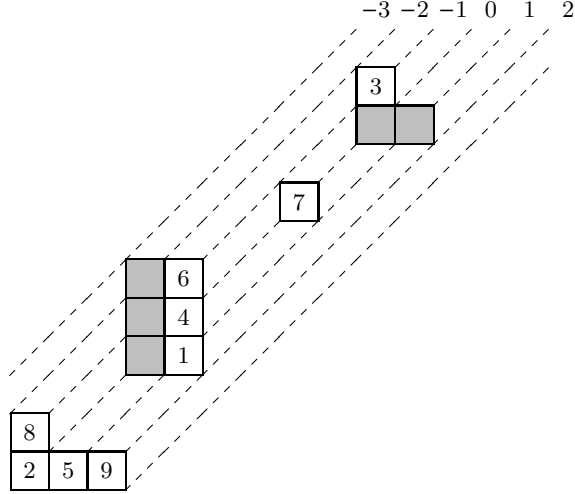
	\begin{center}
	\begin{ytableau}
		\none & \none & \none & \none & \none & \none & \none & \none & \none & \none[-3] & \none[-2] & \none[-1] & \none[0] & \none[1] & \none[2] \\			
		\none & \none & \none & \none & \none & \none & \none & \none &\none[\diag] &\none[\diag] &\none[\diag] &\none[\diag] & \none[\diag] & \none[\diag] \\
		\none & \none & \none & \none & \none & \none & \none &\none[\diag] &\none[\diag] & 3 &\none[\diag] &\none[\diag] & \none[\diag] & \none[\diag] \\
		\none & \none & \none & \none & \none &\none &\none[\diag] &\none[\diag] & \none[\diag] & *(lightgray) & *(lightgray) &\none[\diag] & \none[\diag] \\
		\none & \none & \none & \none &\none &\none[\diag] &\none[\diag] &\none[\diag] &\none[\diag] &\none[\diag] &\none[\diag] &\none[\diag] & \none \\
		\none & \none & \none &\none &\none[\diag] &\none[\diag] &\none[\diag] & 7 &\none[\diag] &\none[\diag] &\none[\diag] &\none \\
		\none & \none &\none &\none[\diag] &\none[\diag] &\none[\diag] &\none[\diag] &\none[\diag] &\none[\diag] &\none[\diag] &\none &\none \\
		\none & \none &\none[\diag] & *(lightgray) & 6 &\none[\diag] &\none[\diag] &\none[\diag] &\none[\diag] &\none &\none & \none \\
		\none & \none[\diag] & \none[\diag] & *(lightgray) & 4 &\none[\diag] &\none[\diag] &\none[\diag] &\none &\none & \none & \none \\
		\none[\diag] & \none[\diag] & \none[\diag] & *(lightgray) & 1 &\none[\diag] &\none[\diag] &\none &\none & \none & \none & \none \\
		\none[\diag] & \none[\diag] & \none[\diag] &\none[\diag] &\none[\diag] &\none[\diag] &\none &\none & \none & \none & \none & \none \\
		8 &\none[\diag] &\none[\diag] &\none[\diag] &\none[\diag] &\none &\none & \none & \none & \none & \none & \none \\
		2 & 5 & 9 &\none[\diag] &\none &\none & \none & \none & \none & \none & \none & \none \\
	\end{ytableau}	
	\end{center}
	\label{exampleLLT}
\end{example}

Let $T = (T^{(1)}, \ldots, T^{(k)})$ be a SSYT on a tuple of skew partitions. Given a cell $u$ in $T^{(r)}$, we define the \textbf{adjusted content} to be $\tilde{c}(u) = c(u)k + r-1$. We choose the reading order on cells so that their adjusted contents increase. In other words, we read from smallest to largest content line, moving along a fixed content line from the SW to NE direction.

We say two cells \textbf{attack} each other if their adjusted contents differ by less than $k$. In other words, two cells attack each other if either (1) they are on the same content line or (2) they are on adjacent content lines, with the cell on the larger content line in an earlier shape. We define an \textbf{attacking inversion} of $T$ to be a pair of attacking boxes with different entries in which the larger entry precedes the smaller in reading order.

\begin{definition} \label{inv-LLT}
Let $\bg$ be a tuple of skew partitions. The inversion LLT polynomial is the generating function
\[ \mathcal{G}_{\bg}(X; t) = \sum_{T \in \SSYT(\bg)} t^{\inv(T)} x^T \]
where $\inv(T)$ is the number of attacking inversions of $T$.
\end{definition}

\begin{remark}
Definition \ref{inv-LLT} was first given in \cite{HHLRU}, however is not evidently related to the original spin-generating functions $G^{(k)}_{\lambda/\mu}(X;q)$ defined in \cite{LLT}. The connection materializes via a correspondence due to Stanton and White \cite{stanton1985schensted}, which is a weight-preserving bijection between semistandard ribbon tableaux and tuples of semistandard Young tableaux on the quotient shape. Details on $k$-quotients, $k$-cores, and $k$-ribbons can be found for example in \cite{macdonald1998symmetric}.

It was shown in \cite{HHLRU} that if $\bg = \text{quot}(\lambda/\mu)$, then there is some constant $e$ depending only on the shape $\bg$ such that
\begin{equation} G^{(k)}_{\lambda/\mu}(X; t) = t^e \mc{G}_{\bg}(X; t^{-2}) \label{spin-inv} \end{equation}
\end{remark}

As is the case for Macdonald polynomials, the number of attacking inversions can be reformulated as the number of inversion triples, which we now define. Given a tuple $\bg$ of skew partitions, we say that three cells $u, v, w \in \Z \times \Z$ form a \textbf{triple} of $\bg$ if (i) $v \in \bg$, (ii) they are situated as below
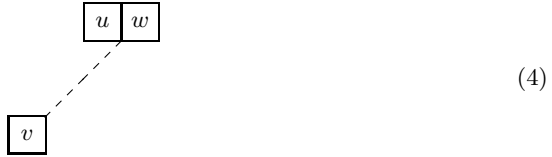
\begin{equation} \label{triple}
\ytableausetup{nobaseline}
\begin{ytableau}
\none & \none & \none &\none & u & w  \\
\none & \none & \none & \none & \none[\diag]  \\
\none & \none & \none  & \none[\diag] \\
\none & \none & v \\
\end{ytableau}
\end{equation}
namely with $v$ and $w$ on the same content line and $w$ in a later shape, and $u$ on a content line one smaller, in the same row as $w$, and (iii) if $u, w$ are in row $r$ of $\beta^{(j)}/\gamma^{(j)}$, then $u$ and $w$ must be between the cells $(r, \gamma^{(j)}_r-1), (r, \beta^{(j)}_r+1)$, inclusive. It is important to note that while $v$ must be a cell in $\bg$, we allow the cells $u$ and $w$ to not be in any of the skew shapes, in which case $u$ must be at the end of some row in $\bm{\gamma}$ and $w$ must be the cell directly to the right of the end of some row in $\bm{\beta}$. 

\begin{definition}
Let $\bg$ be a tuple of skew partitions and let $T \in \SSYT(\bg)$. Let $a, b, c$ be the entries in the cells of a triple $(u, v, w)$, where we set $a = 0$ and $c=\infty$ if the respective cell is not in $\bg$. Given the triple of entries
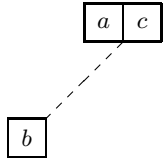
\ytableausetup{nobaseline}
\[
\begin{ytableau}
\none & \none & \none &\none & a & c  \\
\none & \none & \none & \none & \none[\diag]  \\
\none & \none & \none  & \none[\diag] \\
\none & \none & b \\
\end{ytableau}
\]
we say this is a \textbf{coinversion triple} of $T$ if $a \leq b \leq c$ and it is an \textbf{inversion triple} if $b < a \leq c$ or $a \leq c < b$.
\label{inv-coinv-triple}
\end{definition}

There are 7 coinversion triples in Example \ref{exampleLLT} above: (0, 2, 4), (0, 2, 7), (3,4,$\infty$), (0,4,7), (4,5,$\infty$), (1,9,$\infty$), and (0,9,$\infty$). However, we note that Definition \ref{inv-coinv-triple}, and that of a triple, depends not merely on the tuple of skew partitions $\bg$, but on the individual tuples of partitions $\bm{\beta}$, $\bm{\gamma}$. Indeed, if in Example \ref{exampleLLT}, we made the superficial change in the third skew shape from (1)/(0) to (2,2)/(2,1), then we would introduce another coinversion triple $(0,9,\infty)$. Likewise if we consider the third shape being instead (1,0)/(0,0), then we introduce the coinversion triples $(0,8,\infty)$ and $(0,6,\infty)$. It's easily seen that any extra coinversion triples present are independent of the filling $T$.

\begin{definition} \label{coinv-LLT}
Let $\bg$ be a tuple of skew partitions. The coinversion LLT polynomial is the generating function
\[ \mc{L}_{\bg}(X; t) = \sum_{T \in \SSYT(\bg)} t^{\coinv(T)} x^T \]
where $\coinv(T)$ is the number of coinversion triples of $T$.
\end{definition}

In light of the preceding remarks, we note that if $\bg$ and $\bm{\beta}'/\bm{\gamma}'$ are two representations of the same skew shapes, then their coinversion LLT polynomials differ by an overall power of $t$.

Note that in a semistandard filling $T$ on some tuple of skew partitions, a pair of attacking entries forms an inversion if and only if they are in a (unique) inversion triple. Indeed, if $b < a \leq c$, then $(a,b)$ is an attacking inversion, and likewise if $a \leq c < b$, then $(b,c)$ is an attacking inversion. Hence, we have the identity
\begin{equation} 
\mc{L}_{\bg}(X; t) = t^m \mc{G}_{\bg}(X; t^{-1}) \label{inv-coinv} \end{equation}
where $m$ is the total number of triples in $\bg$. Explicit formulae for $m$ will be given in Section~\ref{sec:one-row}.

\begin{remark}
A simplified version of Definition \ref{coinv-LLT}, in which each shape in $\bg$ consists of a single row, can be found in \cite{shuffle}. There, the coinversion LLT polynomials are first defined, via the action of a Hecke algebra, as a polynomial truncation of a certain formal power series. It is then shown that this algebraic definition results in the combinatorial definition above.

The formal power series in question consists of terms that are $GL_n$ characters. This is the reason why our polynomials depend on the individual tuples of partitions $\bm{\beta}$, $\bm{\gamma}$. Indeed, the reader is welcome to view $\bm{\beta}$ and $\bm{\gamma}$ not as partitions, but really as dominant weights of $GL_n$, i.e. non-increasing lists of integers. In fact, this ``LLT series" can be defined for any complex reductive Lie group, see \cite{GrojnowskiHaiman}.
\end{remark}

\section{Lattice Model \label{sec:latticemodel}}

In this section we define our lattice model. We begin by defining the local weights of the model. We then show that for a certain choice of boundary conditions the partition function of the model is equal to the coinversion LLT polynomials defined above.

\indent For any vector $\I = (I_1,...,I_k) \in \{0,1\}^k$ and indices $i,j \in [k]$, we define
\[ |\I| = \sum_{m = 1}^k I_m, \hspace{1cm} \I_{[i,j]} = \sum_{m=i}^j I_m \]
\noindent where the empty sum is 0.  The \textbf{L-matrix} is an infinite-dimensional matrix whose components are represented as faces (sometimes called vertices), with each face assigned a label in $\{0,1\}^k$ as follows:
\[ L_x(\I,\J;\K,\L) = 
\resizebox{2cm}{!}{
\begin{tikzpicture}[baseline=(current bounding box.center)]
\draw (0,0) -- (0,1); \draw (1,0) -- (1,1); 
\draw (0,0) -- (1,0); \draw (0,1) -- (1,1); 
\node[below] at (0.5,0) {$\I$}; \node[above] at (0.5,1) {$\K$};
\node[left] at (0,0.5) {$\J$}; \node[right] at (1,0.5) {$\L$}; 
\node at (0.5,0.5) {$x$};
\end{tikzpicture}}, 
\hspace{1cm} \I,\J,\K,\L \in \{0,1\}^k.
\]
\noindent The face weights are given by
\[ L_x(\I,\J;\K,\L) = x^{|\L|} \prod_{\substack{m \in [k] \\ L_m=1}} t^{\I_{[m+1,k]}+\J_{[m+1,k]}} \]
\noindent whenever $\I+\J = \K+\L$ and there is no $i \in [k]$ such that $I_i = J_i = 1$, and $L_x(\I,\J;\K,\L) = 0$ otherwise.  

\indent We can think of the faces in terms of colored paths which enter via the left/bottom edges and which exit via the right/top edges.  Whenever an edge assumes the value $\I = (I_1,...,I_k)$, then for each index $i \in [k]$, there is a path of color $i$ incident at the edge if and only if $I_i = 1$.  The constraint $\I+\J = \K+\L$ is a conservation property, meaning that the paths entering and the paths exiting must be the same.  The constraint that there be no indices $i \in [k]$ such that $I_i = J_i = 1$ means that there can be at most one path of any given color.  If these two constraints are satisfied, then the face weight can be expressed as
\begin{equation} \label{FaceWeight}
L_x(\I,\J;\K,\L) = x^{\substack{\text{\# colors exiting the} \\ \text{vertex to the right}}} \prod_{\substack{\text{colors $i$ exiting the} \\ \text{vertex to the right}}} t^{\substack{\text{\# colors larger than $i$ that} \\ \text{appear in the vertex}}}.
\end{equation}
\noindent When $k=1$, this is the non-intersecting path model (also known as the five vertex model), as illustrated in \cite[Thm. 7.16.1]{stanley1999enumerative}. 

\begin{example}
Let $k = 3$.  Let blue be color 1, red be color 2, and green be color 3.  Then
\[ L_x((1,0,0),(0,1,1);(0,0,1),(1,1,0)) = 
\resizebox{2cm}{!}{
\begin{tikzpicture}[baseline=(current bounding box.center)]
\draw (0,0) -- (0,1); \draw (1,0) -- (1,1); 
\draw (0,0) -- (1,0); \draw (0,1) -- (1,1); 
\node[below] at (0.5,0) {$(1,0,0)$}; \node[above] at (0.5,1) {$(0,0,1)$};
\node[left] at (0,0.8) {$(0,$}; \node[left] at (0,0.5) {$1,$}; \node[left] at (0,0.2) {$1)$}; 
\node[right] at (1,0.8) {$(1,$}; \node[right] at (1,0.5) {$1,$}; \node[right] at (1,0.2) {$0)$}; 
\node at (0.5,0.5) {$x$};
\draw[blue] (0.2,0)--(0.2,0.8)--(1,0.8);
\draw[red] (0,0.5)--(1,0.5);
\draw[green] (0,0.2)--(0.8,0.2)--(0.8,1);
\end{tikzpicture}}
= x^2t^3. \]
\end{example}

\indent A \textbf{lattice} is a rectangular grid of faces, with the variables and the labels on the outer edges specified, but with the labels on the internal edges unspecified.  A \textbf{lattice configuration} is a lattice with the labels on the internal edges specified, such that the weight of each face is non-zero.  The weight of a lattice configuration is the product of the face weights.  Thinking of lattice configurations in terms of colored paths, the constraint that the weight of each face be non-zero is equivalent to the following two constraints.
\begin{enumerate}
\item At each vertex, the paths entering and the paths exiting must be the same.
\item At each vertex, there can be at most one path of any given color.
\end{enumerate}

\begin{example}
Let blue be color 1, red be color 2, and green be color 3.  Then
\[ 
\resizebox{3cm}{!}{
\begin{tikzpicture}[baseline=(current bounding box.center)]
\draw (0,0) -- (0,2); \draw (1,0) -- (1,2); \draw (2,0) -- (2,2);
\draw (0,0) -- (2,0); \draw (0,1) -- (2,1); \draw (0,2) -- (2,2);
\node at (0.2,0) {$1$}; \node at (0.5,0) {$0$}; \node at (0.8,0) {$0$}; 
\node at (0.2,1) {$0$}; \node at (0.5,1) {$0$}; \node at (0.8,1) {$1$};
\node at (0,0.8) {$0$}; \node at (0,0.5) {$1$}; \node at (0,0.2) {$1$}; 
\node at (1,0.8) {$1$}; \node at (1,0.5) {$1$}; \node at (1,0.2) {$0$}; 
\node at (0.5,0.5) {$x_1$};
\draw[blue] (0.2,0)--(0.2,0.8)--(1,0.8);
\draw[red] (0,0.5)--(1,0.5);
\draw[green] (0,0.2)--(0.8,0.2)--(0.8,1);
\node at (1.2,0) {$0$}; \node at (1.5,0) {$0$}; \node at (1.8,0) {$0$}; 
\node at (1.2,1) {$0$}; \node at (1.5,1) {$1$}; \node at (1.8,1) {$0$};
\node at (2,0.8) {$1$}; \node at (2,0.5) {$0$}; \node at (2,0.2) {$0$}; 
\node at (1.5,0.5) {$x_1$};
\draw[blue] (1,0.8)--(2,0.8);
\draw[red] (1,0.5)--(1.5,0.5)--(1.5,1);
\node at (0,1.8) {$1$}; \node at (0,1.5) {$1$}; \node at (0,1.2) {$0$}; 
\node at (0.2,2) {$0$}; \node at (0.5,2) {$1$}; \node at (0.8,2) {$0$}; 
\node at (1,1.8) {$1$}; \node at (1,1.5) {$0$}; \node at (1,1.2) {$1$}; 
\node at (0.5,1.5) {$x_2$};
\draw[blue] (0,1.8)--(1,1.8);
\draw[red] (0,1.5)--(0.5,1.5)--(0.5,2);
\draw[green] (0.8,1)--(0.8,1.2)--(1,1.2);
\node at (1.2,2) {$0$}; \node at (1.5,2) {$1$}; \node at (1.8,2) {$0$};
\node at (2,1.8) {$1$}; \node at (2,1.5) {$0$}; \node at (2,1.2) {$1$}; 
\node at (1.5,1.5) {$x_3$};
\draw[blue] (1,1.8)--(2,1.8);
\draw[red] (1.5,1)--(1.5,2);
\draw[green] (1,1.2)--(2,1.2);
\end{tikzpicture}}
=
\begin{tabular}{c}
$x_2^2t^2 \cdot x_3^2 t^2$ \\
$\cdot x_1^2 t^3 \cdot x_1t$
\end{tabular}
=
x_1^3x_2^2x_3^2t^8.
\]
\end{example}

\indent Given a tuple of partitions $\mu = (\mu^{(1)},...,\mu^{(k)})$ and an integer $i$, let $\mu(i) \in \{0,1\}^k$ be the vector whose $j$-th component is 1 if and only if
\[ i = \mu^{(j)}_m - m + 1 \]
\noindent for some $m \in [\ell(\mu^{(j)})]$, for each index $j \in [k]$.  Let $\bg = (\beta^{(1)}/\gamma^{(1)},...,\beta^{(k)}/\gamma^{(k)})$ be a tuple of skew partitions.  For each $i$, we assume $\beta^{(i)}$ and $\gamma^{(i)}$ have the same number of non-negative parts. Let
\begin{align*}
    &r = r(\bg) = \min \{ i \in \Z : \gamma(i) \neq \0\}, \\
    &s = s(\bg) = \max \{ i \in \Z : \beta(i) \neq \0\}.
\end{align*}
With this notation, we introduce a lattice that will be of particular interest to us:
\begin{equation} L_{\bg} := \speciallattice  \label{speciallattice} \end{equation}
Following \cite{ChrisThesis}, we define the \textbf{bandwidth} $\band(\bg):=s(\bg)-r(\bg)$ to be one less the number of columns in this lattice.  To simplify notation we will often replace $\gamma(r)\ldots\gamma(s)$ with $\bm{\gamma}$ and $\beta(r)\ldots\beta(s)$ with $\bm{\beta}$ to indicate the top and bottom boundary conditions of the lattice.

\begin{example}
Let $\bg = ((3,3)/(2,1), (3,1)/(1,0))$ and $n = 2$.  
\[ \bg = \Lemptyexone \]
We compute $r = -1$ and $s = 3$, and 
\[ L_{\bg} = \Remptyexone \]
where we include the column indices (in addition to the variables and the outer edge labels).
\end{example}

\indent We let $\mc{Z}_{\bg}(X_n; t)$ denote the partition function of $L_{\bg}$, that is 
\[ \mc{Z}_{\bg}(X_n; t) = \sum_{L \in LC_{\bg}} \weight(L).  \]
where $LC_{\bg}$ denotes the set of lattice configurations on $L_{\bg}$. 

\begin{thm} \label{LatticeModel}
Let $\bg$ be a tuple of skew partitions. Then,
\[ \mc{Z}_{\bg}(X_n; t) = \mathcal{L}_\bg(X_n;t). \]
\end{thm}

Some examples of using this theorem to compute LLT polynomials are given in Appendix \ref{AppendixExamples}.

This theorem will follow from a weight-preserving bijection $\varphi : \SSYT(\bg) \rightarrow LC_{\bg}$. Given $T \in \SSYT(\bg)$, the corresponding lattice configuration $L = \varphi(T)$ is constructed as follows.  Each row in $T$ corresponds to a colored path in $L$.  Fix a row
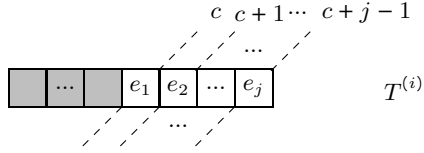
\begin{center}
\ytableausetup{nosmalltableaux}
\ytableausetup{nobaseline}
\begin{ytableau}
\none & \none & \none & \none & \none & \none[c] & \none[\enspace c+1] & \none[\enspace...] & \none[\hspace{1cm} c+j-1] \\
\none & \none & \none & \none & \none[\diag] & \none[\diag] & \none[...] & \none[\diag] \\
*(lightgray) & *(lightgray)... & *(lightgray) & e_1 & e_2 & ... & e_j & \none & \none & \none & \none[T^{(i)}] \\
\none & \none & \none[\diag] & \none[\diag] & \none[...] & \none[\diag]
\end{ytableau}
\end{center}
\noindent in $T$.  The corresponding path in $L$ has color $i$, enters via the bottom of column $c$ and exits via the top of column $c+j$, and crosses from column $c+m-1$ to column $c+m$ at row $e_m$ for each index $m \in [j]$.  

\begin{example}
Let $\bg = ((3,3)/(2,1), (3,1)/(1,0))$ and $n = 2$.  Then
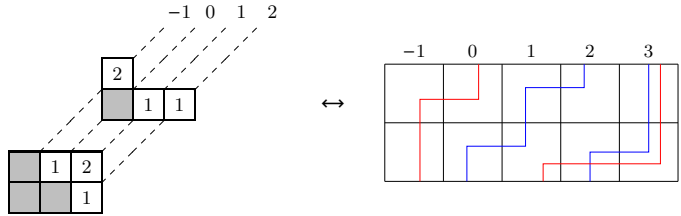
\[\modexone{1}{1}{2}{1}{1}{2}\]
where in the lattice configuration blue is color 1, red is color 2, and we have omitted the variables and the edge labels (as we will do often).  
\end{example}

The invertibility of $\varphi$ is straightforward; indeed this is a well-known fact for one color \cite[Thm. 7.16.1]{stanley1999enumerative} and the proof for $k$ colors follows from applying the proof for one color to each of the $k$ colors independently. The following proposition then completes the proof of Theorem \ref{LatticeModel}.
\begin{prop}
Let $T \in \SSYT(\bg)$ and let $\varphi(T)$ be the corresponding lattice configuration, as defined above.  Then
\[ \coinv(T) = \sum_{\text{vertices $V$ in $\varphi(T)$}} \enspace \sum_{\text{colors $i$ exiting $V$ to the right}} \enspace \text{(\# colors larger than $i$ that appear in $V$)}. \]
In particular, the weight of the lattice configuration $\varphi(T)$ is $x^T t^{\coinv(T)}$.
\end{prop}

\begin{proof} The left-hand side counts the number of coinversion triples in $T$; that is, the number of triples of boxes in $T$
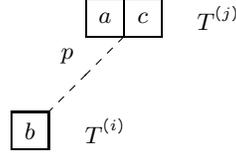
\begin{equation*}
\ytableausetup{nobaseline}
\begin{ytableau}
\none &\none & a & c & \none & \none[T^{(j)}]  \\
\none & \none[p] & \none[\diag]  \\
\none & \none[\diag] \\
b & \none & \none[T^{(i)}] \\
\end{ytableau} 
\end{equation*}
with $p$ denoting the content line, $i < j$, and $a \leq b \leq c$, where we set $a=0$ and $c=\infty$ if the respective box is empty. One can observe that the coinversion triple can be recovered just from $i,j,b,p$, so we will identify each coinversion triple with the corresponding quadruple $(i,j,b,p)$.  The right-hand side counts the number of quadruples $(i,j,b,p)$ where $V$ is the vertex in $\varphi(T)$ at row $b$ and column $p$, $i$ is a color exiting $V$, and $j$ is a color larger than $i$ that appears in $V$.  The proposition follows from the fact that the left-hand side quadruples correspond to the right-hand side quadruples via $\varphi$.  To verify this correspondence, we split into cases based on the form of the coinversion triple. \\

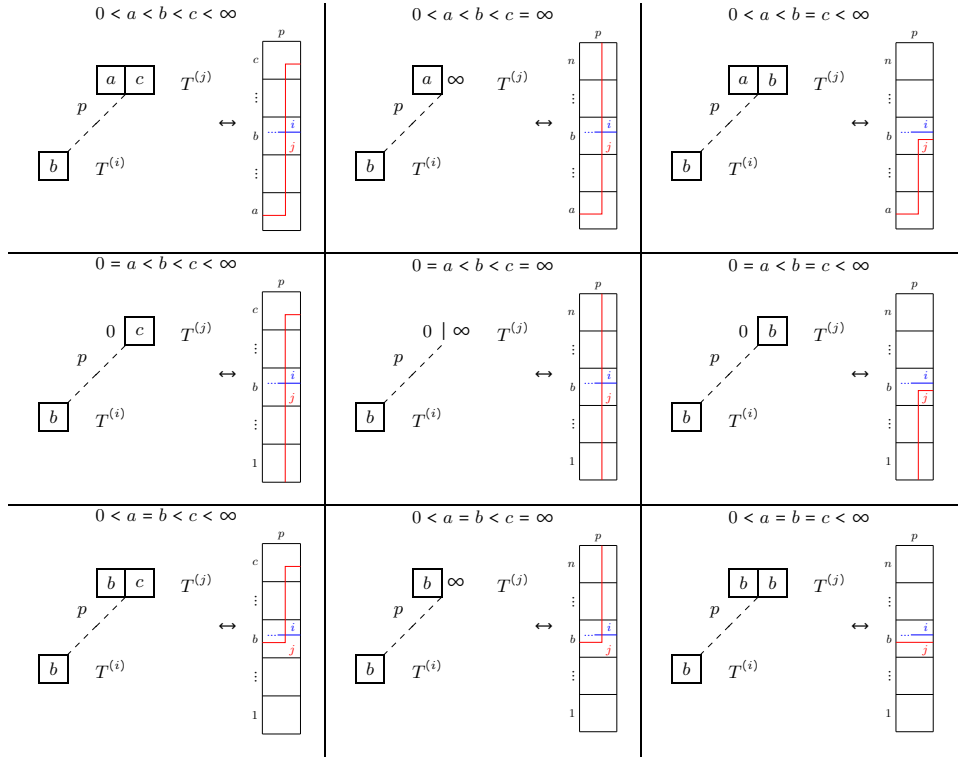
\begin{center} \noindent \resizebox{5in}{!}{
\begin{tabular}{c|c|c}
\begin{tabular}{c}
$0 < a < b < c < \infty$ \\
\ytableausetup{nobaseline}
\begin{ytableau}
\none &\none & a & c & \none & \none[T^{(j)}]  \\
\none & \none[p] & \none[\diag]  \\
\none & \none[\diag] \\
b & \none & \none[T^{(i)}] \\
\end{ytableau} $\leftrightarrow$
\resizebox{1cm}{!}{
\begin{tikzpicture}[baseline=(current bounding box.center)]
\draw (0,0) -- (0,5); \draw (1,0) -- (1,5); 
\draw (0,0) -- (1,0); \draw (0,1) -- (1,1); \draw (0,2) -- (1,2); \draw (0,3) -- (1,3); \draw (0,4) -- (1,4); \draw (0,5) -- (1,5);
\node[left] at (0,0.5) {$a$}; \node[left] at (0,1.5) {$\vdots$}; \node[left] at (0,2.5) {$b$}; \node[left] at (0,3.5) {$\vdots$}; \node[left] at (0,4.5) {$c$};
\node[above] at (0.5,5) {$p$}; \node[right] at (0.6,2.2) {\color{red}{$j$}}; \node[above] at (0.8,2.6) {\color{blue}{$i$}};
\node[left] at (0.5,2.6) {\color{blue}{...}};
\draw[blue] (0.4,2.6)--(1,2.6);
\draw[red] (0,0.4)--(0.6,0.4)--(0.6,4.4)--(1,4.4);
\end{tikzpicture}}
\end{tabular}
&
\begin{tabular}{c}
$0 < a < b < c = \infty$ \\
\ytableausetup{nobaseline}
\begin{ytableau}
\none &\none & a & \none[\infty] & \none & \none[T^{(j)}]  \\
\none & \none[p] & \none[\diag]  \\
\none & \none[\diag] \\
b & \none & \none[T^{(i)}] \\
\end{ytableau} $\leftrightarrow$
\resizebox{1cm}{!}{
\begin{tikzpicture}[baseline=(current bounding box.center)]
\draw (0,0) -- (0,5); \draw (1,0) -- (1,5); 
\draw (0,0) -- (1,0); \draw (0,1) -- (1,1); \draw (0,2) -- (1,2); \draw (0,3) -- (1,3); \draw (0,4) -- (1,4); \draw (0,5) -- (1,5);
\node[left] at (0,0.5) {$a$}; \node[left] at (0,1.5) {$\vdots$}; \node[left] at (0,2.5) {$b$}; \node[left] at (0,3.5) {$\vdots$}; \node[left] at (0,4.5) {$n$};
\node[above] at (0.5,5) {$p$}; \node[right] at (0.6,2.2) {\color{red}{$j$}}; \node[above] at (0.8,2.6) {\color{blue}{$i$}};
\node[left] at (0.5,2.6) {\color{blue}{...}};
\draw[blue] (0.4,2.6)--(1,2.6);
\draw[red] (0,0.4)--(0.6,0.4)--(0.6,5);
\end{tikzpicture}}
\end{tabular}
&
\begin{tabular}{c}
$0 < a < b = c < \infty$ \\
\ytableausetup{nobaseline}
\begin{ytableau}
\none &\none & a & b & \none & \none[T^{(j)}]  \\
\none & \none[p] & \none[\diag]  \\
\none & \none[\diag] \\
b & \none & \none[T^{(i)}] \\
\end{ytableau} $\leftrightarrow$
\resizebox{1cm}{!}{
\begin{tikzpicture}[baseline=(current bounding box.center)]
\draw (0,0) -- (0,5); \draw (1,0) -- (1,5); 
\draw (0,0) -- (1,0); \draw (0,1) -- (1,1); \draw (0,2) -- (1,2); \draw (0,3) -- (1,3); \draw (0,4) -- (1,4); \draw (0,5) -- (1,5);
\node[left] at (0,0.5) {$a$}; \node[left] at (0,1.5) {$\vdots$}; \node[left] at (0,2.5) {$b$}; \node[left] at (0,3.5) {$\vdots$}; \node[left] at (0,4.5) {$n$};
\node[above] at (0.5,5) {$p$}; \node[right] at (0.6,2.2) {\color{red}{$j$}}; \node[above] at (0.8,2.6) {\color{blue}{$i$}};
\node[left] at (0.5,2.6) {\color{blue}{...}};
\draw[blue] (0.4,2.6)--(1,2.6);
\draw[red] (0,0.4)--(0.6,0.4)--(0.6,2.4)--(1,2.4);
\end{tikzpicture}}
\end{tabular} 
\\ && \\ \hline
\begin{tabular}{c}
$0 = a < b < c < \infty$ \\
\ytableausetup{nobaseline}
\begin{ytableau}
\none &\none & \none[0] & c & \none & \none[T^{(j)}]  \\
\none & \none[p] & \none[\diag]  \\
\none & \none[\diag] \\
b & \none & \none[T^{(i)}] \\
\end{ytableau} $\leftrightarrow$
\resizebox{1cm}{!}{
\begin{tikzpicture}[baseline=(current bounding box.center)]
\draw (0,0) -- (0,5); \draw (1,0) -- (1,5); 
\draw (0,0) -- (1,0); \draw (0,1) -- (1,1); \draw (0,2) -- (1,2); \draw (0,3) -- (1,3); \draw (0,4) -- (1,4); \draw (0,5) -- (1,5);
\node[left] at (0,0.5) {$1$}; \node[left] at (0,1.5) {$\vdots$}; \node[left] at (0,2.5) {$b$}; \node[left] at (0,3.5) {$\vdots$}; \node[left] at (0,4.5) {$c$};
\node[above] at (0.5,5) {$p$}; \node[right] at (0.6,2.2) {\color{red}{$j$}}; \node[above] at (0.8,2.6) {\color{blue}{$i$}};
\node[left] at (0.5,2.6) {\color{blue}{...}};
\draw[blue] (0.4,2.6)--(1,2.6);
\draw[red] (0.6,0)--(0.6,4.4)--(1,4.4);
\end{tikzpicture}}
\end{tabular}
&
\begin{tabular}{c}
$0 = a < b < c = \infty$ \\
\ytableausetup{nobaseline}
\begin{ytableau}
\none &\none & \none[0] & \none[\mid \infty] & \none & \none[T^{(j)}]  \\
\none & \none[p] & \none[\diag]  \\
\none & \none[\diag] \\
b & \none & \none[T^{(i)}] \\
\end{ytableau} $\leftrightarrow$
\resizebox{1cm}{!}{
\begin{tikzpicture}[baseline=(current bounding box.center)]
\draw (0,0) -- (0,5); \draw (1,0) -- (1,5); 
\draw (0,0) -- (1,0); \draw (0,1) -- (1,1); \draw (0,2) -- (1,2); \draw (0,3) -- (1,3); \draw (0,4) -- (1,4); \draw (0,5) -- (1,5);
\node[left] at (0,0.5) {$1$}; \node[left] at (0,1.5) {$\vdots$}; \node[left] at (0,2.5) {$b$}; \node[left] at (0,3.5) {$\vdots$}; \node[left] at (0,4.5) {$n$};
\node[above] at (0.5,5) {$p$}; \node[right] at (0.6,2.2) {\color{red}{$j$}}; \node[above] at (0.8,2.6) {\color{blue}{$i$}};
\node[left] at (0.5,2.6) {\color{blue}{...}};
\draw[blue] (0.4,2.6)--(1,2.6);
\draw[red] (0.6,0)--(0.6,5);
\end{tikzpicture}}
\end{tabular}
&
\begin{tabular}{c}
$0 = a < b = c < \infty$ \\
\ytableausetup{nobaseline}
\begin{ytableau}
\none &\none & \none[0] & b & \none & \none[T^{(j)}]  \\
\none & \none[p] & \none[\diag]  \\
\none & \none[\diag] \\
b & \none & \none[T^{(i)}] \\
\end{ytableau} $\leftrightarrow$
\resizebox{1cm}{!}{
\begin{tikzpicture}[baseline=(current bounding box.center)]
\draw (0,0) -- (0,5); \draw (1,0) -- (1,5); 
\draw (0,0) -- (1,0); \draw (0,1) -- (1,1); \draw (0,2) -- (1,2); \draw (0,3) -- (1,3); \draw (0,4) -- (1,4); \draw (0,5) -- (1,5);
\node[left] at (0,0.5) {$1$}; \node[left] at (0,1.5) {$\vdots$}; \node[left] at (0,2.5) {$b$}; \node[left] at (0,3.5) {$\vdots$}; \node[left] at (0,4.5) {$n$};
\node[above] at (0.5,5) {$p$}; \node[right] at (0.6,2.2) {\color{red}{$j$}}; \node[above] at (0.8,2.6) {\color{blue}{$i$}};
\node[left] at (0.5,2.6) {\color{blue}{...}};
\draw[blue] (0.4,2.6)--(1,2.6);
\draw[red] (0.6,0)--(0.6,2.4)--(1,2.4);
\end{tikzpicture}}
\end{tabular}
\\ && \\ \hline
\begin{tabular}{c}
$0 < a = b < c < \infty$ \\
\ytableausetup{nobaseline}
\begin{ytableau}
\none &\none & b & c & \none & \none[T^{(j)}]  \\
\none & \none[p] & \none[\diag]  \\
\none & \none[\diag] \\
b & \none & \none[T^{(i)}] \\
\end{ytableau} $\leftrightarrow$
\resizebox{1cm}{!}{
\begin{tikzpicture}[baseline=(current bounding box.center)]
\draw (0,0) -- (0,5); \draw (1,0) -- (1,5); 
\draw (0,0) -- (1,0); \draw (0,1) -- (1,1); \draw (0,2) -- (1,2); \draw (0,3) -- (1,3); \draw (0,4) -- (1,4); \draw (0,5) -- (1,5);
\node[left] at (0,0.5) {$1$}; \node[left] at (0,1.5) {$\vdots$}; \node[left] at (0,2.5) {$b$}; \node[left] at (0,3.5) {$\vdots$}; \node[left] at (0,4.5) {$c$};
\node[above] at (0.5,5) {$p$}; \node[right] at (0.6,2.2) {\color{red}{$j$}}; \node[above] at (0.8,2.6) {\color{blue}{$i$}};
\node[left] at (0.5,2.6) {\color{blue}{...}};
\draw[blue] (0.4,2.6)--(1,2.6);
\draw[red] (0,2.4)--(0.6,2.4)--(0.6,4.4)--(1,4.4);
\end{tikzpicture}}
\end{tabular}
&
\begin{tabular}{c}
$0 < a = b < c = \infty$ \\
\ytableausetup{nobaseline}
\begin{ytableau}
\none &\none & b & \none[\infty] & \none & \none[T^{(j)}]  \\
\none & \none[p] & \none[\diag]  \\
\none & \none[\diag] \\
b & \none & \none[T^{(i)}] \\
\end{ytableau} $\leftrightarrow$
\resizebox{1cm}{!}{
\begin{tikzpicture}[baseline=(current bounding box.center)]
\draw (0,0) -- (0,5); \draw (1,0) -- (1,5); 
\draw (0,0) -- (1,0); \draw (0,1) -- (1,1); \draw (0,2) -- (1,2); \draw (0,3) -- (1,3); \draw (0,4) -- (1,4); \draw (0,5) -- (1,5);
\node[left] at (0,0.5) {$1$}; \node[left] at (0,1.5) {$\vdots$}; \node[left] at (0,2.5) {$b$}; \node[left] at (0,3.5) {$\vdots$}; \node[left] at (0,4.5) {$n$};
\node[above] at (0.5,5) {$p$}; \node[right] at (0.6,2.2) {\color{red}{$j$}}; \node[above] at (0.8,2.6) {\color{blue}{$i$}};
\node[left] at (0.5,2.6) {\color{blue}{...}};
\draw[blue] (0.4,2.6)--(1,2.6);
\draw[red] (0,2.4)--(0.6,2.4)--(0.6,5);
\end{tikzpicture}}
\end{tabular}
&
\begin{tabular}{c}
$0 < a = b = c < \infty$ \\
\ytableausetup{nobaseline}
\begin{ytableau}
\none &\none & b & b & \none & \none[T^{(j)}]  \\
\none & \none[p] & \none[\diag]  \\
\none & \none[\diag] \\
b & \none & \none[T^{(i)}] \\
\end{ytableau} $\leftrightarrow$
\resizebox{1cm}{!}{
\begin{tikzpicture}[baseline=(current bounding box.center)]
\draw (0,0) -- (0,5); \draw (1,0) -- (1,5); 
\draw (0,0) -- (1,0); \draw (0,1) -- (1,1); \draw (0,2) -- (1,2); \draw (0,3) -- (1,3); \draw (0,4) -- (1,4); \draw (0,5) -- (1,5);
\node[left] at (0,0.5) {$1$}; \node[left] at (0,1.5) {$\vdots$}; \node[left] at (0,2.5) {$b$}; \node[left] at (0,3.5) {$\vdots$}; \node[left] at (0,4.5) {$n$};
\node[above] at (0.5,5) {$p$}; \node[right] at (0.6,2.2) {\color{red}{$j$}}; \node[above] at (0.8,2.6) {\color{blue}{$i$}};
\node[left] at (0.5,2.6) {\color{blue}{...}};
\draw[blue] (0.4,2.6)--(1,2.6);
\draw[red] (0,2.4)--(1,2.4);
\end{tikzpicture}}
\end{tabular} 
\\ && 
\end{tabular}} \end{center}

\end{proof}

\section{Yang-Baxter Equation \label{sec:YBE}}

In this section we show that the vertex model we have defined is quantum integrable, that is, the weights of the model satisfy the Yang-Baxter equation.  This integrability can be used to show that the coinversion LLT polynomials are symmetric in the $x_i$ and is essential to the Cauchy identity given in Section \ref{sec:Cauchy}. 

To this end, define the \textbf{R-matrix} by 
\[
R_{y/x}(\I,\J;\K,\L) = 
\begin{tikzpicture}[baseline=(current bounding box.center)]
\draw (0,-1)--(2,1); \draw (0,1)--(2,-1);
\node[below,left] at (0,-1) {$\I$};
\node[above,left] at (0,1) {$\J$};
\node[above,right] at (2,1) {$\K$};
\node[below,right] at (2,-1) {$\L$};
\end{tikzpicture}
\]
where the weights are given by 
\begin{equation}\label{Rweight}
\begin{aligned}
& R_{y/x}(\I,\J;\K,\L) = \prod_{i=1}^k \left(\frac{y}{x t^{\delta_i} }\right)^{K_i} \left(1-\frac{y}{x t^{\delta_i}}\right)^{(J_i-I_i+L_i-K_i)/2}, \\
& \delta_i = \frac{\J_{[i+1,k]}-\I_{[i+1,k]}+\L_{[i+1,k]}-\K_{[i+1,k]}}{2}
\end{aligned}
\end{equation}
whenever $\I + \J = \K + \L$ and there is no $i \in [k]$ such that $I_i = K_i = 1$ and $J_i = L_i =0$. Otherwise $R_{y/x}(\I,\J;\K,\L) = 0$. In terms of paths this means, at a vertex, color $i$ must look like one of five possible configurations and contributes weight
\begin{center}
    \begin{tabular}{ccccc}
         Type 1 & Type 2 & Type 3 & Type 4 & Type 5 \\
         \begin{tikzpicture} \draw[ultra thick] (0,1)--(1,0);\draw[thin] (0,0)--(1,1); \end{tikzpicture} & \begin{tikzpicture} \draw[ultra thick] (0,1)--(0.5,0.5)--(1,1);\draw[thin] (0,0)--(0.5,0.5)--(1,0); \end{tikzpicture} & \begin{tikzpicture} \draw[thin] (0,1)--(0.5,0.5)--(1,1);\draw[ultra thick] (0,0)--(0.5,0.5)--(1,0); \end{tikzpicture} & \begin{tikzpicture} \draw[ultra thick] (0,1)--(1,0);\draw[ultra thick] (0,0)--(1,1); \end{tikzpicture} & \begin{tikzpicture} \draw[thin] (0,1)--(1,0);\draw[thin] (0,0)--(1,1); \end{tikzpicture}  \\
         $1-\frac{y}{xt^{\delta_i}}$ & $\frac{y}{xt^{\delta_i}}$ & $1$ & $\frac{y}{xt^{\delta_i}}$ & $1$
    \end{tabular}
\end{center}
where $\delta_i$ counts the number of colors greater than $i$ of Type 1.

\begin{thm}[Yang-Baxter Equation (YBE)] \label{prop:YBE}
The matrix entries of $L$ and $R$ satisfy the intertwining relation
\begin{equation}
\begin{aligned}
  \sum_{\K_1,\K_2,\K_3 \in \{0,1\}^k} L_y(\K_3,\K_2;\J_3,\J_2) L_x(\I_3,\K_1;\K_3,\J_1) R_{y/x}(\I_2,\I_1;\K_2,\K_1) \\ = \sum_{\L_1,\L_2,\L_3 \in \{0,1\}^k} R_{y/x}(\L_2,\L_1;\J_2,\J_1) L_x(\L_3,\I_1;\J_3,\L_1) L_y(\I_3,\I_2;\L_3,\L_2) 
\end{aligned}
\end{equation}
for any $\I_1,\I_2,\I_3,\J_1,\J_2,\J_3 \in \{0,1\}^k$.
\end{thm}
Graphically this means that $R$ and $L$ satisfy
\[
 \begin{tikzpicture}[baseline=(current bounding box.center)] 
 \draw (-1,0.5) -- (0,1.5); \draw (-1,1.5) -- (0,0.5); 
 \draw[step=1.0,black,thin] (0,0) grid (1,2); 
 \node[left] at (-1,1.5) {$\I_1$}; \node[left] at (-1,0.5) {$\I_2$}; \node[below] at (0.5,0) {$\I_3$};
 \node[right] at (1,1.5) {$\J_2$}; \node[right] at (1,0.5) {$\J_1$}; \node[above] at (0.5,2) {$\J_3$};
 \node at (0.5,0.5) {$x$}; \node at (0.5,1.5) {$y$};
 \end{tikzpicture} 
=
 \begin{tikzpicture}[baseline=(current bounding box.center)] 
 \draw (2,0.5) -- (1,1.5); \draw (2,1.5) -- (1,0.5); 
 \draw[step=1.0,black,thin] (0,0) grid (1,2); 
 \node[left] at (0,1.5) {$\I_1$}; \node[left] at (0,0.5) {$\I_2$}; \node[below] at (0.5,0) {$\I_3$};
 \node[right] at (2,0.5) {$\J_1$}; \node[right] at (2,1.5) {$\J_2$}; \node[above] at (0.5,2) {$\J_3$};
 \node at (0.5,0.5) {$y$}; \node at (0.5,1.5) {$x$};
 \end{tikzpicture}
\]
where the choice of $\I$'s and $\J$'s correspond to a choice of boundary condition and we sum over the weight of all possible configurations that satisfy the chosen boundary condition.

Denote by $L^{(k)}_x$ and $R^{(k)}_x$ the L-matrix and R-matrix for $k$ colors. We can recursively define the L- and R-matrices for $k+1$ colors by adding a $(k+1)^{st}$ color larger than all $k$ of the previous colors. We have 
\begin{equation} \label{LRrecursive}
    \begin{aligned}
    L^{(k+1)}_x((\I,I_{k+1}),(\J,J_{k+1});(\K,K_{k+1});(\L,L_{k+1})) = & \;  L^{(k)}_x(\I,\J;\K;\L) \otimes E(I_{k+1},J_{k+1};K_{k+1},L_{k+1}) \\
    & + L^{(k)}_{xt}(\I,\J;\K;\L) \otimes F_x(I_{k+1},J_{k+1};K_{k+1},L_{k+1}) \\
    R^{(k+1)}_{y/x}((\I,I_{k+1}),(\J,J_{k+1});(\K,K_{k+1});(\L,L_{k+1})) = & \;  R^{(k)}_{y/(x t)}(\I,\J;\K;\L) \otimes \tilde E_{y/x}(I_{k+1},J_{k+1};K_{k+1},L_{k+1}) \\
    &  R^{(k)}_{y/x}(\I,\J;\K;\L) \otimes \tilde F_{y/x}(I_{k+1},J_{k+1};K_{k+1},L_{k+1})
    \end{aligned}
\end{equation}
where $L^{(k)}$ and $R^{(k)}$ act on colors $1,\ldots,k$, and $E,F,\tilde E, \tilde F$ act only on the  $(k+1)^{st}$ color with weights given in the tables below.
\begin{center}
\begin{equation} \label{EFweights}
\begin{tabular}{|c|c|c|c|c|c|}
\hline
& & & & & \\
$E$ & \begin{tikzpicture}\draw[thin] (0,0) rectangle (1,1);\end{tikzpicture} & \begin{tikzpicture}\draw[thin] (0,0) rectangle (1,1);\draw[thick] (0.5,0)--(0.5,0.5)--(1,0.5);\end{tikzpicture} & \begin{tikzpicture}[thin] \draw (0,0) rectangle (1,1); \draw[thick] (0,0.5)--(1,0.5); \end{tikzpicture} & \begin{tikzpicture} \draw[thin] (0,0) rectangle (1,1); \draw[thick] (0.5,0)--(0.5,1); \end{tikzpicture} & \begin{tikzpicture} \draw[thin] (0,0) rectangle (1,1); \draw[thick] (0,0.5)--(0.5,0.5)--(0.5,1); \end{tikzpicture}
\\
& $1$ & $0$ & $0$ & $0$ & $0$ \\
\hline
& & & & & \\
$F_x$ & \begin{tikzpicture}\draw[thin] (0,0) rectangle (1,1);\end{tikzpicture} & \begin{tikzpicture}\draw[thin] (0,0) rectangle (1,1);\draw[thick] (0.5,0)--(0.5,0.5)--(1,0.5);\end{tikzpicture} & \begin{tikzpicture} \draw[thin] (0,0) rectangle (1,1); \draw[thick] (0,0.5)--(1,0.5); \end{tikzpicture} & \begin{tikzpicture} \draw[thin] (0,0) rectangle (1,1); \draw[thick] (0.5,0)--(0.5,1); \end{tikzpicture} & \begin{tikzpicture} \draw[thin] (0,0) rectangle (1,1); \draw[thick] (0,0.5)--(0.5,0.5)--(0.5,1); \end{tikzpicture}\\
& $0$ & $x$ & $x$ & $1$ & $1$ \\
\hline
\end{tabular}
\end{equation}
\end{center}

\begin{center}
\begin{tabular}{|c|c|c|c|c|c|}
\hline
& & & & & \\
$\tilde E_{y/x}$ & \begin{tikzpicture} \draw[ultra thick] (0,1)--(1,0);\draw[thin] (0,0)--(1,1); \end{tikzpicture} & \begin{tikzpicture} \draw[ultra thick] (0,1)--(0.5,0.5)--(1,1);\draw[thin] (0,0)--(0.5,0.5)--(1,0); \end{tikzpicture} & \begin{tikzpicture} \draw[thin] (0,1)--(0.5,0.5)--(1,1);\draw[ultra thick] (0,0)--(0.5,0.5)--(1,0); \end{tikzpicture} & \begin{tikzpicture} \draw[ultra thick] (0,1)--(1,0);\draw[ultra thick] (0,0)--(1,1); \end{tikzpicture} & \begin{tikzpicture} \draw[thin] (0,1)--(1,0);\draw[thin] (0,0)--(1,1); \end{tikzpicture} \\
& $1-y/x$ & $0$ & $0$ & $0$ & $0$ \\
\hline 
& & & & & \\
$\tilde F_{y/x}$ & \begin{tikzpicture} \draw[ultra thick] (0,1)--(1,0);\draw[thin] (0,0)--(1,1); \end{tikzpicture} & \begin{tikzpicture} \draw[ultra thick] (0,1)--(0.5,0.5)--(1,1);\draw[thin] (0,0)--(0.5,0.5)--(1,0); \end{tikzpicture} & \begin{tikzpicture} \draw[thin] (0,1)--(0.5,0.5)--(1,1);\draw[ultra thick] (0,0)--(0.5,0.5)--(1,0); \end{tikzpicture} & \begin{tikzpicture} \draw[ultra thick] (0,1)--(1,0);\draw[ultra thick] (0,0)--(1,1); \end{tikzpicture} & \begin{tikzpicture} \draw[thin] (0,1)--(1,0);\draw[thin] (0,0)--(1,1); \end{tikzpicture} \\
& $0$ & $y/x$ & $1$ & $y/x$ & $1$ \\
\hline
\end{tabular}
\end{center}

\noindent It is easy to check that the recursive formulae agree with the weights given in equations (\ref{FaceWeight}) and (\ref{Rweight}). These recursive formulae allow us to prove the YBE by induction on the number of colors. The details are given in Appendix \ref{YBE-proof}.

The YBE can be used to give an alternate proof that the coinversion LLT polynomials $\mathcal{L}_{\bm{\beta}/\bm{\gamma}}(X;t)$ are symmetric in the $x_i$ which was originally shown in \cite{LLT}, and later given a purely combinatorial proof in \cite{HHLsym}.
\begin{thm}
The coinversion LLT polynomials $\mathcal{L}_{\bm{\beta}/\bm{\gamma}}(X;t)$ are symmetric in the variables $X$.
\end{thm}
\begin{proof}
Consider two rows of our lattice model
\[
\begin{tikzpicture}
\draw (0,0) -- (0,2); \draw (1,0) -- (1,2); \draw (2,0) -- (2,2); \draw (3,0) -- (3,2); 
\draw (0,0) -- (3,0); \draw (0,1) -- (3,1); \draw (0,2) -- (3,2);
\node[left] at (0,0.5) {$\0$}; \node at (0.5,0.5) {$x_i$}; \node at (1.5,0.5) {$\ldots$};\node at (2.5,0.5) {$x_i$}; \node[right] at (3,0.5) {$\0$};
\node[left] at (0,1.5) {$\0$}; \node at (0.5,1.5) {$x_j$}; \node at (1.5,1.5) {$\ldots$}; \node at (2.5,1.5) {$x_j$}; \node[right] at (3,1.5) {$\0$};
\end{tikzpicture}
\]
\noindent We insert an R-matrix at the left end and note that the only non-zero entry of the R-matrix with $\K=\L=0$ has $\I=\J=0$ as well. Now repeatedly applying the YBE give
\[
\begin{tikzpicture}[baseline=(current bounding box.center)]
\draw (-1,1.5)--(0,0.5); \draw (-1,0.5)--(0,1.5);
\node[left] at (-1,0.5) {$\0$}; \node[left] at (-1,1.5) {$\0$};
\draw (0,0) -- (0,2); \draw (1,0) -- (1,2); \draw (2,0) -- (2,2); \draw (3,0) -- (3,2); 
\draw (0,0) -- (3,0); \draw (0,1) -- (3,1); \draw (0,2) -- (3,2);
 \node at (0.5,0.5) {$x_i$}; \node at (1.5,0.5) {$\ldots$};\node at (2.5,0.5) {$x_i$}; \node[right] at (3,0.5) {$\0$}; \node at (0.5,1.5) {$x_j$}; \node at (1.5,1.5) {$\ldots$}; \node at (2.5,1.5) {$x_j$}; \node[right] at (3,1.5) {$\0$};
\end{tikzpicture}
=
\begin{tikzpicture}[baseline=(current bounding box.center)]
\draw (3,1.5)--(4,0.5); \draw (3,0.5)--(4,1.5);
\node[right] at (4,0.5) {$\0$}; \node[right] at (4,1.5) {$\0$};
\draw (0,0) -- (0,2); \draw (1,0) -- (1,2); \draw (2,0) -- (2,2); \draw (3,0) -- (3,2); 
\draw (0,0) -- (3,0); \draw (0,1) -- (3,1); \draw (0,2) -- (3,2);
\node[left] at (0,0.5) {$\0$}; \node at (0.5,0.5) {$x_j$}; \node at (1.5,0.5) {$\ldots$};\node at (2.5,0.5) {$x_j$}; 
\node[left] at (0,1.5) {$\0$}; \node at (0.5,1.5) {$x_i$}; \node at (1.5,1.5) {$\ldots$}; \node at (2.5,1.5) {$x_i$}; 
\end{tikzpicture}
\]
where now we note that $\I=\J=0$ then we must have $\K=\L=0$ to get a non-zero element of the R-matrix. As $R_{x_j/x_i}(\0,\0;\0,\0) = 1$, we see that swapping $x_i$ and $x_j$ leaves the partition function unchanged. 
\end{proof}

\section{Single Rows \label{sec:one-row}}

In this section we consider the case when $\bm{\gamma}=\emptyset^k$ and each partition in $\bm{\beta}$ has a single part. In other words, $\bm{\beta}$ is a tuple of single rows. In this case, we can view the tuple $\bm{\beta}$ as a composition $\beta$, and will abuse notation by writing $\mc{L}_\beta(X;t)$ to mean the polynomial $\mc{L}_{\bm{\beta}}(X;t)$. Using the results of the previous section, we give another proof that LLT polynomials indexed by $\beta$ coincide with certain Hall-Littlewood polynomials. The exact statement is as follows.

Let $H_\mu(X; t)$ denote the transformed Hall-Littlewood polynomials, given by their Schur expansion
\begin{equation}
    H_\mu(X; t) = \sum_{\lambda \geq \mu} K_{\lambda, \mu}(t)s_\lambda(X)
\end{equation}
where $K_{\lambda, \mu}(t) \in \Z[t]$ denotes the Kostka-Foulkes polynomial. We will also make use of the modified Hall-Littlewood polyomials $\widetilde{H}_\mu(X; t)$ defined via
\begin{equation}
    \widetilde{H}_\mu(X; t) = t^{n(\mu)} H_\mu(X; t^{-1})
    \label{HL-HLt}
\end{equation}
where $n(\mu) = \sum_i (i-1) \mu_i$. Similarly, we define the modified Kostka-Foulkes polynomials $\widetilde{K}_{\lambda, \mu}(t) :=  t^{n(\mu)} K_{\lambda, \mu}(t^{-1})$.

The following is due to \cite{LLT}, albeit in a different form than stated below.
\begin{prop} \label{LLT-HL-inv}
Let $\mu$ be a partition, viewed also as a tuple of rows. Then, 
\begin{equation} \label{eq:LLT-HL-inv}
    \mathcal{G}_\mu(X; t) = \widetilde{H}_\mu(X; t).
\end{equation}
where we recall that $\mc{G}_\mu(X;t)$ denotes the inversion LLT polynomials.
\end{prop}
The polynomials $\widetilde{K}_{\lambda, \mu}(t)$ have many geometric interpretations. It is a well known fact \cite{lusztig1981green} that when $t=q$ is the cardinality of a finite field, then $GL_n(\mathbb{F}_q)$ acts on the set of $\mathbb{F}_q$-rational points of a flag variety, and $\widetilde{K}_{\lambda, \mu}(q)$ is equal to the value of an irreducible character $\chi_\lambda$ in this representation on a unipotent element $u$ with Jordan form specified by $\mu$. Proposition \ref{LLT-HL-inv} is proved in \cite{LLT} by showing that there is a cell decomposition of a similar flag variety whose cells are indexed by $k$ tuples of tableaux and whose dimensions are precisely the inversion statistic defined in Section \ref{sec:LLT}. 

As far as the authors are aware, the only known proofs rely on the geometry of an underlying flag variety. We present a new proof using coinversion LLT polynomials. Using (\ref{HL-HLt}) and (\ref{inv-coinv}), Equation (\ref{eq:LLT-HL-inv}) becomes
\begin{equation} \label{eq:LLT-HL-coinv}
    \mathcal{L}_\mu(X; t) = t^{m(\mu)-n(\mu)} H_\mu(X; t)
\end{equation}
where we recall that $m(\mu)$ is the total number of triples in the $\mu$. We derive the following explicit formula for $m$ which holds for all tuples of partitions $\bm{\beta}$, not just those consisting of single rows.
\begin{prop}
    Let $\bm{\beta}$ be a tuple of partitions. Then,
\begin{equation} \label{eq:m}
    m(\bm{\beta}) = \#\{ a < b, \; i, j \mid 0 \leq \beta_j^{(b)} - j + i < \beta_i^{(a)}\} + \sum_{\substack{a < b \\ i, j}} \max(\min(\beta^{(a)}_i - i, \beta^{(b)}_j - j) + \min(i,j), 0)
\end{equation}
\end{prop}
\begin{proof}
We count triples by their cell labelled $v$ in (\ref{triple}), as this cell is always in the shape $\bm{\beta}$. Fix a cell $v = (i,\ell) \in \beta^{(a)}$ . If there is a triple $(u,v,w)$, then $u, w$ must lie in some (or adjacent to some) $\beta^{(b)}$ for $b > a$. For each row $\beta^{(b)}_j$, let $u,w$ be the unique pair of cells in this row with $w$ on the same content line as $v$ and $u$ directly to the left of $w$. Then, $(u,v,w)$ form a triple if either (1) $u,w$ are both in $\beta^{(b)}_j$, (2) $u$ is the cell $(j,0)$ just before the beginning of the row, or (3), $u$ is the cell $(j, \beta^{(b)}_j)$ at the end of the row. In other words, $(u,v,w)$ is a triple exactly when $\beta^{(b)}_j$ has a cell of content $c(v)$ or $c(v)-1$. As the set of contents in the row $\beta^{(b)}_j$ is precisely the interval $[1-j, \beta^{(b)}_j-j]$, then 
\[ (u,v,w) \text{ is a triple } \iff 1-j \leq c(v) \leq \beta^{(b)}_j - j + 1 \iff i-j \leq \ell-1 \leq \beta^{(b)}_j + i - j  \]
As $\ell-1$ ranges over the interval $[0, \beta_i^{(a)}-1]$, after summing over $\ell, i,j$ and $a < b$, we find that the number of triples $(u,v,w)$ is
\begin{align*} m(\bm{\beta}) &= \sum_{\substack{a < b \\ i,j}} \# \left( [0,\beta^{(a)}_i-1] \cap [i-j, \beta^{(b)}_j + i - j] \right) \\
&= \sum_{\substack{a < b \\ i,j}} \# \left( [-i,\beta^{(a)}_i-i-1] \cap [-j, \beta^{(b)}_j - j] \right) \\
&= \sum_{\substack{a < b \\ i,j}} \max(\min(\beta^{(a)}_i-i-1, \beta^{(b)}_j - j) - \max(-i, -j) + 1, 0) \\
&= \sum_{\substack{a < b \\ i,j}} \max(\min(\beta^{(a)}_i-i, \beta^{(b)}_j - j) + \min(i,j), 0) + 
\begin{cases}
1 &: -\min(i,j) \leq \beta^{(b)}_j - j \leq \beta^{(a)}_i-i-1 \\
0 &: \text{else}
\end{cases} \\
&= \sum_{\substack{a < b \\ i,j}} \max(\min(\beta^{(a)}_i-i, \beta^{(b)}_j - j) + \min(i,j), 0) + 
\begin{cases}
1 &: i-\min(i,j) \leq \beta^{(b)}_j - j +i < \beta^{(a)}_i \\
0 &: \text{else}
\end{cases}
\end{align*}
The condition $i-\min(i,j) \leq \beta^{(b)}_j - j + i$ is seen to be equivalent to $0 \leq \beta^{(b)}_j - j + i$ in either case $i \leq j$ or $j \leq i$.
\end{proof}
Given a composition $\beta$, define $\inv(\beta) = \# \{i < j \mid \beta_i > \beta_j \}$.
\begin{cor} \label{m-cor}
	Let $\mu$ be a partition and let $\beta$ be any rearrangement of its parts. Then,
	\[ m(\beta) = n(\mu) + \inv(\beta) \]
	where $n(\mu) = \sum_i (i-1) \mu_i$.
\end{cor}
\begin{proof}
As $\bm{\beta}$ consists of single rows, the only non-zero terms in (\ref{eq:m}) are when $i=j=1$.
Thus,
\begin{align*}
    m(\beta) &= \#\{ a < b \mid 0 \leq \beta^{(b)} < \beta^{(a)}\} + \sum_{a < b} \max(\min(\beta^{(a)} - 1, \beta^{(b)} - 1) + 1, 0) \\
    &= \inv(\beta) + \sum_{a < b} \min(\beta^{(a)}, \beta^{(b)})
\end{align*}
The result follows from the identity $n(\mu) = \sum_{a < b} \min(\beta^{(a)}, \beta^{(b)})$.
\end{proof}
Hence with Corollary \ref{m-cor}, (\ref{eq:LLT-HL-coinv}) becomes
\begin{equation} \label{eq:LLT-HL-coinv-2}
    \mathcal{L}_\mu(X; t) = t^{\inv(\mu)} H_\mu(X; t)
\end{equation}
We establish the following statement, giving another proof of (\ref{eq:LLT-HL-coinv-2}).
\begin{prop} \label{LLT-HL-coinv}
	Let $\mu$ be a partition and let $\beta$ be any rearrangement of the parts of $\mu$. 
	Then,
	\[ \mc{L}_{\beta}(X;t) = t^{\inv(\beta)} H_\mu(X; t) \]
	In particular, if $\mu^{\rev} = (\mu_n, \ldots, \mu_1)$ is $\mu$ in reverse order, then
	\[ \mc{L}_{\mu^{\rev}}(X;t) = H_\mu(X; t) \]
\end{prop}
\begin{proof}
Using Proposition \ref{swap-single-rows} below, it suffices to prove the particular case with $\mu^{\rev}$. Viewing $\mu^{\rev}$ as a composition rather than as a tuple of rows, we recast the definition of $\mc{L}_{\mu^{\rev}}$ as a sum over fillings of the diagram $D(\mu^{\rev})$:
\begin{equation} \label{LLT-single-rows}
\mc{L}_{\mu^{\rev}}(X; t) = \sum_{\substack{\sigma : \mu^{\rev} \to \Z_+ \\ \text{rows} \leq }} t^{\coinv(\sigma)} x^\sigma
\end{equation}
where the sum is over fillings of $\mu^{\rev}$ with weakly increasing rows, and a coinversion triple of a filling $\sigma$ is of the form
\begin{equation} \label{coinv-single-rows}
\begin{ytableau}
a & c \\
\none & \none \\
\none & b
\end{ytableau}
\end{equation}
with $a \leq b \leq c$, and we set $a=0$ if it is not in $\mu^{\rev}$. 

We now recall the combinatorial definition of $H_\mu(X;t)$, given in \cite{HHLsym}, by exchanging $q$ and $t$ in the combinatorial formula for Macdonald polynomials and setting $q=0$:
\begin{equation} \label{H-comb}
H_\mu(X; t) = \sum_{\substack{\sigma : \mu' \to \Z_+ \\ \text{cols} \geq }} t^{\coinv'(\sigma)} x^\sigma
\end{equation}
Here, the sum is over fillings $\sigma$ on the conjugate Young diagram $D(\mu')$ with weakly decreasing columns, and $\coinv'(\sigma)$ counts triples of the form
\begin{equation} \label{coinv-H}
\begin{ytableau}
x & \none & \none & z \\
y
\end{ytableau}
\end{equation}
with $x \leq z \leq y$, and we set $y=0$ if it is not in $\mu'$. 

We produce a weight preserving bijection $\Phi$ from terms in (\ref{LLT-single-rows}) to terms in (\ref{H-comb}). Define $\Phi$ to be the map from fillings on the shape $\mu^{\rev}$ to fillings on the shape $\mu'$ which rotates 90 degrees counterclockwise and swaps entries $i$ and $n-i+1$. For example,
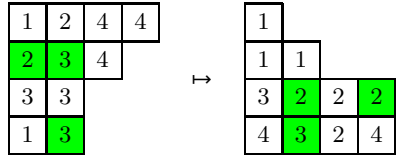
\[ \begin{ytableau}
1 & 2 &  4 & 4 \\
*(green) 2 & *(green) 3 &  4 \\
3 &  3 \\
1 & *(green) 3
\end{ytableau}
\quad \mapsto \quad
\begin{ytableau}
1 \\
1 & 1 \\
3 & *(green) 2 & 2 & *(green) 2 \\
4 & *(green) 3 & 2 & 4
\end{ytableau}
\]
It is clear that $\Phi$ is a bijection and after setting $a=y,$ $b=z$ and $c=x$, maps coinversion triples as in (\ref{coinv-single-rows}) precisely to triples as in (\ref{coinv-H}).
\end{proof}

\begin{prop} \label{swap-single-rows}
	Let $\mu$ be a partition and let $\beta$ be any rearrangement of the parts of $\mu$.  Then,
	\[ \mc{L}_\mu(X; t) = t^{\inv(\mu) - \inv(\beta)}  \mc{L}_\beta(X; t) \]
\end{prop}
\begin{proof}
	It suffices to prove this in the case when $\beta$ is obtained from $\mu$ by a single swap of consecutive rows, and in fact we only need to consider when $\mu = (\mu_1, \mu_2)$ and $\beta = (\mu_2, \mu_1)$. We order the colors blue $<$ red.
	
	As the case $\mu_1=\mu_2$ is trivial, we assume $\mu_1 > \mu_2$. In this case, the blue path will end to the right of the red path. Since the paths start at the same face, there exists a last face of the form 
	\begin{tikzpicture}[baseline=(current bounding box.center)]
	\draw (0,0)--(1,0)--(1,1)--(0,1)--(0,0); \draw[blue] (0.4,0.6)--(1,0.6); \draw[red] (0.6,0.4)--(0.6,1);
	\end{tikzpicture}. Taking a path configuration with upper boundary $\mu=(\mu_1,\mu_2)$ and and swapping the color of the paths after this face gives a bijection with configurations with upper boundary $\beta=(\mu_2,\mu_1)$. For example,
	
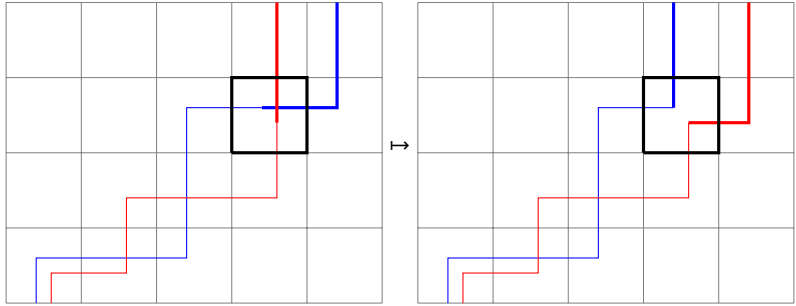
\[
	\begin{tikzpicture}[baseline=(current bounding box.center)]
	\draw[help lines] (0,0) grid (5,4);
	
	\draw[blue] (0.4,0)--(0.4,0.6)--(1.4,0.6)--(2.4,0.6)--(2.4,2.6)--(3.4,2.6);
	\draw[blue, very thick] (3.4,2.6)--(4.4,2.6)--(4.4,4);
	\draw[red] (0.6,0)--(0.6,0.4)--(1.6,0.4)--(1.6,1.4)--(3.6,1.4)--(3.6,2.4);
	\draw[red, very thick] (3.6,2.4)--(3.6,4);
	
	\draw[black, very thick] (3,2)--(4,2)--(4,3)--(3,3)--(3,2);

	\end{tikzpicture}
	\mapsto
	\begin{tikzpicture}[baseline=(current bounding box.center)]
	\draw[help lines] (0,0) grid (5,4);
	
	\draw[blue] (0.4,0)--(0.4,0.6)--(1.4,0.6)--(2.4,0.6)--(2.4,2.6)--(3.4,2.6);
	\draw[blue, very thick] (3.4,2.6)--(3.4,4);
	\draw[red] (0.6,0)--(0.6,0.4)--(1.6,0.4)--(1.6,1.4)--(3.6,1.4)--(3.6,2.4);
	\draw[red, very thick] (3.6,2.4)--(4.4,2.4)--(4.4,4);
	
	\draw[black, very thick] (3,2)--(4,2)--(4,3)--(3,3)--(3,2);
	
	\end{tikzpicture}
	\]
	It's clear that the number of coinversions decreases by one after swapping, so that
	\[ \mathcal{L}_\mu (X;t) = t \; \mathcal{L}_\beta (X;t). \]
\end{proof}

\section{A Cauchy Identity}\label{sec:Cauchy}

In this section, we give another proof of a Cauchy identity for LLT polynomials. A Cauchy identity was given in \cite{lam2005ribbon} for the original spin-generating LLT polynomials, whereas our Cauchy identity is for coinversion LLT polynomials; however the reader can readily use (\ref{spin-inv}) to rederive the identity in \cite{lam2005ribbon}.

\begin{thm}[Cauchy Identity]\label{thm:CauchyIdentity}
Fix $n, k \geq 1$. Then
\begin{equation}
    \sum_{\bm{\lambda}} t^{d(\bm{\lambda})} \mathcal{L}_{\bm{\lambda}}(X_n;t)\mathcal{L}_{\bm{\lambda}}(Y_n;t) = \prod_{i,j=1}^n\prod_{m=0}^{k-1} \left(1-x_iy_j t^m \right)^{-1}
    \label{CauchyIdentityEq}
\end{equation}
where the sum is over $k$-tuples $\bm{\lambda}$ of partitions with $n$ non-negative parts, $d(\bm{\lambda})$ is given in Lemma \ref{lem:BoxComp}.
\end{thm}

The following method of using vertex models to prove Cauchy-type identities was given in \cite{WHEELER2016543}. We start by introducing the $L^*$-matrix, represented by a gray box,
\begin{align*}
L^*_{\bar x}(\I,\J;\K,\L)  =
\resizebox{2cm}{!}{
\begin{tikzpicture}[baseline=(current bounding box.center)]  \draw[fill=lightgray] (0,0) rectangle (2,2); 
\node at (1,1) {{\Large $\bar{x}$}}; 
\node[left] at (0,1) {$\J$}; \node[below] at (1,0) {$\I$};
\node[right] at (2,1) {$\L$}; \node[above] at (1,2) {$\K$};
\end{tikzpicture}
}
\end{align*}
\noindent where $\bar{x} = \frac{1}{xt^{k-1}}$. The weights are defined by the equation
\begin{equation}
L^*_{\bar x}(\I,\J;\K,\L)  = x^k t^{\binom{k}{2}}L_{\bar{x}}(\I,\J;\K,\L).
\end{equation}

\begin{prop}
The matrix entries of $R$, $L$, and $L^*$ satisfy the intertwining equation
\begin{equation}\label{intertwining2}
\begin{aligned}
      \sum_{\K_1,\K_2,\K_3 \in \{0,1\}^k} L_y(\K_3,\K_2;\J_3,\J_2) L^*_{\bar x}(\I_3,\K_1;\K_3,\J_1) R_{y /\bar{x}}(\I_2,\I_1;\K_2,\K_1) \\ = \sum_{\L_1,\L_2,\L_3 \in \{0,1\}^k} R_{y /\bar{x}}(\L_2,\L_1;\J_2,\J_1) L^*_{\bar x}(\L_3,\I_1;\J_3,\L_1) L_y(\I_3,\I_2;\L_3,\L_2) 
\end{aligned}
\end{equation}
for any $\I_1,\I_2,\I_3,\J_1,\J_2,\J_3 \in \{0,1\}^k$. In other words, graphically, $R$, $L$, and $L^*$ satisfy
\[
  \begin{tikzpicture}[baseline=(current bounding box.center)] \draw (-1,0.5) -- (0,1.5); \draw (-1,1.5) -- (0,0.5); \draw[fill=lightgray] (0,0) rectangle (1,1); \draw[step=1.0,black,thin] (0,0) grid (1,2); \node at (0.5,1.5) {$y$}; \node at (0.5,0.5) {$\bar{x}$}; \end{tikzpicture} 
=
 \begin{tikzpicture}[baseline=(current bounding box.center)] \draw (2,0.5) -- (1,1.5); \draw (2,1.5) -- (1,0.5); \draw[fill=lightgray] (0,1) rectangle (1,2); \draw[step=1.0,black,thin] (0,0) grid (1,2); \node at (0.5,1.5) {$\bar{x}$}; \node at (0.5,0.5) {$y$}; \end{tikzpicture}  
\]
\end{prop} 

\noindent 

\begin{proof}
After plugging in the definition of $L^*$ this follows from the YBE, Theorem \ref{prop:YBE}.
\end{proof}

\subsection{Row-to-Row Transfer Matrices}

We introduce two types of semi-infinite \textbf{row-to-row transfer matrices}.  The first is obtained by concatenating infinitely many $L$ matrices, and the second is obtained by concatenating infinitely many $L^*$ matrices. Graphically, we represent these matrices as
\begin{align*}
T(x) = \begin{tikzpicture}[baseline=(current bounding box.center)] \draw[step=1.0,black,thin] (0,0) grid (5,1); \node[left] at (0,0.5) {$x$}; \node[right] at (5,0.5) {$\ldots$}; \end{tikzpicture}  \\
T^*(x) = \begin{tikzpicture}[baseline=(current bounding box.center)] \draw[fill=lightgray] (0,0) rectangle (5,1); \draw[step=1.0,black,thin] (0,0) grid (5,1); \node[left] at (0,0.5) {$\bar{x}$}; \node[right] at (5,0.5) {$\ldots$}; \end{tikzpicture}
\end{align*}

\noindent For $T(x)$ we write the parameter $x$ to the left of the row to indicate that every face has weights given by $L_x$, and similarly for $T^*(x)$. An entry of the matrix is given by fixing the incoming (left and bottom) and outgoing (top and right) paths on the boundary and summing over the weight of all configurations respecting these boundary conditions.

Note that the matrix entries are well-defined only if we assume $|x|<1$, which means that any matrix entry with unbounded degree in $x$ is equal to 0. We are interested in certain submatrices given by
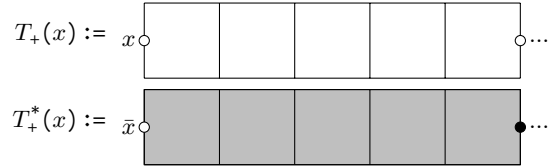
\begin{align*}
T_+(x) := \begin{tikzpicture}[baseline=(current bounding box.center)] \draw[step=1.0,black,thin] (0,0) grid (5,1); \node[left] at (0,0.5) {$x$}; \node[right] at (5,0.5) {$...$}; \draw[black,fill=white] (5,0.5) circle (.5ex); \draw[black,fill=white] (0,0.5) circle (.5ex); \end{tikzpicture}
\\
 T^*_+(x) := \begin{tikzpicture}[baseline=(current bounding box.center)] \draw[fill=lightgray] (0,0) rectangle (5,1); \draw[step=1.0,black,thin] (0,0) grid (5,1); \node[left] at (0,0.5) {$\bar{x}$}; \node[right] at (5,0.5) {$...$}; \draw[black,fill=black] (5,0.5) circle (.5ex); \draw[black,fill=white] (0,0.5) circle (.5ex); \end{tikzpicture}
\end{align*}
\noindent where to simplify notation we use a black circle to indicate the vector $\1$ and a white circle to indicate the vector $\0$.
Graphically, fixing how the paths enter and exit the row from the top and bottom, the corresponding entry $T_+(x)$ is the sum over the weight all possible infinite row configurations of white boxes such that the leftmost edge is unoccupied and only empty boxes appear sufficiently far to the right, respecting the top and bottom boundary conditions. The corresponding entry of $T^*_+(x)$ is the sum over the weight all possible infinite row configurations of gray boxes such that the leftmost edge is unoccupied and only boxes with all paths horizontal appear sufficiently far to the right respecting the top and bottom boundary conditions.

Repeated application of the Yang-Baxter equation (\ref{intertwining2}) yields
\begin{prop}\label{prop:Tcommute}
The matrices $T_+(y)$ and $T^*_+(x)$ satisfy
\begin{equation}
\resizebox{6cm}{!}{
    \begin{tikzpicture}[baseline=(current bounding box.center)]
    \draw[fill=lightgray] (0,0) rectangle (5,1);
    \draw[step=1.0,black,thin] (0,0) grid (5,2);
    \draw (0,0.5)--(-1,1.5); \draw (0,1.5)--(-1,0.5);
    \draw[black,fill=black] (5,0.5) circle (.5ex);
    \draw[black,fill=white] (5,1.5) circle (.5ex);
    \draw[black,fill=white] (-1,0.5) circle (.5ex);
    \draw[black,fill=white] (-1,1.5) circle (.5ex);
    \node[left] at (-1,0.5) {$y$};
    \node[left] at (-1,1.5) {$\bar x$};
    \node[above] at (4.5,2) {$\ldots$};
    \end{tikzpicture}
} = 
\resizebox{6cm}{!}{
    \begin{tikzpicture}[baseline=(current bounding box.center)]
    \draw[fill=lightgray] (0,1) rectangle (5,2);
    \draw[step=1.0,black,thin] (0,0) grid (5,2);
    \draw (5,0.5)--(6,1.5); \draw (5,1.5)--(6,0.5);
    \draw[black,fill=black] (6,0.5) circle (.5ex);
    \draw[black,fill=white] (6,1.5) circle (.5ex);
    \draw[black,fill=white] (0,0.5) circle (.5ex);
    \draw[black,fill=white] (0,1.5) circle (.5ex);
    \node[left] at (0,0.5) {$\bar x$};
    \node[left] at (0,1.5) {$y$};
    \node[above] at (4.5,2) {$\ldots$};
    \end{tikzpicture}
}
\end{equation}
\end{prop}

\subsection{Proof of Theorem~\ref{thm:CauchyIdentity}}

The proof comes in several steps. We begin with the equation
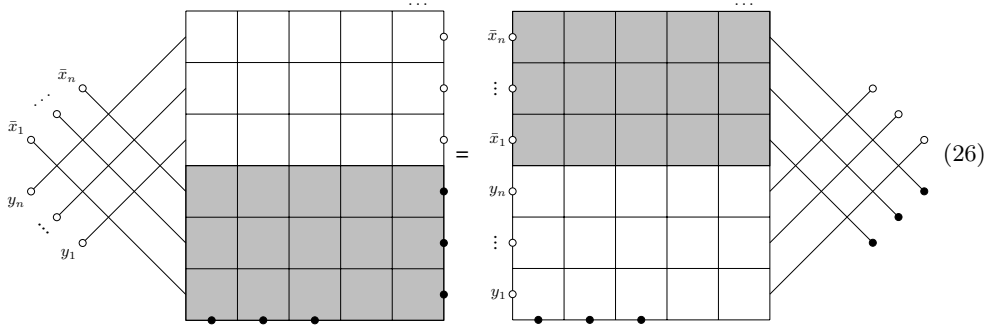
\begin{equation}
\label{CauchyEqn}
\resizebox{6cm}{!}{ \begin{tikzpicture}[baseline=(current bounding box.center)]
\draw[fill=lightgray] (0,0) rectangle (5,3);
\draw[step=1.0,black,thin] (0,0) grid (5,6);
\draw (0,0.5) -- (-3,3.5); \draw (0,1.5) -- (-2.5,4); \draw (0,2.5) -- (-2,4.5); 
\draw (0,3.5) -- (-2,1.5); \draw (0,4.5) -- (-2.5,2); \draw (0,5.5) -- (-3,2.5); 
\node[above] at (4.5,6) {$\ldots$};
\node[below left] at (-2,1.5) {$y_1$}; \node[below left] at (-2.5,2) {$\ddots$}; \node[below left] at (-3,2.5) {$y_n$};
\node[above left] at (-3,3.5) {$\bar{x}_1$}; \node[above left] at (-2.5,4) {$\iddots$}; \node[above left] at (-2,4.5) {$\bar{x}_n$};
\draw[black,fill=black] (5,0.5) circle (.5ex); \draw[black,fill=black] (5,1.5) circle (.5ex); \draw[black,fill=black] (5,2.5) circle (.5ex); \draw[black,fill=white] (5,3.5) circle (.5ex); \draw[black,fill=white] (5,4.5) circle (.5ex); \draw[black,fill=white] (5,5.5) circle (.5ex);
\draw[black,fill=black] (0.5,0) circle (.5ex); \draw[black,fill=black] (1.5,0) circle (.5ex); \draw[black,fill=black] (2.5,0) circle (.5ex);
\draw[black,fill=white] (-2,1.5) circle (.5ex); \draw[black,fill=white] (-2.5,2) circle (.5ex); \draw[black,fill=white] (-3,2.5) circle (.5ex);
\draw[black,fill=white] (-2,4.5) circle (.5ex); \draw[black,fill=white] (-2.5,4) circle (.5ex); \draw[black,fill=white] (-3,3.5) circle (.5ex);
\end{tikzpicture}}
=
\resizebox{6cm}{!}{ \begin{tikzpicture}[baseline=(current bounding box.center)]
\draw[fill=lightgray] (0,3) rectangle (5,6);
\draw[step=1.0,black,thin] (0,0) grid (5,6);
\draw (5,0.5) -- (8,3.5); \draw (5,1.5) -- (7.5,4); \draw (5,2.5) -- (7,4.5); 
\draw (5,3.5) -- (7,1.5); \draw (5,4.5) -- (7.5,2); \draw (5,5.5) -- (8,2.5); 
\node[above] at (4.5,6) {$\ldots$};
\node[left] at (0,0.5) {$y_1$}; \node[left] at (0,1.5) {$\vdots \enspace$}; \node[left] at (0,2.5) {$y_n$};
\node[left] at (0,3.5) {$\bar{x}_1$}; \node[left] at (0,4.5) {$\vdots \enspace$}; \node[left] at (0,5.5) {$\bar{x}_n$};
\draw[black,fill=black] (0.5,0) circle (.5ex); \draw[black,fill=black] (1.5,0) circle (.5ex); \draw[black,fill=black] (2.5,0) circle (.5ex);
\draw[black,fill=black] (7,1.5) circle (.5ex); \draw[black,fill=black] (7.5,2) circle (.5ex); \draw[black,fill=black] (8,2.5) circle (.5ex);
\draw[black,fill=white] (7,4.5) circle (.5ex); \draw[black,fill=white] (7.5,4) circle (.5ex); \draw[black,fill=white] (8,3.5) circle (.5ex);
\draw[black,fill=white] (0,0.5) circle (.5ex); \draw[black,fill=white] (0,1.5) circle (.5ex); \draw[black,fill=white] (0,2.5) circle (.5ex); \draw[black,fill=white] (0,3.5) circle (.5ex); \draw[black,fill=white] (0,4.5) circle (.5ex); \draw[black,fill=white] (0,5.5) circle (.5ex);
\end{tikzpicture}}
\end{equation} \\
\noindent which follows from repeated applications of Proposition \ref{prop:Tcommute}. Here the gray rows are generated by $T_+^*$ and the white rows are generated by $T_+$.

\indent We first simplify the left-hand side of (\ref{CauchyEqn}).  As no paths enter from the left, we can factorize the left-hand side as follows:
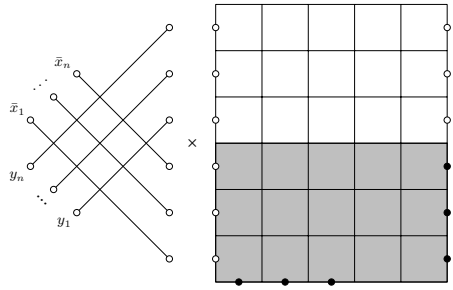
\begin{align*}
\resizebox{6cm}{!}{ \begin{tikzpicture}
\draw[fill=lightgray] (0+1,0) rectangle (5+1,3);
\draw[step=1.0,black,thin] (0+1,0) grid (5+1,6);
\draw (0,0.5) -- (-3,3.5); \draw (0,1.5) -- (-2.5,4); \draw (0,2.5) -- (-2,4.5); 
\draw (0,3.5) -- (-2,1.5); \draw (0,4.5) -- (-2.5,2); \draw (0,5.5) -- (-3,2.5); 
\node[above] at (4.5+1,6) {$\ldots$};
\node[below left] at (-2,1.5) {$y_1$}; \node[below left] at (-2.5,2) {$\ddots$}; \node[below left] at (-3,2.5) {$y_n$};
\node[above left] at (-3,3.5) {$\bar{x}_1$}; \node[above left] at (-2.5,4) {$\iddots$}; \node[above left] at (-2,4.5) {$\bar{x}_n$};
\draw[black,fill=black] (5+1,0.5) circle (.5ex); \draw[black,fill=black] (5+1,1.5) circle (.5ex); \draw[black,fill=black] (5+1,2.5) circle (.5ex);
\draw[black,fill=black] (0.5+1,0) circle (.5ex); \draw[black,fill=black] (1.5+1,0) circle (.5ex); \draw[black,fill=black] (2.5+1,0) circle (.5ex);
\draw[black,fill=white] (0,0.5) circle (.5ex); \draw[black,fill=white] (0,1.5) circle (.5ex); \draw[black,fill=white] (0,2.5) circle (.5ex); \draw[black,fill=white] (0,3.5) circle (.5ex); \draw[black,fill=white] (0,4.5) circle (.5ex); \draw[black,fill=white] (0,5.5) circle (.5ex);
\draw[black,fill=white] (0+1,0.5) circle (.5ex); \draw[black,fill=white] (0+1,1.5) circle (.5ex); \draw[black,fill=white] (0+1,2.5) circle (.5ex); \draw[black,fill=white] (0+1,3.5) circle (.5ex); \draw[black,fill=white] (0+1,4.5) circle (.5ex); \draw[black,fill=white] (0+1,5.5) circle (.5ex);
\draw[black,fill=white] (5+1,3.5) circle (.5ex); \draw[black,fill=white] (5+1,4.5) circle (.5ex); \draw[black,fill=white] (5+1,5.5) circle (.5ex);
\draw[black,fill=white] (-2,1.5) circle (.5ex); \draw[black,fill=white] (-2.5,2) circle (.5ex); \draw[black,fill=white] (-3,2.5) circle (.5ex);
\draw[black,fill=white] (-2,4.5) circle (.5ex); \draw[black,fill=white] (-2.5,4) circle (.5ex); \draw[black,fill=white] (-3,3.5) circle (.5ex);
\node at (0.5,3) {$\times$};
\end{tikzpicture}}
\end{align*}
\noindent The first factor equals 1, since it is void of paths.  To evaluate the second factor, we observe that the only non-zero contribution occurs for the following  configuration:
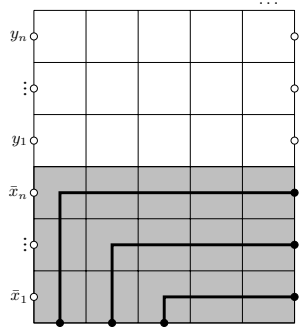
\begin{align*}
\resizebox{4cm}{!}{ \begin{tikzpicture}
\draw[fill=lightgray] (0,0) rectangle (5,3);
\draw[step=1.0,black,thin] (0,0) grid (5,6);
\node[left] at (0,0.5) {$\bar{x}_1$}; \node[left] at (0,1.5) {$\vdots$}; \node[left] at (0,2.5) {$\bar{x}_n$};
\node[left] at (0,3.5) {$y_1$}; \node[left] at (0,4.5) {$\vdots$}; \node[left] at (0,5.5) {$y_n$};
\node[above] at (4.5,6) {$\ldots$};
\draw[black,fill=black] (5,0.5) circle (.5ex); \draw[black,fill=black] (5,1.5) circle (.5ex); \draw[black,fill=black] (5,2.5) circle (.5ex);
\draw[black,fill=black] (0.5,0) circle (.5ex); \draw[black,fill=black] (1.5,0) circle (.5ex); \draw[black,fill=black] (2.5,0) circle (.5ex);
\draw[black, ultra thick] (2.5,0)--(2.5,0.5)--(5,0.5); \draw[black, ultra thick] (1.5,0)--(1.5,1.5)--(5,1.5); \draw[black, ultra thick] (0.5,0)--(0.5,2.5)--(5,2.5);
\draw[black,fill=white] (0,0.5) circle (.5ex); \draw[black,fill=white] (0,1.5) circle (.5ex); \draw[black,fill=white] (0,2.5) circle (.5ex); \draw[black,fill=white] (0,3.5) circle (.5ex); \draw[black,fill=white] (0,4.5) circle (.5ex); \draw[black,fill=white] (0,5.5) circle (.5ex);
 \draw[black,fill=white] (5,3.5) circle (.5ex); \draw[black,fill=white] (5,4.5) circle (.5ex); \draw[black,fill=white] (5,5.5) circle (.5ex);
\end{tikzpicture}}
\end{align*}
\noindent in which the bold paths indicate that paths of every color follow the trajectory. This has weight
\begin{equation}
\left(x_1^k t^{\binom{k}{2}}\right)^{n-1}\left(x_2^k t^{\binom{k}{2}}\right)^{n-2}\cdots \left(x_{n-1}^k t^{\binom{k}{2}}\right)^{1} = \left(x^\rho\right)^k t^{\binom{n}{2}\binom{k}{2}}.
\label{eq:CauchyLHS}
\end{equation}

\noindent where we recall $\rho = (n-1, \ldots, 0)$, and $x^\rho$ denotes the monomial $x_1^{n-1} \cdots x_n^0$.

\indent We now simplify the right-hand side of (\ref{CauchyEqn}).  The right edge of the lattice is situated at infinity, hence a non-zero contribution is possible only if the lowest $n$ edges are unoccupied and the highest $n$ edges are occupied.  Therefore we can factorize the right-hand side as follows:
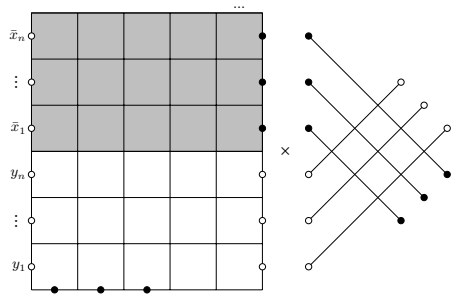
\begin{align*}
\resizebox{6cm}{!}{ \begin{tikzpicture}
\draw[fill=lightgray] (0,3) rectangle (5,6);
\draw[step=1.0,black,thin] (0,0) grid (5,6);
\draw (5+1,0.5) -- (8+1,3.5); \draw (5+1,1.5) -- (7.5+1,4); \draw (5+1,2.5) -- (7+1,4.5); 
\draw (5+1,3.5) -- (7+1,1.5); \draw (5+1,4.5) -- (7.5+1,2); \draw (5+1,5.5) -- (8+1,2.5); 
\node[above] at (4.5,6) {$...$};
\node[left] at (0,0.5) {$y_1$}; \node[left] at (0,1.5) {$\vdots \enspace$}; \node[left] at (0,2.5) {$y_n$};
\node[left] at (0,3.5) {$\bar{x}_1$}; \node[left] at (0,4.5) {$\vdots \enspace$}; \node[left] at (0,5.5) {$\bar{x}_n$};
\draw[black,fill=black] (0.5,0) circle (.5ex); \draw[black,fill=black] (1.5,0) circle (.5ex); \draw[black,fill=black] (2.5,0) circle (.5ex);
\draw[black,fill=black] (7+1,1.5) circle (.5ex); \draw[black,fill=black] (7.5+1,2) circle (.5ex); \draw[black,fill=black] (8+1,2.5) circle (.5ex);
\draw[black,fill=black] (5,3.5) circle (.5ex); \draw[black,fill=black] (5,4.5) circle (.5ex); \draw[black,fill=black] (5,5.5) circle (.5ex);
\draw[black,fill=black] (5+1,3.5) circle (.5ex); \draw[black,fill=black] (5+1,4.5) circle (.5ex); \draw[black,fill=black] (5+1,5.5) circle (.5ex);
\draw[black,fill=white] (0,0.5) circle (.5ex); \draw[black,fill=white] (0,1.5) circle (.5ex); \draw[black,fill=white] (0,2.5) circle (.5ex); \draw[black,fill=white] (0,3.5) circle (.5ex); \draw[black,fill=white] (0,4.5) circle (.5ex); \draw[black,fill=white] (0,5.5) circle (.5ex);
\draw[black,fill=white] (5,0.5) circle (.5ex); \draw[black,fill=white] (5,1.5) circle (.5ex); \draw[black,fill=white] (5,2.5) circle (.5ex);
\draw[black,fill=white] (5+1,0.5) circle (.5ex); \draw[black,fill=white] (5+1,1.5) circle (.5ex); \draw[black,fill=white] (5+1,2.5) circle (.5ex);
\draw[black,fill=white] (8+1,3.5) circle (.5ex); \draw[black,fill=white] (7.5+1,4) circle (.5ex); \draw[black,fill=white] (7+1,4.5) circle (.5ex);
\node at (5.5,3) {$\times$};
\end{tikzpicture}}
\end{align*}
\noindent The only non-zero contribution to the second factor occurs when all paths travel in a straight line from the North-West to the South-East. This configuration has weight
\begin{equation}
w\left(
\resizebox{3cm}{!}{ \begin{tikzpicture}[baseline=(current bounding box.center)]
\draw (5,0.5) -- (8,3.5); \draw (5,1.5) -- (7.5,4); \draw (5,2.5) -- (7,4.5); 
\draw (5,3.5) -- (7,1.5); \draw (5,4.5) -- (7.5,2); \draw (5,5.5) -- (8,2.5); 
\draw[black, ultra thick] (5,3.5) -- (7,1.5); \draw[black, ultra thick] (5,4.5) -- (7.5,2); \draw[black, ultra thick] (5,5.5) -- (8,2.5); 
\node[left] at (5,0.5) {$y_1$}; \node[left] at (5,1.5) {$\vdots \enspace$}; \node[left] at (5,2.5) {$y_n$};
\node[left] at (5,3.5) {$\bar{x}_1$}; \node[left] at (5,4.5) {$\vdots \enspace$}; \node[left] at (5,5.5) {$\bar{x}_n$};
\draw[black,fill=black] (5,3.5) circle (.5ex); \draw[black,fill=black] (5,4.5) circle (.5ex); \draw[black,fill=black] (5,5.5) circle (.5ex);
\draw[black,fill=white] (5,2.5) circle (.5ex); \draw[black,fill=white] (5,1.5) circle (.5ex); \draw[black,fill=white] (5,0.5) circle (.5ex);
\draw[black,fill=black] (7,1.5) circle (.5ex); \draw[black,fill=black] (7.5,2) circle (.5ex); \draw[black,fill=black] (8,2.5) circle (.5ex);
\draw[black,fill=white] (7,4.5) circle (.5ex); \draw[black,fill=white] (7.5,4) circle (.5ex); \draw[black,fill=white] (8,3.5) circle (.5ex);
\end{tikzpicture}}
\right)=\prod_{i,j=1}^n \prod_{m=0}^{k-1}\left(1-x_i y_j t^m \right).
\label{eq:CauchyRHS}
\end{equation}

\noindent The first factor can be expressed as
\begin{align*}
\resizebox{4cm}{!}{  \begin{tikzpicture}[baseline=(current bounding box.center)]
\draw[fill=lightgray] (0,3) rectangle (5,6);
\draw[step=1.0,black,thin] (0,0) grid (5,6);
\node[above] at (4.5,6) {$...$};
\node[left] at (0,0.5) {$y_1$}; \node[left] at (0,1.5) {$\vdots \enspace$}; \node[left] at (0,2.5) {$y_n$};
\node[left] at (0,3.5) {$\bar{x}_1$}; \node[left] at (0,4.5) {$\vdots \enspace$}; \node[left] at (0,5.5) {$\bar{x}_n$};
\draw[black,fill=black] (0.5,0) circle (.5ex); \draw[black,fill=black] (1.5,0) circle (.5ex); \draw[black,fill=black] (2.5,0) circle (.5ex);
\draw[black,fill=black] (5,3.5) circle (.5ex); \draw[black,fill=black] (5,4.5) circle (.5ex); \draw[black,fill=black] (5,5.5) circle (.5ex);
\draw[black,fill=white] (5,0.5) circle (.5ex); \draw[black,fill=white] (5,1.5) circle (.5ex); \draw[black,fill=white] (5,2.5) circle (.5ex);
\draw[black,fill=white] (0,3.5) circle (.5ex); \draw[black,fill=white] (0,4.5) circle (.5ex); \draw[black,fill=white] (0,5.5) circle (.5ex);
\draw[black,fill=white] (0,0.5) circle (.5ex); \draw[black,fill=white] (0,1.5) circle (.5ex); \draw[black,fill=white] (0,2.5) circle (.5ex);
\end{tikzpicture}  } 
= \sum_{\bm{\lambda} }
\resizebox{8cm}{!}{ \begin{tikzpicture}[baseline=(current bounding box.center)]
\draw[step=1.0,black,thin] (0,0) grid (5,3);
\node[above] at (4.5,3) {$...$};
\node[left] at (0,0.5) {$y_1$}; \node[left] at (0,1.5) {$\vdots \enspace$}; \node[left] at (0,2.5) {$y_n$};
\draw[black,fill=black] (0.5,0) circle (.5ex); \draw[black,fill=black] (1.5,0) circle (.5ex); \draw[black,fill=black] (2.5,0) circle (.5ex);
\node[above] at (2.5,3) {$\bm{\lambda}$};
\draw[black,fill=white] (0,0.5) circle (.5ex); \draw[black,fill=white] (0,1.5) circle (.5ex); \draw[black,fill=white] (0,2.5) circle (.5ex);
\draw[black,fill=white] (5,0.5) circle (.5ex); \draw[black,fill=white] (5,1.5) circle (.5ex); \draw[black,fill=white] (5,2.5) circle (.5ex);
\node at (6,1.5) {$\times$};
\draw[fill=lightgray] (0+7,0) rectangle (5+7,3);
\draw[step=1.0,black,thin] (0+7,0) grid (5+7,3);
\node[above] at (4.5+7,3) {$\cdots$};
\node[left] at (0+7,0.5) {$\bar{x}_1$}; \node[left] at (0+7,1.5) {$\vdots \enspace$}; \node[left] at (0+7,2.5) {$\bar{x}_n$};
\draw[black,fill=black] (5+7,0.5) circle (.5ex); \draw[black,fill=black] (5+7,1.5) circle (.5ex); \draw[black,fill=black] (5+7,2.5) circle (.5ex);
 \node[below] at (2.5+7,0) {$\bm{\lambda}$};
\draw[black,fill=white] (0+7,0.5) circle (.5ex); \draw[black,fill=white] (0+7,1.5) circle (.5ex); \draw[black,fill=white] (0+7,2.5) circle (.5ex);
\end{tikzpicture}  } 
\end{align*}
We know from Theorem \ref{LatticeModel} that the first factor in the sum is exactly $\mathcal{L}_{\bm{\lambda}}(y_1,\ldots,y_n;t)$, hence we are left with computing
\begin{equation}
\mc{L}^*_{\bm{\lambda}}(x_1,...,x_n;t) := \resizebox{4cm}{!}{ \begin{tikzpicture}[baseline=(current bounding box.center)]
\draw[fill=lightgray] (0,0) rectangle (5,3);
\draw[step=1.0,black,thin] (0,0) grid (5,3);
\node[above] at (4.5,3) {$\cdots$};
\node[left] at (0,0.5) {$\bar{x}_1$}; \node[left] at (0,1.5) {$\vdots \enspace$}; \node[left] at (0,2.5) {$\bar{x}_n$};
\draw[black,fill=black] (5,0.5) circle (.5ex); \draw[black,fill=black] (5,1.5) circle (.5ex); \draw[black,fill=black] (5,2.5) circle (.5ex);
\node[below] at (2.5,0) {$\bm{\lambda}$}; 
\draw[black,fill=white] (0,0.5) circle (.5ex); \draw[black,fill=white] (0,1.5) circle (.5ex); \draw[black,fill=white] (0,2.5) circle (.5ex);
\end{tikzpicture} }.
\end{equation}
To that end, we start with the following identities of LLT polynomials.
\begin{prop} \label{BoxSkew}
Let $\bm{\lambda} = (\lambda^{(1)}, \ldots, \lambda^{(k)})$ be a tuple of partitions, each with $n$ non-negative parts. Fix the number of columns $M$ such that $\band(\bm{\lambda}) < M$, in the notation of Section \ref{sec:latticemodel}. Let $\bm{\lambda}^c = (\lambda^{(k),c},\ldots,\lambda^{(1),c})$, where $\lambda^{(i),c}$ denotes the complement of $\lambda^{(i)}$ in an $(M-n)\times n$ box. We have
\[
\resizebox{4cm}{!}{ \begin{tikzpicture}[baseline=(current bounding box.center)]
\draw (0,0) rectangle (5,3);
\draw[step=1.0,black,thin] (0,0) grid (5,3);
\node at (2.5,3.5) {$M$}; \draw[<-] (0,3.5)--(2.2,3.5); \draw[->] (2.8,3.5)--(5,3.5);
\node[left] at (0,0.5) {$x_1$}; \node[left] at (0,1.5) {$\vdots \enspace$}; \node[left] at (0,2.5) {$x_n$};
\draw[black,fill=black] (2.5,3) circle (.5ex); \draw[black,fill=black] (3.5,3) circle (.5ex); \draw[black,fill=black] (4.5,3) circle (.5ex);
\node[below] at (2.5,0) {$\bm{\lambda}$};
\draw[black,fill=white] (0,0.5) circle (.5ex); \draw[black,fill=white] (0,1.5) circle (.5ex); \draw[black,fill=white] (0,2.5) circle (.5ex);
\draw[black,fill=white] (5,0.5) circle (.5ex); \draw[black,fill=white] (5,1.5) circle (.5ex); \draw[black,fill=white] (5,2.5) circle (.5ex);
\end{tikzpicture} } = t^{d(\bm{\lambda})} \mathcal{L}_{\bm{\lambda}^c}(X_n;t)
\]
where $d(\bm{\lambda)}$ is given below in Lemma~\ref{lem:BoxComp}. In other words, if $B = (M-n,\ldots,M-n)$ with $n$ parts, and $\bm{B} = (B,\ldots,B)$ with $k$ copies of $B$, then
\begin{equation} \label{eq:BoxSkew}
\mathcal{L}_{\bm{B}/\bm{\lambda}}(X_n;t) = t^{d(\bm{\lambda})} \mathcal{L}_{\bm{\lambda}^c}(X_n;t)
\end{equation}
\end{prop}

\begin{prop}\label{Complement}
Let $\bm{\lambda}$, $M$ be as in Proposition \ref{BoxSkew}. Then,
\begin{equation}
\mathcal{L}_{\bm{\lambda}}(X_n;t) = (x_1\cdots x_n)^{k(M-n)}t^{\widetilde{d}(\bm{\lambda})}\mathcal{L}_{\bm{\lambda}^c}(X_n^{-1};t)
\label{complementEq}
\end{equation}
where $\widetilde{d}(\bm{\lambda})$ is given in Lemma~\ref{lem:Comp}.
\end{prop}

The proofs for Proposition \ref{BoxSkew} and Proposition \ref{Complement} come in a similar flavor. For each, we construct a bijection between configurations of the shape $\bm{\lambda^c}$ and those of the shape $\bm{B}/\bm{\lambda}$ for the former, and shape $\bm{\lambda}$ for the latter. We then show each bijection is weight-preserving, up to an overall power of $t$ and the variables $X$. To avoid clutter, we postpone the weight-preserving part of the proofs to Appendix \ref{lemma-proofs}.

\begin{proof}[Proof of Proposition \ref{BoxSkew}]
There is a bijection between configurations with bottom boundary $\bm{\lambda}$ and top boundary $\bm{B}$, and configurations with bottom boundary $\bm{\emptyset}$ and top boundary $\bm{\lambda}^c$ given by rotating the configuration by 180 degrees and reversing the colors. For example, with two colors we have
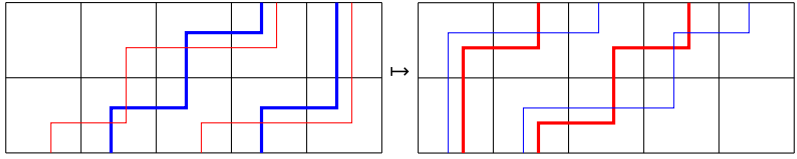
\[
\begin{tikzpicture}[baseline=(current bounding box.center)]
\draw (0,0) -- (0,2); \draw (1,0) -- (1,2); \draw (2,0) -- (2,2);  \draw (3,0) -- (3,2); \draw (4,0) -- (4,2); \draw (5,0) -- (5,2);
\draw (0,0) -- (5,0); \draw (0,1) -- (5,1); \draw (0,2) -- (5,2);

\draw[blue, very thick] (1.4,0)--(1.4,0.6)--(2.4,0.6)--(2.4,1.6)--(3.4,1.6)--(3.4,2);
\draw[blue, very thick] (3.4,0)--(3.4,0.6)--(4.4,0.6)--(4.4,2);
\draw[red] (0.6,0)--(0.6,0.4)--(1.6,0.4)--(1.6,1.4)--(3.6,1.4)--(3.6,2);
\draw[red] (2.6,0)--(2.6,0.4)--(4.6,0.4)--(4.6,2);
\end{tikzpicture}
\mapsto
\begin{tikzpicture}[xscale=-1,yscale=-1,baseline=(current bounding box.center)]
\draw (0,0) -- (0,2); \draw (1,0) -- (1,2); \draw (2,0) -- (2,2);  \draw (3,0) -- (3,2); \draw (4,0) -- (4,2); \draw (5,0) -- (5,2);
\draw (0,0) -- (5,0); \draw (0,1) -- (5,1); \draw (0,2) -- (5,2);

\draw[red, very thick] (1.4,0)--(1.4,0.6)--(2.4,0.6)--(2.4,1.6)--(3.4,1.6)--(3.4,2);
\draw[red, very thick] (3.4,0)--(3.4,0.6)--(4.4,0.6)--(4.4,2);
\draw[blue] (0.6,0)--(0.6,0.4)--(1.6,0.4)--(1.6,1.4)--(3.6,1.4)--(3.6,2);
\draw[blue] (2.6,0)--(2.6,0.4)--(4.6,0.4)--(4.6,2);
\end{tikzpicture}
\]
Note that horizontal steps remain horizontal under this bijection, so the $X$ weight remains unchanged. The proposition will follow from Lemma \ref{lem:BoxComp} below.
\end{proof}

\begin{lem}\label{lem:BoxComp}
Fix $\bm{\lambda}$, $M$, and $\bm{B}$ as in Proposition~\ref{BoxSkew}. If $T$ is a configuration of $\bm{B}/\bm{\lambda}$, we let $\Phi(T)$ be the configuration of $\bm{\lambda}^c$ under the bijection in the proof of Proposition~\ref{BoxSkew}. Then $d(\bm{\lambda}):=\coinv(T)-\coinv(\Phi(T))$ is independent of the choice of $T$. In particular,
\begin{equation}
d(\bm{\lambda}) = \binom{n}{2}\binom{k}{2} - \# \{i,j, a < b\; |\; \lambda_i^{(b)}-i < \lambda_j^{(a)}-j\}.
\end{equation}
\end{lem}
\begin{proof} Postponed to Appendix \ref{lemma-proofs}. \end{proof}

\begin{proof}[Proof of Proposition \ref{Complement}]
The reader will note that upon specialization at $t=1$ of (\ref{complementEq}), we arrive at a well-known identity of Schur functions \cite[Ex. 7.41]{stanley1999enumerative}, which is proven combinatorially using a bijection between SSYT of a shape and those of its complement shape. We recall that bijection here.

Given a partition $\mu$ with $n \geq \ell(\mu)$ and $N \geq \mu_1$, let $T \in \SSYT(\mu)$. Let $\nu^1, \ldots, \nu^{N}$ denote the (possibly empty) columns of $T$, left to right. Let $\widetilde{\nu}^i$ be the column whose entries are $ [n] \setminus \{j \mid j \in \nu^i \}$, arranged in increasing order. Define $\Phi(T)$ to be the tableau with columns $\widetilde{\nu}^{N}, \ldots, \widetilde{\nu}^1$, left to right. An example is given below with $n=4$ and $N=3$.
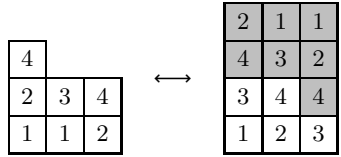
\[ \begin{ytableau}
\none \\
4 \\
2 & 3 & 4 \\
1 & 1 & 2
\end{ytableau}
\quad \longleftrightarrow \quad
\begin{ytableau}
*(lightgray) 2 & *(lightgray) 1 & *(lightgray) 1\\
*(lightgray) 4 & *(lightgray) 3 & *(lightgray) 2 \\
3 & 4 & *(lightgray) 4 \\
1 & 2 & 3
\end{ytableau}
\]
Now, we set $N=M-n$ and if $T=(T^{(1)}, \ldots, T^{(k)}) \in \SSYT(\bm{\lambda)}$, then we abuse notation and write $\Phi(T) = (\Phi(T^{(k)}), \ldots, \Phi(T^{(1)}))$, to be the reversal of applying the bijection individually to each tableau. We will abuse notation further and also write $\Phi(T)$ for the corresponding configuration of $\bm{\lambda}^c$, when $T$ is a configuration of $\bm{\lambda}$.

Note that if $T$ has weight $x^T$, then $\Phi(T)$ will have weight $(x^{-1})^T (x_1 \cdots x_n)^{k(M-n)}$. The proposition then follows from Lemma \ref{lem:Comp} below.
\end{proof}

\begin{lem}\label{lem:Comp}
Fix $\bm{\lambda}$ and $M$ as in Proposition~\ref{BoxSkew}. If $T$ is a configuration of $\bm{\lambda}$, we let $\Phi(T)$ be the configuration of $\bm{\lambda}^c$ under the bijection in the proof of Proposition~\ref{Complement}. Then, $\widetilde{d}(\bm{\lambda}) = \coinv T - \coinv \Phi(T)$ is independent of $T$. In particular,
\begin{equation}
\widetilde{d}(\bm{\lambda}) = (k-1)|\bm{\lambda}| - n(M-n)\binom{k}{2}
\end{equation}
\end{lem}
\begin{proof} Postponed to Appendix \ref{lemma-proofs}. \end{proof}

\begin{remark}
We note that Proposition \ref{Complement} has a representation-theoretic meaning. At $q=1$, (\ref{complementEq}) is simply a statement of contragredient duality for (tensor products of) irreducible representations of $GL_n$. One can use this same duality, together with the machinery of affine Hecke algebras, to conclude the statement for arbitrary $q$.

A similar result was proven in \cite[Prop. 10]{shimozono2000graded} for the class of generalized Hall-Littlewood polynomials defined therein. These polynomials coincide, up to a power of $t$, with coinversion LLT polynomials when the indexing tuple consists of rectangles.
\end{remark}

Continuing with our calculation of $\mc{L}^\ast_{\bm{\lambda}}$, we switch from the $L$ weights in Proposition \ref{BoxSkew} to those of $L^*$.
\begin{cor} \label{cor:CauchyCor1}
Fix $\bm{\lambda}$ and $M$ as in Proposition~\ref{BoxSkew}. We have
\[
\resizebox{4cm}{!}{ \begin{tikzpicture}[baseline=(current bounding box.center)]
\draw[fill=lightgray] (0,0) rectangle (5,3);
\draw[step=1.0,black,thin] (0,0) grid (5,3);
\node at (2.5,3.5) {$M$}; \draw[<-] (0,3.5)--(2.2,3.5); \draw[->] (2.8,3.5)--(5,3.5);
\node[left] at (0,0.5) {$\bar{x}_1$}; \node[left] at (0,1.5) {$\vdots \enspace$}; \node[left] at (0,2.5) {$\bar{x}_n$};
\draw[black,fill=black] (2.5,3) circle (.5ex); \draw[black,fill=black] (3.5,3) circle (.5ex); \draw[black,fill=black] (4.5,3) circle (.5ex);
\node[below] at (2.5,0) {$\bm{\lambda}$};
\draw[black,fill=white] (0,0.5) circle (.5ex); \draw[black,fill=white] (0,1.5) circle (.5ex); \draw[black,fill=white] (0,2.5) circle (.5ex);
\draw[black,fill=white] (5,0.5) circle (.5ex); \draw[black,fill=white] (5,1.5) circle (.5ex); \draw[black,fill=white] (5,2.5) circle (.5ex);
\end{tikzpicture} } = (x_1\cdots x_n)^{kM} t^{n^2\binom{k}{2}} t^{\widetilde{d}(\bm{\lambda}) + d(\bm{\lambda})} \mathcal{L}_{\bm{\lambda}^c}(X_n^{-1};t)
\]
\end{cor}
\begin{proof}
To switch to $L^*$ weights in Proposition \ref{BoxSkew} we first take $x_i\mapsto \frac{1}{x_i t^{k-1}}$ and then multiply each face by $x_i^k t^{\binom{k}{2}}$. This gives
\[
\begin{aligned}
& (x_1\cdots x_n)^{kM} t^{nM\binom{k}{2}} t^{d(\bm{\lambda})} \mathcal{L}_{\bm{\lambda}^c}(x_1^{-1}t^{-k+1},\ldots,x_n^{-1}t^{-k+1};t) \\
& = (x_1\cdots x_n)^{kM} t^{nM\binom{k}{2}} t^{d(\bm{\lambda})} t^{-(k-1)|\bm{\lambda}^c|} \mathcal{L}_{\bm{\lambda}^c}(x_1^{-1},\ldots,x_n^{-1};t) \\
& = (x_1\cdots x_n)^{kM} t^{nM\binom{k}{2}} t^{d(\bm{\lambda})} t^{(k-1)|\bm{\lambda}|-k(k-1)n(M-n)} \mathcal{L}_{\bm{\lambda}^c}(X_n^{-1};t) \\
& = (x_1\cdots x_n)^{kM} t^{n^2\binom{k}{2}}t^{-n(M-n)\binom{k}{2}} t^{(k-1)|\bm{\lambda}|}t^{d(\bm{\lambda})} \mathcal{L}_{\bm{\lambda}^c}(X_n^{-1};t) \\
& = (x_1\cdots x_n)^{kM} t^{n^2\binom{k}{2}}t^{\widetilde{d}(\bm{\lambda}) + d(\bm{\lambda})} \mathcal{L}_{\bm{\lambda}^c}(X_n^{-1};t)
\end{aligned}
\]
as desired.
\end{proof}

\begin{cor}\label{cor:CauchyCor2}
Fix $\bm{\lambda}$ and $M$ as in Proposition \ref{BoxSkew}. We have
\[
\resizebox{4cm}{!}{ \begin{tikzpicture}[baseline=(current bounding box.center)]
\draw[fill=lightgray] (0,0) rectangle (5,3);
\draw[step=1.0,black,thin] (0,0) grid (5,3);
\node at (2.5,3.5) {$M$}; \draw[<-] (0,3.5)--(2.2,3.5); \draw[->] (2.8,3.5)--(5,3.5);
\node[left] at (0,0.5) {$\bar{x}_1$}; \node[left] at (0,1.5) {$\vdots \enspace$}; \node[left] at (0,2.5) {$\bar{x}_n$};
\draw[black,fill=black] (5,0.5) circle (.5ex); \draw[black,fill=black] (5,1.5) circle (.5ex); \draw[black,fill=black] (5,2.5) circle (.5ex);
\node[below] at (2.5,0) {$\bm{\lambda}$};
\draw[black,fill=white] (0,0.5) circle (.5ex); \draw[black,fill=white] (0,1.5) circle (.5ex); \draw[black,fill=white] (0,2.5) circle (.5ex);
\end{tikzpicture} } =\left(x^\rho \right)^k (x_1\cdots x_n)^{k(M-n)} t^{\binom{n}{2}\binom{k}{2} } t^{\widetilde{d}(\bm{\lambda}) + d(\bm{\lambda})} \mathcal{L}_{\bm{\lambda}^c}(X_n^{-1};t)
\]
\end{cor}
\begin{proof}
There is a weight-preserving bijection from configurations in Corollary \ref{cor:CauchyCor1} to configurations here by taking
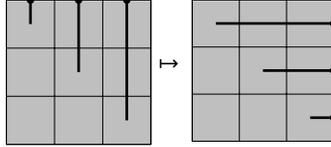
\[
\resizebox{2cm}{!}{ \begin{tikzpicture}[baseline=(current bounding box.center)]
\draw[fill=lightgray] (0,0) rectangle (3,3);
\draw[step=1.0,black,thin] (0,0) grid (3,3);
\draw[black,fill=black] (0.5,3) circle (.5ex); \draw[black,fill=black] (1.5,3) circle (.5ex); \draw[black,fill=black] (2.5,3) circle (.5ex);
\draw[ultra thick] (0.5,2.5)--(0.5,3); \draw[ultra thick] (1.5,1.5)--(1.5,3); \draw[ultra thick] (2.5,0.5)--(2.5,3); 
\end{tikzpicture}}
\mapsto
\resizebox{2cm}{!}{ \begin{tikzpicture}[baseline=(current bounding box.center)]
\draw[fill=lightgray] (0,0) rectangle (3,3);
\draw[step=1.0,black,thin] (0,0) grid (3,3);
\draw[black,fill=black] (3,0.5) circle (.5ex); \draw[black,fill=black] (3,1.5) circle (.5ex); \draw[black,fill=black] (3,2.5) circle (.5ex);
\draw[ultra thick] (0.5,2.5)--(3,2.5); \draw[ultra thick] (1.5,1.5)--(3,1.5); \draw[ultra thick] (2.5,0.5)--(3,0.5); 
\end{tikzpicture}}
\]
and multiplying by $(x_1\cdots x_n)^{-nk} \left(x^\rho\right)^k t^{-\binom{n+1}{2}\binom{k}{2}}$.
\end{proof}

\begin{prop} \label{L-star}
Let $\bm{\lambda} = (\lambda^{(1)}, \ldots, \lambda^{(k)})$ be a tuple of partitions, each with $n$ non-negative parts. Then,
\[
\mc{L}^\ast_{\bm{\lambda}}(X_n;t) = \left(x^\rho \right)^k  t^{\binom{n}{2}\binom{k}{2} } t^{d(\bm{\lambda})}  \mathcal{L}_{\bm{\lambda}}(X_n;t)
\]
\end{prop}
\begin{proof}
Fix the number of columns $M$ as usual. Combining Proposition \ref{Complement} and Corollary \ref{cor:CauchyCor2} yields
\[
\resizebox{4cm}{!}{ \begin{tikzpicture}[baseline=(current bounding box.center)]
\draw[fill=lightgray] (0,0) rectangle (5,3);
\draw[step=1.0,black,thin] (0,0) grid (5,3);
\node at (2.5,3.5) {$M$}; \draw[<-] (0,3.5)--(2.2,3.5); \draw[->] (2.8,3.5)--(5,3.5);
\node[left] at (0,0.5) {$\bar{x}_1$}; \node[left] at (0,1.5) {$\vdots \enspace$}; \node[left] at (0,2.5) {$\bar{x}_n$};
\draw[black,fill=black] (5,0.5) circle (.5ex); \draw[black,fill=black] (5,1.5) circle (.5ex); \draw[black,fill=black] (5,2.5) circle (.5ex);
\node[below] at (2.5,0) {$\bm{\lambda}$};
\draw[black,fill=white] (0,0.5) circle (.5ex); \draw[black,fill=white] (0,1.5) circle (.5ex); \draw[black,fill=white] (0,2.5) circle (.5ex);
\end{tikzpicture} } = \left(x^\rho \right)^k  t^{\binom{n}{2}\binom{k}{2} } t^{d(\bm{\lambda})}  \mathcal{L}_{\bm{\lambda}}(X_n;t)
\]
Note that this is independent of $M$. Taking $M\to \infty$ gives the proposition.
\end{proof}

We are finally ready to prove the Cauchy identity.
\begin{proof}[Proof of Theorem \ref{thm:CauchyIdentity}]
Combining (\ref{eq:CauchyLHS}), (\ref{eq:CauchyRHS}) and Proposition \ref{L-star}, we see that (\ref{CauchyEqn}) becomes
\[
\left(x^\rho \right)^k t^{\binom{n}{2}\binom{k}{2}} =\prod_{i,j=1}^n\prod_{m=0}^{k-1} \left(1-x_iy_j t^m \right) \sum_{\bm{\lambda}} \left(x^\rho \right)^k  t^{\binom{n}{2}\binom{k}{2} } t^{d(\bm{\lambda})}  \mathcal{L}_{\bm{\lambda}}(X_n;t)\mathcal{L}_{\bm{\lambda}}(Y_n;t).
\]
Rearranging gives the desired identity.
\end{proof}

The same style of arguments hold when we let the paths enter the bottom at positions indexed by a tuple of partitions $\bm{\mu}$ 
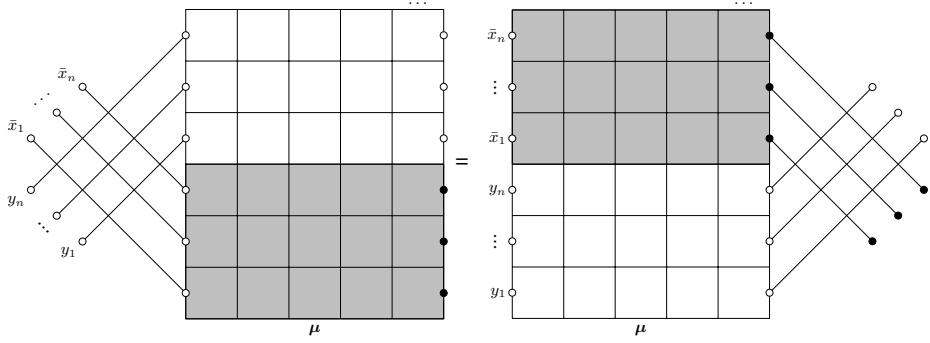
\[
\resizebox{6cm}{!}{ \begin{tikzpicture}[baseline=(current bounding box.center)]
\draw[fill=lightgray] (0,0) rectangle (5,3);
\draw[step=1.0,black,thin] (0,0) grid (5,6);
\draw (0,0.5) -- (-3,3.5); \draw (0,1.5) -- (-2.5,4); \draw (0,2.5) -- (-2,4.5); 
\draw (0,3.5) -- (-2,1.5); \draw (0,4.5) -- (-2.5,2); \draw (0,5.5) -- (-3,2.5); 
\node[above] at (4.5,6) {$\ldots$};
\node[below left] at (-2,1.5) {$y_1$}; \node[below left] at (-2.5,2) {$\ddots$}; \node[below left] at (-3,2.5) {$y_n$};
\node[above left] at (-3,3.5) {$\bar{x}_1$}; \node[above left] at (-2.5,4) {$\iddots$}; \node[above left] at (-2,4.5) {$\bar{x}_n$};
\draw[black,fill=black] (5,0.5) circle (.5ex); \draw[black,fill=black] (5,1.5) circle (.5ex); \draw[black,fill=black] (5,2.5) circle (.5ex);
\node[below] at (2.5,0) {$\bm{\mu}$};
\draw[black,fill=white] (0,0.5) circle (.5ex); \draw[black,fill=white] (0,1.5) circle (.5ex); \draw[black,fill=white] (0,2.5) circle (.5ex);
\draw[black,fill=white] (0,3.5) circle (.5ex); \draw[black,fill=white] (0,4.5) circle (.5ex); \draw[black,fill=white] (0,5.5) circle (.5ex);
\draw[black,fill=white] (5,3.5) circle (.5ex); \draw[black,fill=white] (5,4.5) circle (.5ex); \draw[black,fill=white] (5,5.5) circle (.5ex);
\draw[black,fill=white] (-2,1.5) circle (.5ex); \draw[black,fill=white] (-2.5,2) circle (.5ex); \draw[black,fill=white] (-3,2.5) circle (.5ex);
\draw[black,fill=white] (-2,4.5) circle (.5ex); \draw[black,fill=white] (-2.5,4) circle (.5ex); \draw[black,fill=white] (-3,3.5) circle (.5ex);
\end{tikzpicture}}
=
\resizebox{6cm}{!}{ \begin{tikzpicture}[baseline=(current bounding box.center)]
\draw[fill=lightgray] (0,3) rectangle (5,6);
\draw[step=1.0,black,thin] (0,0) grid (5,6);
\draw (5,0.5) -- (8,3.5); \draw (5,1.5) -- (7.5,4); \draw (5,2.5) -- (7,4.5); 
\draw (5,3.5) -- (7,1.5); \draw (5,4.5) -- (7.5,2); \draw (5,5.5) -- (8,2.5); 
\node[above] at (4.5,6) {$\ldots$};
\node[left] at (0,0.5) {$y_1$}; \node[left] at (0,1.5) {$\vdots \enspace$}; \node[left] at (0,2.5) {$y_n$};
\node[left] at (0,3.5) {$\bar{x}_1$}; \node[left] at (0,4.5) {$\vdots \enspace$}; \node[left] at (0,5.5) {$\bar{x}_n$};
\draw[black,fill=black] (7,1.5) circle (.5ex); \draw[black,fill=black] (7.5,2) circle (.5ex); \draw[black,fill=black] (8,2.5) circle (.5ex);
\draw[black,fill=black] (5,3.5) circle (.5ex); \draw[black,fill=black] (5,4.5) circle (.5ex); \draw[black,fill=black] (5,5.5) circle (.5ex);
\node[below] at (2.5,0) {$\bm{\mu}$};
\draw[black,fill=white] (0,0.5) circle (.5ex); \draw[black,fill=white] (0,1.5) circle (.5ex); \draw[black,fill=white] (0,2.5) circle (.5ex);
\draw[black,fill=white] (0,3.5) circle (.5ex); \draw[black,fill=white] (0,4.5) circle (.5ex); \draw[black,fill=white] (0,5.5) circle (.5ex);
\draw[black,fill=white] (5,0.5) circle (.5ex); \draw[black,fill=white] (5,1.5) circle (.5ex); \draw[black,fill=white] (5,2.5) circle (.5ex);
\draw[black,fill=white] (7,4.5) circle (.5ex); \draw[black,fill=white] (7.5,4) circle (.5ex); \draw[black,fill=white] (8,3.5) circle (.5ex);
\end{tikzpicture}}
\]
From this we have
\begin{prop}
\[
\sum_{\bm{\lambda}} t^{d(\bm{\lambda})} \mathcal{L}_{\bm{\lambda}}(X_n;t)\mathcal{L}_{\bm{\lambda}/\bm{\mu}}(Y_n;t) = t^{d(\bm{\mu})}\mathcal{L}_{\bm{\mu}}(X_n;t) \prod_{i,j=1}^n\prod_{m=0}^{k-1} \left(1-x_iy_j t^m \right)^{-1}.
\]
\end{prop}

Define $\bm{\lambda}^{\text{rot}}:= \left(\lambda^{(k),\text{rot}},\ldots,\lambda^{(1),\text{rot}}\right)$ where $\lambda^{\text{rot}}$ is $\lambda$ rotated by $180$ degrees. Take $\bm{B}$ to be the smallest box containing $\bm{\lambda}$ and let $\bm{\lambda}^c$ be the complement taken in this box. Plugging in $\bm{\lambda}^c$ into Equation \eqref{eq:BoxSkew} for this choice of $\bm{B}$ we see that
\[
\mathcal{L}_{\bm{\lambda}^{\text{rot}}}(X_n;t) = t^{d\left(\bm{\lambda}^c\right)} \mathcal{L}_{\bm{\lambda}}(X_n;t).
\]
Using this to replace one of the $\mathcal{L}_{\bm{\lambda}}$ in the sum in Theorem \ref{thm:CauchyIdentity}, and noting that $d(\bm{\lambda}^c)=d(\bm{\lambda})$, we can reformulate the Cauchy identity as
\begin{cor}
\[
 \sum_{\bm{\lambda}}  \mathcal{L}_{\bm{\lambda}}(X_n;t)\mathcal{L}_{\bm{\lambda}^{\text{rot}}}(Y_n;t) = \prod_{i,j=1}^n\prod_{m=0}^{k-1} \left(1-x_iy_j t^m \right)^{-1}.
\]
\end{cor}

\appendix
\addcontentsline{toc}{section}{Appendices}
\section*{Appendices}

\section{Proof of the YBE \label{YBE-proof}}

\tikzstyle{every picture}=[baseline=(current bounding box.center)]

\indent Recall the recursive formulations of the $L$ and $R$ matrices from Equations \eqref{LRrecursive} and \eqref{EFweights}.  Our goal is to show that, for any choice of $\I_1,\I_2,\I_3,\J_1,\J_2,\J_3 \in \{0,1\}^k$,
\begin{equation} \label{YBE(matrices)}
\begin{aligned}
&\sum_{\L_1,\L_2,\L_3 \in \{0,1\}^k} R^{(k)}_{y/x}(\L_2,\L_1;\J_2,\J_1) L^{(k)}_x(\L_3,\I_1;\J_3,\L_1) L^{(k)}_y(\I_3,\I_2;\L_3,\L_2) \\
&= \sum_{\K_1,\K_2,\K_3 \in \{0,1\}^k} L^{(k)}_y(\K_3,\K_2;\J_3,\J_2) L^{(k)}_x(\I_3,\K_1;\K_3,\J_1) R^{(k)}_{y/x}(\I_2,\I_1;\K_2,\K_1).
\end{aligned}
\end{equation}
\noindent Graphically this means that $L$ and $R$ satisfy
\begin{equation} \label{YBE(lattices)}
 \begin{tikzpicture}[baseline=(current bounding box.center)] 
 \draw (-1,0.5) -- (0,1.5); \draw (-1,1.5) -- (0,0.5); 
 \draw[step=1.0,black,thin] (0,0) grid (1,2); 
 \node[left] at (-1,1.5) {$\I_1$}; \node[left] at (-1,0.5) {$\I_2$}; \node[below] at (0.5,0) {$\I_3$};
 \node[right] at (1,1.5) {$\J_2$}; \node[right] at (1,0.5) {$\J_1$}; \node[above] at (0.5,2) {$\J_3$};
 \node at (0.5,0.5) {$x$}; \node at (0.5,1.5) {$y$};
 \end{tikzpicture} 
=
 \begin{tikzpicture}[baseline=(current bounding box.center)] 
 \draw (2,0.5) -- (1,1.5); \draw (2,1.5) -- (1,0.5); 
 \draw[step=1.0,black,thin] (0,0) grid (1,2); 
 \node[left] at (0,1.5) {$\I_1$}; \node[left] at (0,0.5) {$\I_2$}; \node[below] at (0.5,0) {$\I_3$};
 \node[right] at (2,0.5) {$\J_1$}; \node[right] at (2,1.5) {$\J_2$}; \node[above] at (0.5,2) {$\J_3$};
 \node at (0.5,0.5) {$y$}; \node at (0.5,1.5) {$x$};
 \end{tikzpicture}
\end{equation}
where the choice of $\I$'s and $\J$'s correspond to a choice of boundary condition and we sum over the weight of all possible configurations that satisfy the chosen boundary condition. 

\indent We will now simplify the notation slightly.  We will let $V$ be the $2^k$-dimensional vector space consisting of formal linear combinations of elements of $\{0,1\}^k$.  If $A : V \otimes V \rightarrow V \otimes V$ is an operator, we define $A_{12},A_{13},A_{23}:V\otimes V\otimes V \to V\otimes V\otimes V$ by the rule that the subscripts indicate which factors of the tensor product the operator acts on by $A$. For example, $A_{12}=A\otimes I$ where $I$ is the identity on $V$. 
\noindent In particular, we will adopt this notation for the L and R matrices, which we can think of as operators $V \otimes V \rightarrow V \otimes V$ as follows.
\begin{align*}
&L^{(k)}_x : (\I,\J) \mapsto \sum_{\K,\L \in \{0,1\}^k} L^{(k)}_x(\I,\J;\K,\L) (\K \otimes \L) \\
&R^{(k)}_{y/x} : (\I,\J) \mapsto \sum_{\K,\L \in \{0,1\}^k} R^{(k)}_{y/x}(\I,\J;\K,\L) (\K \otimes \L)
\end{align*}
\noindent We will also let $L^{(k)}(x) = L^{(k)}_x$ and $R^{(k)}(y/x) = R^{(k)}_{y/x}$.  We can now rewrite Equation \eqref{YBE(matrices)} as
\begin{equation*}
G^{(k)}(y/x) := R^{(k)}_{12}(y/x) L^{(k)}_{13}(x) L^{(k)}_{23}(y) = L^{(k)}_{23}(y) L^{(k)}_{13}(x) R^{(k)}_{12}(y/x) =: D^{(k)}(y/x).
\end{equation*}
\noindent Here $G$ stands for gauche (left) and $D$ stands for droite (right).  Finally, we will often omit the superscripts on the L and R matrices, so we can again rewrite Equation \eqref{YBE(matrices)} as
\begin{equation*}
G(y/x) := R_{12}(y/x) L_{13}(x) L_{23}(y) = L_{23}(y) L_{13}(x) R_{12}(y/x) =: D(y/x).
\end{equation*}

\begin{definition} We can interpret Equation \eqref{YBE(lattices)} in terms of $k$-tuples of colored lattice paths.
\begin{enumerate}
    \item We will refer to the left-hand side of Equation \eqref{YBE(lattices)} as the gauche (or $G$) lattice.  A specification of the edge labels on the lattice corresponds a $k$-tuple of colored paths on the lattice.  These paths enter at the edges labelled $\I$ below; pass through the vertices labelled 12, 13, and 23 below; and exit at the edges labelled $\J$ below.
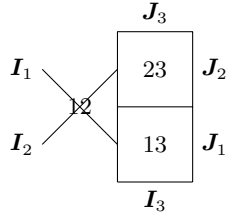
\begin{align*}
\begin{tikzpicture}
\gauche
\node at (1.5,1.5) {$23$};
\node at (1.5,0.5) {$13$};
\node at (0.5,1) {$12$};
\node[left] at (0,0.5) {$\I_2$};
\node[left] at (0,1.5) {$\I_1$};
\node[below] at (1.5,0) {$\I_3$};
\node[above] at (1.5,2) {$\J_3$};
\node[right] at (2,0.5) {$\J_1$};
\node[right] at (2,1.5) {$\J_2$};
\end{tikzpicture}
\end{align*}
    \item We will refer to the right-hand side of Equation \eqref{YBE(lattices)} as the droite (or $D$) lattice.  A specification of the edge labels on the lattice corresponds a $k$-tuple of colored paths on the lattice.  These paths enter at the edges labelled $\I$ below; pass through the vertices labelled 12, 13, and 23 below; and exit at the edges labelled $\J$ below.
\begin{align*}
\begin{tikzpicture}
\droite 
\node at (0.5,1.5) {$13$};
\node at (0.5,0.5) {$23$};
\node at (1.5,1) {$12$};
\node[left] at (0,1.5) {$\I_1$};
\node[left] at (0,0.5) {$\I_2$};
\node[below] at (0.5,0) {$\I_3$};
\node[right] at (2,0.5) {$\J_1$};
\node[right] at (2,1.5) {$\J_2$};
\node[above] at (0.5,2) {$\J_3$};
\end{tikzpicture}
\end{align*} 
\end{enumerate}
\end{definition} 

\indent Each entry in $G^{(k)}$ or $D^{(k)}$ is indexed by a word $w = (w_1,...,w_k)$, where the letter $w_i$ records where the paths of color $i$ enter and exit the lattice, 
in the alphabet $\{0,...,13\}$ as follows.

\noindent \begin{center} \resizebox{0.8\textwidth}{!}{ \begin{tabular}{ccccccc}
\begin{tikzpicture} \twobox \node[above] at (0.5,2) {$w_i=0$}; \end{tikzpicture} & \begin{tikzpicture} \twobox \draw[brightpink] (0,0.5)--(1,0.5); \node[above] at (0.5,2) {$w_i=1$}; \end{tikzpicture} & \begin{tikzpicture} \twobox \draw[brightpink] (0,1.5)--(1,1.5); \node[above] at (0.5,2) {$w_i=2$}; \end{tikzpicture} & \begin{tikzpicture} \twobox \draw[brightpink] (0,1.5)--(1,1.5); \draw[brightpink] (0,0.5)--(1,0.5); \node[above] at (0.5,2) {$w_i=3$}; \end{tikzpicture} & \begin{tikzpicture} \twobox \draw[brightpink] (0,0.5)--(0.5,0.5)--(0.5,1.5)--(1,1.5); \node[above] at (0.5,2) {$w_i=4$}; \end{tikzpicture} & \begin{tikzpicture} \twobox \draw[brightpink] (0,1.5)--(0.5,1.5)--(0.5,0.5)--(1,0.5);\node[above] at (0.5,2) {$w_i=5$}; \end{tikzpicture} &
\begin{tikzpicture} \twobox \draw[brightpink] (0.5,0)--(0.5,0.5)--(1,0.5); \node[above] at (0.5,2) {$w_i=6$}; \end{tikzpicture} \\
\begin{tikzpicture} \twobox \draw[brightpink] (0,1.5)--(0.5,1.5)--(0.5,2); \node[above] at (0.5,2) {$w_i=7$}; \end{tikzpicture} & \begin{tikzpicture} \twobox \draw[brightpink] (0,1.5)--(0.5,1.5)--(0.5,2); \draw[brightpink] (0,0.5)--(1,0.5); \node[above] at (0.5,2) {$w_i=8$}; \end{tikzpicture} & \begin{tikzpicture} \twobox \draw[brightpink] (0,1.5)--(0.5,1.5)--(0.5,2); \draw[brightpink] (0.5,0)--(0.5,0.5)--(1,0.5); \node[above] at (0.5,2) {$w_i=9$}; \end{tikzpicture} & \begin{tikzpicture} \twobox \draw[brightpink] (0.5,0)--(0.5,0.5)--(1,0.5); \draw[brightpink] (0,1.5)--(1,1.5); \node[above] at (0.5,2) {$w_i=10$}; \end{tikzpicture} & \begin{tikzpicture} \twobox \draw[brightpink] (0,0.5)--(0.5,0.5)--(0.5,2); \node[above] at (0.5,2) {$w_i=11$}; \end{tikzpicture} & \begin{tikzpicture} \twobox \draw[brightpink] (0.5,0)--(0.5,1.5)--(1,1.5); \node[above] at (0.5,2) {$w_i=12$}; \end{tikzpicture} & \begin{tikzpicture} \twobox \draw[brightpink] (0.5,0)--(0.5,2); \node[above] at (0.5,2) {$w_i=13$}; \end{tikzpicture}
\end{tabular} }
\end{center}

\noindent To be more precise, $w_i$ records the $i$-th components of the labels on the $\I$ and $\J$ edges.

\begin{center}
    \begin{tikzpicture}
    \twobox 
    \node[left] at (0,0.5) {$\I_2$};
    \node[left] at (0,1.5) {$\I_1$};
    \node[below] at (0.5,0) {$\I_3$};
    \node[right] at (1,0.5) {$\J_1$};
    \node[right] at (1,1.5) {$\J_2$};
    \node[above] at (0.5,2) {$\J_3$};
    \end{tikzpicture}
\end{center}

\begin{example}

There is only one path on the droite lattice that gives the letter 5, namely
\begin{align*}
\begin{tikzpicture} \droite \draw[brightpink] (2,0.4) -- (1,1.4) -- (0,1.4); \end{tikzpicture}
\end{align*}
\end{example}

\begin{example}
There are two paths on the gauche lattice that give the letter 2, namely
\begin{align*}
\begin{tikzpicture}[baseline=(current bounding box.center)] \gauche \draw[brightpink] (0,1.6) -- (0.5,1.1) -- (1,1.6) -- (1.4,1.6) -- (2,1.6); \end{tikzpicture}
\enspace \emph{ and }
\begin{tikzpicture}[baseline=(current bounding box.center)] \gauche \draw[brightpink] (0,1.6) -- (1,0.6) -- (1.4,0.6) -- (1.4,1.6) -- (2,1.6); \end{tikzpicture} \enspace.
\end{align*}
\end{example}

\begin{definition} We denote by $M_w(x,y)$ the entry indexed by $w$, with $M \in \{G,D\}$. \end{definition}

\begin{definition}
Given $M \in \{G,D\}$ and a $k$-tuple of paths $P = (P_1,...,P_k)$ for the colors $1,...,k$ on the $M$ lattice, we denote by $M_P(x,y)$ the resulting weight.  
\end{definition}

\begin{remark}
Given a word $w = (w_1,...,w_k)$ in the alphabet $\{0,...,13\}$, we have
\[ M_w(x,y) = \sum_P M_P(x,y) \]
where the sum is taken over all $k$-tuples of paths $P = (P_1,...,P_k)$ on the $M$ lattice so that the resulting word is $w$.
\end{remark}

\begin{example} Let brown be color 1 and let green be color 2.  Then
\begin{align*}
G_{(1,3)} = 
\begin{tikzpicture} \gauche \draw[brown] (0,0.4) -- (0.5,0.9) -- (1,0.4) -- (1.6,0.4) -- (2,0.4); \draw[green] (0,0.6) -- (1,1.6) -- (2,1.6); \draw[green] (0,1.6) -- (1,0.6) -- (2,0.6);  \end{tikzpicture}
\enspace , \enspace
D_{(1,3)} =  \begin{tikzpicture} \droite \draw[brown] (0.6,0) -- (0.6,1.4) -- (1,1.4) -- (2,0.4); \draw[green] (2,0.6) -- (1,1.6) -- (0,1.6); \draw[green] (2,1.6) -- (1,0.6) -- (0,0.6);  \end{tikzpicture} \enspace + 
\begin{tikzpicture} \droite \draw[brown] (0.6,0) -- (0.6,0.4) -- (1,0.4) -- (1.5,0.9) -- (2,0.4); \draw[green] (2,0.6) -- (1,1.6) -- (0,1.6); \draw[green] (2,1.6) -- (1,0.6) -- (0,0.6);  \end{tikzpicture} \enspace .
\end{align*}
\noindent Here we have omitted the superscripts for the sake of clarity, as we do often.
\end{example}

\indent We will also make use of the following definitions and notations.

\begin{definition}
Given a word $u$, we denote by $|u|$ the length of $u$.  Given a word $u$ and $i \in \{0,...,13\}$, we denote by $|u|_i$ the number of $i$'s in $u$.
\end{definition}

\begin{definition}
Recall the weights of $E$, $F(x)$, $\tilde{E}(y/x)$, and $\tilde{F}(y/x)$ from Equation \eqref{EFweights}.  The paths with non-zero weights in $F(x)$ are considered special for the $L$ vertices.  The paths with non-zero weights in $\tilde{E}(y/x)$ are considered special for the $R$ vertices.
\end{definition}

\begin{example}
The following path on the gauche lattice is special at the 12 and 23 vertices, but not at the 13 vertex.
\begin{align*}
\begin{tikzpicture} \gauche \draw[brightpink] (0,1.6) -- (1,0.6) -- (1.4,0.6) -- (1.4,1.6) -- (2,1.6);  \end{tikzpicture}
\end{align*}
\end{example}

\indent To prove Equation \eqref{YBE(matrices)}, it is enough to show that $G_w = D_w$ for all words $w$.  We will use a recursive argument to simplify $w$ as much as possible, and then check the equality on the base cases.  

\subsection{Reducing to the case where $w$ is a word in the alphabet $\{1,2,3,4,5\}$}

\indent We begin by showing that all 0's can be removed.

\begin{lem} \label{YBE0} For any words $u$ and $v$, we have $M_{(u,0,v)}(x,y) = M_{(u,v)}(x,y)$. \end{lem}

\begin{proof}

We do the case $M=G$ in detail.  (The case $M=D$ is similar.)   Note that there is only one path on the gauche lattice that gives the letter 0, namely
\begin{align*}
\pbullet_0 =
\begin{tikzpicture} \gauche \end{tikzpicture}
 \enspace. 
\end{align*}
Let $\ell = |v|$. Choose paths $P_{k-\ell+1},...,P_k$ on the gauche lattice so that the resulting word is $v$.  Let $\alpha,\beta,\gamma$ be the number of special paths among $P_{k-\ell+1},...,P_k$ for the 12, 13, 23 vertices respectively.  We have 
\begin{align*}
G^{(k)}_{(u,0,v)}(x,y) &= G^{(k)}_{(u,\pbullet_0,v)}(x,y) \\
&= \left [ R^{(k-\ell-1)}_{12}(y/xt^\alpha) L^{(k-\ell-1)}_{13}(xt^\beta) L^{(k-\ell-1)}_{23}(yt^\gamma) \right ]_u \cdot 1 \cdot G^{(\ell)}_{(P_{k-\ell+1},...,P_k)}(x,y) \\
&= \left [ R^{(k-\ell-1)}_{12}(y/xt^\alpha) L^{(k-\ell-1)}_{13}(xt^\beta) L^{(k-\ell-1)}_{23}(yt^\gamma) \right ]_u \cdot G^{(\ell)}_{(P_{k-\ell+1},...,P_k)}(x,y) \\
&= G^{(k-1)}_{(u,v)}(x,y).
\end{align*}

\end{proof}

\indent The next two lemmas demonstrate that we can reduce all 6's to 1's; all 7's to 2's; all 8's, 9's, and 10's to 3's; and all 11's, 12's, and 13's to 4's.

\begin{lem} For all $a \in \{6,9,10,12\}$, we have 
\[ M_{(u,a,v)}(x,y) = M_{(u,f(a),v)}(x,y) \]
where $f : \{6,9,10,12,13\} \rightarrow \{1,3,4,8,11\}$ is defined by 
\[ f(6) = 1, f(10) = 3, f(9) = 8, f(12) = 4, f(13) = 11. \] 
\end{lem}

\begin{proof}
The proof can be summarized by the equation
\[
\begin{tikzpicture}[baseline=(current bounding box.center)]
\twobox
\draw[brightpink] (0,0.5)--(0.5,0.5);
\end{tikzpicture}
=
\begin{tikzpicture}[baseline=(current bounding box.center)]
\twobox
\draw[brightpink] (0.5,0)--(0.5,0.5);
\end{tikzpicture}.
\]
We do the case $a=6, M = G$ in detail.  (The other cases are similar.)  Note that there is only one path on the gauche lattice that gives the letter $6$, namely
\begin{align*}
\pbullet_6 =
\begin{tikzpicture} \gauche \draw[brightpink] (1.6,0) -- (1.6,0.4) -- (2,0.4);  \end{tikzpicture}
 \enspace, 
\end{align*}
and that there is only one path on the gauche lattice that gives the letter $f(6) = 1$, namely
\begin{align*}
\pbullet_1 =
\begin{tikzpicture} \gauche \draw[brightpink] (0,0.4) -- (0.5,0.9) -- (1,0.4) -- (1.6,0.4) -- (2,0.4);  \end{tikzpicture} \enspace. 
\end{align*}
Let $\ell = |v|$. Choose paths $P_{k-\ell+1},...,P_k$ on the gauche lattice so that the resulting word is $v$.  Let $\alpha,\beta,\gamma$ be the number of special paths among $P_{k-\ell+1},...,P_k$ for the 12, 13, 23 vertices respectively.  We have 
\begin{align*}
&G^{(k)}_{(u,\pbullet_6,P_{k-\ell+1},...,P_k)}(x,y) \\
&= \left [ R^{(k-\ell-1)}_{12}(y/xt^\alpha) L^{(k-\ell-1)}_{13}(xt^{\beta+1}) L^{(k-\ell-1)}_{23}(yt^\gamma) \right ]_u \cdot xt^\beta \cdot G^{(\ell)}_{(P_{k-\ell+1},...,P_k)}(x,y) \\
&= G^{(k)}_{(u,\pbullet_1,P_{k-\ell+1},...,P_k)}(x,y).
\end{align*}
\end{proof}

\begin{lem} \label{YBE2} For all $a \in \{7,8,9,11,13\}$, if $v$ is a word in the alphabet $\{1,2,3,4,5\}$, then 
\[ y t^{|v|-|v|_1-|v|_5} M_{(u,a,v)}(x,y) = M_{(u,g(a),v)}(x,y),\]
where $g : \{7,8,9,11,13\} \rightarrow \{2,3,4,10,12\}$ is defined by 
\[ g(7) = 2, g(8) = 3, g(9) = 10, g(11) = 4, g(13) = 12.\] \end{lem}

\begin{proof}
The proof can be summarized by the equation
\[
\begin{tikzpicture}[baseline=(current bounding box.center)]
\twobox
\draw[brightpink] (0.5,1.5)--(1,1.5);
\end{tikzpicture}
=
\begin{tikzpicture}[baseline=(current bounding box.center)]
\twobox
\draw[brightpink] (0.5,1.5)--(0.5,2);
\end{tikzpicture}
\cdot y t^{|v|-|v|_1-|v|_5}.
\]
We do the case $a = 7, M = G$ in detail.  (The other cases are similar.)  Note that there are two paths on the gauche lattice that give the letter $7$, namely
\begin{align*}
\pbullet_7 = 
\begin{tikzpicture} \gauche \draw[brightpink] (0,1.6) -- (0.5,1.1) -- (1,1.6) -- (1.4,1.6) -- (1.4,2);  \end{tikzpicture}
 \textrm{ and }
\pdiamond_7 = 
\begin{tikzpicture} \gauche \draw[brightpink] (0,1.6) -- (1,0.6) -- (1.4,0.6) -- (1.4,1.6) -- (1.4,2);  \end{tikzpicture}
 \enspace , 
\end{align*}
and two paths on the gauche lattice that give the letter $g(7) = 2$, namely
\begin{align*}
\pbullet_2 = 
\begin{tikzpicture} \gauche \draw[brightpink] (0,1.6) -- (0.5,1.1) -- (1,1.6) -- (1.4,1.6) -- (2,1.6);  \end{tikzpicture}
 \textrm{ and }
\pdiamond_2 = 
\begin{tikzpicture} \gauche \draw[brightpink] (0,1.6) -- (1,0.6) -- (1.4,0.6) -- (1.4,1.6) -- (2,1.6);  \end{tikzpicture}
 \enspace . 
\end{align*}
Let $\ell = |v|$. Choose paths $P_{k-\ell+1},...,P_k$ on the gauche lattice so that the resulting word is $v$.  Let $\alpha,\beta,\gamma$ be the number of special paths among $P_{k-\ell+1},...,P_k$ for the 12, 13, 23 vertices respectively.  Note that $\gamma = |v| - |v|_1 - |v|_5$ since $v$ is a word in the alphabet $\{1,2,3,4,5\}$.  We have 
\begin{align*}
&G^{(k)}_{(u,\pbullet_2,P_{k-\ell+1},...,P_k)}(x,y) \\
&= \left [ R^{(k-\ell-1)}_{12}(y/xt^\alpha) L^{(k-\ell-1)}_{13}(xt^\beta) L^{(k-\ell-1)}_{23}(yt^{\gamma+1}) \right ]_u \cdot yt^\gamma \cdot y/xt^\alpha \cdot G^{(\ell)}_{(P_{k-\ell+1},...,P_k)}(x,y) \\
&= yt^\gamma \cdot \left [ R^{(k-\ell-1)}_{12}(y/xt^\alpha) L^{(k-\ell-1)}_{13}(xt^\beta) L^{(k-\ell-1)}_{23}(yt^{\gamma+1}) \right ]_u \cdot y/xt^\alpha \cdot G^{(\ell)}_{(P_{k-\ell+1},...,P_k)}(x,y) \\
&= yt^\gamma \cdot G^{(k)}_{(u,\pbullet_7,P_{k-\ell+1},...,P_k)}(x,y)
\end{align*}
and
\begin{align*}
&G^{(k)}_{(u,\pdiamond_2,P_{k-\ell+1},...,P_k)}(x,y) \\
&= \left [ R^{(k-\ell-1)}_{12}(y/xt^{\alpha+1}) L^{(k-\ell-1)}_{13}(xt^{\beta+1}) L^{(k-\ell-1)}_{23}(yt^{\gamma+1}) \right ]_u \cdot (1-y/xt^\alpha) \cdot yt^\gamma \cdot G^{(\ell)}_{(P_{k-\ell+1},...,P_k)}(x,y) \\
&= yt^\gamma \cdot \left [ R^{(k-\ell-1)}_{12}(y/xt^{\alpha+1}) L^{(k-\ell-1)}_{13}(xt^{\beta+1}) L^{(k-\ell-1)}_{23}(yt^{\gamma+1}) \right ]_u \cdot (1-y/xt^\alpha) \cdot G^{(\ell)}_{(P_{k-\ell+1},...,P_k)}(x,y) \\
&= yt^\gamma \cdot G^{(k)}_{(u,\pdiamond_7,P_{k-\ell+1},...,P_k)}(x,y).
\end{align*}
\end{proof}

\indent The following corollary follows from Lemmas \ref{YBE0}-\ref{YBE2}.

\begin{cor} To prove Equation \eqref{YBE(matrices)}, it is enough to show that 
\[ G_w(x,y) = D_w(x,y)\]
for all words $w$ in the alphabet $\{1,2,3,4,5\}$. \end{cor}

\noindent Therefore, from now on, we will assume $w$, $u$, and $v$ are words in the alphabet $\{1,2,3,4,5\}$.  

\subsection{Reducing to the case where $w$ is a word in the alphabet $\{1,2\}$} 

\indent First we show that 3's, 4's, and 5's can be removed from the right.

\begin{lem} \label{YBE4} $M_{(w,3)}(x,y) = y^2 M_w(xt,yt)$ \end{lem}

\begin{proof}
Note that there is only one path on the gauche lattice that gives the letter $3$, namely
\begin{align*}
\begin{tikzpicture} \gauche \draw[brightpink] (0,0.4) -- (1,1.4) -- (2,1.4); \draw[brightpink] (0,1.6) -- (1,0.6) -- (2,0.6);  \end{tikzpicture} \enspace .
\end{align*}
We have
\begin{align*}
G_{(w,3)}(x,y) &= \left [ R_{12}(y/x) L_{13}(xt) L_{23}(yt) \right ]_w \cdot y/x \cdot x \cdot y \\
&= y^2 \cdot \left [ R_{12}(yt/xt) L_{13}(xt) L_{23}(yt) \right ]_w = y^2 \cdot G_w(xt,yt).
\end{align*}
Note that there is only one path on the droite lattice that gives the letter $3$, namely
\begin{align*}
\begin{tikzpicture} \droite \draw[brightpink] (2,0.4) -- (1,1.4) -- (0,1.4); \draw[brightpink] (2,1.6) -- (1,0.6) -- (0,0.6);  \end{tikzpicture} 
 \enspace . 
\end{align*}
We have
\begin{align*}
D_{(w,3)}(x,y) &= \left [ L_{23}(yt) L_{13}(xt) R_{12}(y/x) \right ]_w \cdot y \cdot x \cdot y/x \\
&= y^2 \cdot \left [ L_{23}(yt) L_{13}(xt) R_{12}(yt/xt) \right ]_w = y^2 \cdot D_w(xt,yt).
\end{align*}
\end{proof}

\begin{lem} $M_{(w,4)}(x,y) = y M_w(xt,yt)$ \end{lem}

\begin{proof}
	\begin{align*}
	\begin{tikzpicture} \gauche \draw[brightpink] (0,0.4) -- (0.5,0.9) -- (1,0.4) -- (1.6,0.4) -- (1.6,1.4) -- (2,1.4);  \end{tikzpicture}
	\end{align*}
	\begin{align*}
	G_{(w,4)}(x,y) &= \left [ R_{12}(y/x) L_{13}(xt) L_{23}(yt) \right ]_w \cdot y \\
	&= y \cdot \left [ R_{12}(yt/xt) L_{13}(xt) L_{23}(yt) \right ]_w = y \cdot G_w(xt,yt)
	\end{align*}
	\begin{align*}
	\begin{tikzpicture} \droite \draw[brightpink] (0,0.4) -- (0.6,0.4) -- (0.6,1.4) -- (1,1.4) -- (1.5,0.9) -- (2,1.4);  \end{tikzpicture}
	\end{align*}
	\begin{align*}
	D_{(w,4)}(x,y) &= \left [ L_{23}(yt) L_{13}(xt) R_{12}(x,y) \right ]_w \cdot x \cdot y/x \\
	&= y \cdot \left [ L_{23}(yt) L_{13}(xt) R_{12}(xt,yt) \right ]_w = y \cdot D_w(xt,yt)
	\end{align*}
\end{proof}

\begin{lem} $M_{(w,5)}(x,y) = (x-y) M_w(xt,y)$ \end{lem}

\begin{proof}
\begin{align*}
\begin{tikzpicture} \gauche \draw[brightpink] (0,1.4) -- (0.5,0.9) -- (1,0.4) -- (1.6,0.4) -- (2,0.4);  \end{tikzpicture} 
\end{align*}
\begin{align*}
G_{(w,5)}(x,y) &= \left [ R_{12}(y/xt) L_{13}(xt) L_{23}(y) \right ]_w (1-y/x)x \\
&= (x-y) \left [ R_{12}(y/xt) L_{13}(xt) L_{23}(y) \right ]_w = (x-y) G_w(xt,y)
\end{align*}
\begin{align*}
\begin{tikzpicture} \droite \draw[brightpink] (0,1.4) -- (0.6,1.4) -- (1,1.4) -- (2,0.4);  \end{tikzpicture}
\end{align*}
\begin{align*}
D_{(w,5)}(x,y) &= \left [ L_{23}(y) L_{13}(xt) R_{12}(y/xt) \right ]_w x(1-y/x) \\
&= (x-y) \left [ L_{23}(y) L_{13}(xt) R_{12}(y/xt) \right ]_w = (x-y) D_w(xt,y)
\end{align*}
\end{proof}

\indent Next we show that 3's, 4's, and 5's can be moved to the right.

\begin{lem} \label{YBE7} $M_{(u,3,1,v)}(x,y) = M_{(u,1,3,v)}(x,y)$ \end{lem}

\begin{proof}
We do the case $M = G$ first.  Let $\ell = |v|$. Choose paths $P_{k-\ell+1},...,P_k$ on the gauche lattice so that the resulting word is $v$.  Let $\alpha,\beta,\gamma$ be the number of special paths among $P_{k-\ell+1},...,P_k$ for the 12, 13, 23 vertices respectively.  Note that there is only one tuple of paths on the gauche lattice so that the resulting word is $(3,1)$, namely
\begin{align*}
(\gbullet,\bbullet) = 
\begin{tikzpicture} \gauche \draw[brown] (0,0.4) -- (0.5,0.9) -- (1,0.4) -- (1.6,0.4) -- (2,0.4); \draw[green] (0,0.6) -- (1,1.6) -- (2,1.6); \draw[green] (0,1.6) -- (1,0.6) -- (2,0.6);  \end{tikzpicture}
 \enspace. 
\end{align*}
We have
\begin{align*}
&G_{(u,\gbullet,\bbullet,P_{k-\ell+1},...,P_k)}(x,y) \\
&= \left [ R_{12}(y/xt^\alpha) L_{13}(xt^{\beta+2}) L_{23}(yt^{\gamma+1}) \right ]_u \cdot \textcolor{green}{\frac{y}{xt^\alpha}xt^{\beta+1}yt^\gamma} \textcolor{brown}{xt^\beta} \cdot G_{(P_{k-\ell+1},...,P_k)}(x,y) \\
&= \left [ R_{12}(y/xt^\alpha) L_{13}(xt^{\beta+2}) L_{23}(yt^{\gamma+1}) \right ]_u \cdot \textcolor{brown}{xt^{\beta+1}} \textcolor{green}{\frac{y}{xt^\alpha}xt^\beta yt^\gamma} \cdot G_{(P_{k-\ell+1},...,P_k)}(x,y) \\
&= G_{(u,\bbullet,\gbullet,P_{k-\ell+1},...,P_k)}(x,y).
\end{align*}
The case $M = D$ is similar, except there are two tuples of paths to consider.
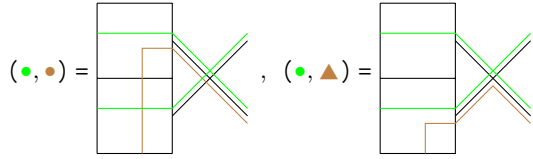
\begin{align*}
(\gbullet,\bbullet) = 
\begin{tikzpicture} \droite \draw[brown] (0.6,0) -- (0.6,1.4) -- (1,1.4) -- (2,0.4); \draw[green] (2,0.6) -- (1,1.6) -- (0,1.6); \draw[green] (2,1.6) -- (1,0.6) -- (0,0.6);  \end{tikzpicture}
 \enspace, \enspace
(\gbullet,\bdiamond) = 
\begin{tikzpicture} \droite \draw[brown] (0.6,0) -- (0.6,0.4) -- (1,0.4) -- (1.5,0.9) -- (2,0.4); \draw[green] (2,0.6) -- (1,1.6) -- (0,1.6); \draw[green] (2,1.6) -- (1,0.6) -- (0,0.6);  \end{tikzpicture} 
\end{align*}
\begin{align*}
&D_{(u,\gbullet,\bbullet,P_{k-\ell+1},...,P_k)}(x,y) \\
&= \left [ L_{23}(yt^{\gamma+2}) L_{13}(xt^{\beta+2}) R_{12}(y/xt^{\alpha+1}) \right ]_u \cdot \textcolor{green}{yt^{\gamma+1}xt^{\beta+1}\frac{y}{xt^{\alpha+1}}} \textcolor{brown}{xt^\beta\left(1-\frac{y}{xt^\alpha}\right)} \cdot D_{(P_{k-\ell+1},...,P_k)}(x,y) \\
&= D_{(P_{k-\ell+1},...,P_k)}(x,y) \cdot \textcolor{brown}{xt^{\beta+1}\left(1-\frac{y}{xt^\alpha}\right)} \textcolor{green}{yt^\gamma xt^\beta \frac{y}{xt^\alpha}} \cdot \left [ L_{23}(yt^{\gamma+2}) L_{13}(xt^{\beta+2}) R_{12}(y/xt^{\alpha+1}) \right ]_u \\
&= D_{(u,\bbullet,\gbullet,P_{k-\ell+1},...,P_k)}(x,y), \\
&D_{(u,\gbullet,\bdiamond,P_{k-\ell+1},...,P_k)}(x,y) \\
&= D_{(P_{k-\ell+1},...,P_k)}(x,y) \cdot \textcolor{green}{yt^{\gamma+1}xt^\beta\frac{y}{xt^\alpha}} \textcolor{brown}{yt^\gamma} \cdot \left [ L_{23}(yt^{\gamma+2}) L_{13}(xt^{\beta+1}) R_{12}(y/xt^\alpha) \right ]_u \\
&= D_{(P_{k-\ell+1},...,P_k)}(x,y) \cdot \textcolor{brown}{yt^{\gamma+1}} \textcolor{green}{yt^\gamma xt^\beta\frac{y}{xt^\alpha}} \cdot \left [ L_{23}(yt^{\gamma+2}) L_{13}(xt^{\beta+1}) R_{12}(y/xt^\alpha) \right ]_u \\
&= D_{(u,\bdiamond,\gbullet,P_{k-\ell+1},...,P_k)}(x,y)
\end{align*}
\end{proof}

\indent Lemmas \ref{YBE8}-\ref{YBE12} are proven in a similar manner as Lemma \ref{YBE7}.   

\begin{lem} \label{YBE8} $M_{(u,3,2,v)}(x,y) = M_{(u,2,3,v)}(x,y)$ \end{lem}

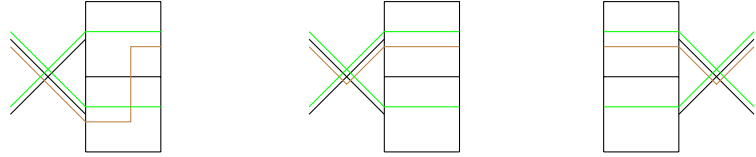
\begin{proof}

\noindent

\noindent \begin{tabular}{ccr}
\begin{tikzpicture} \gauche \draw[brown] (0,1.4) -- (1,0.4) -- (1.6,0.4) -- (1.6,1.4) -- (2,1.4); \draw[green] (0,0.6) -- (1,1.6) -- (2,1.6); \draw[green] (0,1.6) -- (1,0.6) -- (2,0.6);  \end{tikzpicture}
& \begin{tikzpicture} \gauche \draw[brown] (0,1.4) -- (0.5,0.9) -- (1,1.4) -- (2,1.4); \draw[green] (0,0.6) -- (1,1.6) -- (2,1.6); \draw[green] (0,1.6) -- (1,0.6) -- (2,0.6);  \end{tikzpicture}
& \begin{tikzpicture} \droite \draw[brown] (0,1.4) -- (1,1.4) -- (1.5,0.9) -- (2,1.4); \draw[green] (2,0.6) -- (1,1.6) -- (0,1.6); \draw[green] (2,1.6) -- (1,0.6) -- (0,0.6);  \end{tikzpicture} \\ \\
$\textcolor{green}{\frac{y}{xt^{\alpha+1}}xt^{\beta+1}yt^{\gamma+1}} \textcolor{brown}{\left(1-\frac{y}{xt^\alpha} \right)yt^\gamma}$
& $\textcolor{green}{\frac{y}{xt^\alpha}xt^\beta yt^{\gamma+1}} \textcolor{brown}{\frac{y}{xt^\alpha}yt^\gamma}$  
& $\textcolor{green}{yt^\gamma xt^{\beta+1} \frac{y}{xt^\alpha}} \textcolor{brown}{xt^\beta \frac{y}{xt^\alpha}}$ \\
= $\textcolor{brown}{\left(1-\frac{y}{xt^\alpha} \right)yt^{\gamma+1}} \textcolor{green}{\frac{y}{xt^\alpha}xt^\beta yt^\gamma}$ 
& $=\textcolor{brown}{\frac{y}{xt^\alpha}yt^{\gamma+1}} \textcolor{green}{\frac{y}{xt^\alpha}xt^\beta yt^\gamma}$
& $=\textcolor{brown}{xt^{\beta+1} \frac{y}{xt^\alpha}} \textcolor{green}{yt^\gamma xt^\beta \frac{y}{xt^\alpha}}$
\end{tabular} 

\end{proof}

\begin{lem} $M_{(u,4,1,v)}(x,y) = \frac{1}{t} M_{(u,1,4,v)}(x,y)$ \end{lem}

\begin{proof}

\noindent

\noindent \begin{tabular}{rrr}
\begin{tikzpicture} \gauche \draw[brown] (0,0.4) -- (0.5,0.9) -- (1,0.4) -- (1.6,0.4) -- (2,0.4); \draw[green] (0,0.6) -- (0.5,1.1) -- (1,0.6) -- (1.4,0.6) -- (1.4,1.6) -- (2,1.6);  \end{tikzpicture}
& \begin{tikzpicture} \droite \draw[brown] (0,0.4) -- (0.6,0.4) -- (0.6,1.4) -- (1,1.4) -- (2,0.4); \draw[green] (0,0.6) -- (0.4,0.6) -- (0.4,1.6) -- (1,1.6) -- (1.5,1.1) -- (2,1.6);  \end{tikzpicture}
& \begin{tikzpicture} \droite \draw[brown] (0,0.4) -- (1,0.4) -- (1.5,0.9) -- (2,0.4); \draw[green] (0,0.6) -- (0.4,0.6) -- (0.4,1.6) -- (1,1.6) -- (1.5,1.1) -- (2,1.6);  \end{tikzpicture} \\ \\
$\textcolor{green}{yt^\gamma} \textcolor{brown}{xt^\beta}$ 
& $\textcolor{green}{xt^{\beta+1} \frac{y}{xt^{\alpha+1}}} \textcolor{brown}{xt^\beta \left ( 1 - \frac{y}{xt^\alpha} \right )}$
& $\textcolor{green}{xt^\beta \frac{y}{xt^\alpha}} \textcolor{brown}{yt^\gamma}$\\
$=\frac{1}{t} \textcolor{brown}{xt^{\beta+1}} \textcolor{green}{yt^\gamma}$ 
& $=\frac{1}{t} \textcolor{brown}{xt^{\beta+1}\left ( 1 - \frac{y}{xt^\alpha} \right )} \textcolor{green}{xt^\beta \frac{y}{xt^\alpha}}$
& $=\frac{1}{t} \textcolor{brown}{yt^{\gamma+1}} \textcolor{green}{xt^\beta \frac{y}{xt^\alpha}}$
\end{tabular}
\end{proof}

\begin{lem} $M_{(u,4,2,v)}(x,y) = M_{(u,2,4,v)}(x,y)$ \end{lem}

\begin{proof}

\noindent

\noindent \begin{tabular}{rrr}
\begin{tikzpicture} \gauche \draw[brown] (0,1.4) -- (0.5,0.9) -- (1,1.4) -- (2,1.4); \draw[green] (0,0.6) -- (0.5,1.1) -- (1,0.6) -- (1.4,0.6) -- (1.4,1.6) -- (2,1.6);  \end{tikzpicture}
& \begin{tikzpicture} \gauche \draw[brown] (0,1.4) -- (1,0.4) -- (1.6,0.4) -- (1.6,1.4) -- (2,1.4); \draw[green] (0,0.6) -- (0.5,1.1) -- (1,0.6) -- (1.4,0.6) -- (1.4,1.6) -- (2,1.6);  \end{tikzpicture}
& \begin{tikzpicture} \droite \draw[brown] (0,1.4) -- (1,1.4) -- (1.5,0.9) -- (2,1.4); \draw[green] (0,0.6) -- (0.4,0.6) -- (0.4,1.6) -- (1,1.6) -- (1.5,1.1) -- (2,1.6);  \end{tikzpicture} \\ \\
$\textcolor{green}{yt^{\gamma+1}} \textcolor{brown}{\frac{y}{xt^\alpha}yt^\gamma}$ 
& $\textcolor{green}{yt^{\gamma+1}} \textcolor{brown}{\left ( 1 - \frac{y}{xt^\alpha} \right ) yt^\gamma}$ 
& $\textcolor{green}{xt^{\beta+1} \frac{y}{xt^\alpha}} \textcolor{brown}{xt^\beta \frac{y}{xt^\alpha}}$ \\
$=\textcolor{brown}{\frac{y}{xt^\alpha} yt^{\gamma+1}} \textcolor{green}{yt^\gamma}$ 
& $=\textcolor{brown}{\left ( 1 - \frac{y}{xt^\alpha} \right ) yt^{\gamma+1}} \textcolor{green}{yt^\gamma}$
& $=\textcolor{brown}{xt^{\beta+1} \frac{y}{xt^\alpha}} \textcolor{green}{xt^\beta \frac{y}{xt^\alpha}}$
\end{tabular}
\end{proof}

\begin{lem} $M_{(u,5,1,v)}(x,y) = M_{(u,1,5,v)}(x,y)$ \end{lem}

\begin{proof}

\noindent

\noindent \begin{tabular}{rrr}
\begin{tikzpicture} \gauche \draw[brown] (0,0.4) -- (0.5,0.9) -- (1,0.4) -- (1.6,0.4) -- (2,0.4); \draw[green] (0,1.6) -- (1,0.6) -- (2,0.6);  \end{tikzpicture}
& \begin{tikzpicture} \droite \draw[brown] (0,0.4) -- (0.6,0.4) -- (0.6,1.4) -- (1,1.4) -- (2,0.4); \draw[green] (0,1.6) -- (1,1.6) -- (2,0.6);  \end{tikzpicture}
& \begin{tikzpicture} \droite \draw[brown] (0,0.4) -- (1,0.4) -- (1.5,0.9) -- (2,0.4); \draw[green] (0,1.6) -- (1,1.6) -- (2,0.6);  \end{tikzpicture} \\ \\
$\textcolor{green}{\left ( 1 - \frac{y}{xt^\alpha} \right ) xt^{\beta+1}} \textcolor{brown}{xt^\beta}$
& $\textcolor{green}{xt^{\beta+1} \left ( 1 - \frac{y}{xt^{\alpha+1}} \right )} \textcolor{brown}{xt^\beta \left ( 1 - \frac{y}{xt^\alpha} \right )}$
& $\textcolor{green}{xt^\beta \left ( 1 - \frac{y}{xt^\alpha} \right )} \textcolor{brown}{yt^\gamma}$ \\
$=\textcolor{brown}{xt^{\beta+1}} \textcolor{green}{\left ( 1 - \frac{y}{xt^\alpha} \right ) xt^\beta}$
& $=\textcolor{brown}{xt^{\beta+1}\left ( 1 - \frac{y}{xt^{\alpha+1}} \right )} \textcolor{green}{xt^\beta \left ( 1 - \frac{y}{xt^\alpha} \right )}$
& $=\textcolor{brown}{yt^\gamma} \textcolor{green}{xt^\beta \left ( 1 - \frac{y}{xt^\alpha} \right )}$
\end{tabular}
\end{proof}

\begin{lem} \label{YBE12} $M_{(u,5,2,v)}(x,y) = t M_{(u,2,5,v)}(x,y)$ \end{lem}

\begin{proof}

\noindent

\noindent \begin{tabular}{rrr}
\begin{tikzpicture} \gauche \draw[brown] (0,1.4) -- (0.5,0.9) -- (1,1.4) -- (2,1.4); \draw[green] (0,1.6) -- (1,0.6) -- (2,0.6);  \end{tikzpicture}
& \begin{tikzpicture} \gauche \draw[brown] (0,1.4) -- (1,0.4) -- (1.6,0.4) -- (1.6,1.4) -- (2,1.4); \draw[green] (0,1.6) -- (1,0.6) -- (2,0.6);  \end{tikzpicture}
& \begin{tikzpicture} \droite \draw[brown] (0,1.4) -- (1,1.4) -- (1.5,0.9) -- (2,1.4); \draw[green] (0,1.6) -- (1,1.6) -- (2,0.6);  \end{tikzpicture} \\ \\
$\textcolor{green}{\left ( 1 - \frac{y}{xt^\alpha} \right ) xt^\beta} \textcolor{brown}{\frac{y}{xt^\alpha} yt^\gamma}$
& $\textcolor{green}{\left ( 1 - \frac{y}{xt^{\alpha+1}} \right ) xt^{\beta+1}} \textcolor{brown}{\left ( 1 - \frac{y}{xt^\alpha} \right ) yt^\gamma}$
& $\textcolor{green}{xt^{\beta+1} \left ( 1 - \frac{y}{xt^\alpha} \right )} \textcolor{brown}{xt^\beta \frac{y}{xt^\alpha}}$ \\
$=t \textcolor{brown}{\frac{y}{xt^{\alpha+1}} yt^\gamma} \textcolor{green}{\left ( 1 - \frac{y}{xt^\alpha} \right ) xt^\beta}$
& $=t \textcolor{brown}{\left ( 1 - \frac{y}{xt^{\alpha+1}} \right ) yt^\gamma} \textcolor{green}{\left ( 1 - \frac{y}{xt^\alpha} \right ) xt^\beta}$
& $=t \textcolor{brown}{xt^{\beta+1} \frac{y}{xt^{\alpha+1}}} \textcolor{green}{xt^\beta \left ( 1 - \frac{y}{xt^\alpha} \right )}$
\end{tabular}
\end{proof}

\indent The following corollary follows from Lemmas \ref{YBE4}-\ref{YBE12}.

\begin{cor} To prove Equation \eqref{YBE(matrices)}, it is enough to show that 
\[ G_w(x,y) = D_w(x,y) \]
for all words $w$ in the alphabet $\{1,2\}$. \end{cor}

\noindent Therefore, from now on, we will assume $w$, $u$, and $v$ are words in the alphabet $\{1,2\}$.

\subsection{Reducing to the case where $w = 1^k$ or $w = 2^k$}

\begin{lem} \label{YBE14} $M_{(w,2,1)}(x,y) = (x-y)y M_w(xt^2,yt) + y^2 M_w(xt,yt)$ \end{lem}

\begin{proof} 
Note that there are two tuples of paths on the gauche lattice so that the resulting word is $(2,1)$, namely 
\begin{align*}
(\bbullet,\gbullet) = 
\begin{tikzpicture} \gauche \draw[brown] (0,1.6) -- (1,0.6) -- (1.6,0.6) -- (1.6,1.6) -- (2,1.6); \draw[green] (0,0.4) -- (0.5,0.9) -- (1,0.4) -- (2,0.4);  \end{tikzpicture}
\textrm{ and }
(\bdiamond,\gbullet) =
\begin{tikzpicture} \gauche \draw[brown] (0,1.6) -- (0.5,1.1) -- (1,1.6) -- (2,1.6); \draw[green] (0,0.4) -- (0.5,0.9) -- (1,0.4) -- (2,0.4);  \end{tikzpicture} \enspace.
\end{align*}
We have
\begin{align*}
G_{(w,2,1)}(x,y) &= G_{(w,\bbullet,\gbullet)}(x,y) + G_{(w,\bdiamond,\gbullet)}(x,y) \\
&= \left[R_{12}(y/xt)L_{13}(xt^2)L_{23}(yt)\right]_w\textcolor{brown}{(1-y/x)y}\textcolor{green}{x} + \left[R_{12}(y/x)L_{13}(xt)L_{23}(yt)\right]_w\textcolor{brown}{(y/x)y}\textcolor{green}{x} \\
&= (x-y)y \left[R_{12}(yt/xt^2)L_{13}(xt^2)L_{23}(yt)\right]_w + y^2 \left[R_{12}(yt/xt)L_{13}(xt)L_{23}(yt)\right]_w \\
&= (x-y)y G_w(xt^2,yt) + y^2 G_w(xt,yt).
\end{align*}
Note that there are two tuples of paths on the droite lattice so that the resulting word is $(2,1)$, namely
\begin{align*}
(\bbullet,\gbullet) = 
\begin{tikzpicture} \droite \draw[brown] (0,1.6) -- (1,1.6) -- (1.5,1.1) -- (2,1.6); \draw[green] (0,0.4) -- (0.6,0.4) -- (0.6,1.4) -- (1,1.4) -- (2,0.4);  \end{tikzpicture}
 \textrm{ and }
(\bbullet,\gdiamond) = 
\begin{tikzpicture} \droite \draw[brown] (0,1.6) -- (1,1.6) -- (1.5,1.1) -- (2,1.6); \draw[green] (0,0.4) -- (1,0.4) -- (1.5,0.9) -- (2,0.4);  \end{tikzpicture} \enspace.
\end{align*}
We have
\begin{align*}
D_{(w,2,1)}(x,y) &= D_{(w,\bbullet,\gbullet)}(x,y) + D_{(w,\bbullet,\gdiamond)}(x,y) \\
&= \left[L_{23}(yt)L_{13}(xt^2)R_{12}(y/xt)\right]_w\textcolor{brown}{xt (y/xt)}\textcolor{green}{x(1-y/x)} + \left[L_{23}(yt)L_{13}(xt)R_{12}(y/x)\right]_w\textcolor{brown}{x(y/x)}\textcolor{green}{y} \\
&= (x-y)y\left[L_{23}(yt)L_{13}(xt^2)R_{12}(yt/xt^2)\right]_w + y^2\left[L_{23}(yt)L_{13}(xt)R_{12}(yt/xt)\right]_w \\
&= (x-y)y D_w(xt^2,yt) + y^2 D_w(xt,yt).
\end{align*}
\end{proof}

\begin{lem} \label{YBE15} $M_{(u,1,2,v)}(x,y) = t M_{(u,2,1,v)}(x,y) + (1-t) M_{(u,3,v)}(x,y)$ \end{lem}

\begin{proof}
We do the case $M = D$ first.  Let $\ell = |v|$. Choose paths $P_{k-\ell+1},...,P_k$ on the gauche lattice so that the resulting word is $v$.  Let $\alpha,\beta,\gamma$ be the number of special paths among $P_{k-\ell+1},...,P_k$ for the 12, 13, 23 vertices respectively.  Note that there are two tuples of paths on the droite lattice so that the resulting word is $(2,1)$, namely
\begin{align*}
(\bbullet,\gbullet) = 
\begin{tikzpicture} \droite \draw[brown] (0,1.6) -- (1,1.6) -- (1.5,1.1) -- (2,1.6); \draw[green] (0,0.4) -- (0.6,0.4) -- (0.6,1.4) -- (1,1.4) -- (2,0.4);  \end{tikzpicture}
 \textrm{ and }
(\bbullet,\gdiamond) =
\begin{tikzpicture} \droite \draw[brown] (0,1.6) -- (1,1.6) -- (1.5,1.1) -- (2,1.6); \draw[green] (0,0.4) -- (1,0.4) -- (1.5,0.9) -- (2,0.4);  \end{tikzpicture} \enspace.
\end{align*}
Also note that there is only one path on the droite lattice that gives the letter 3, namely
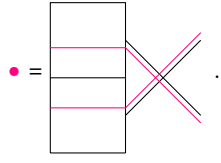
\begin{align*}
\pbullet = 
\begin{tikzpicture} \droite \draw[brightpink] (2,0.4) -- (1,1.4) -- (0,1.4); \draw[brightpink] (2,1.6) -- (1,0.6) -- (0,0.6);  \end{tikzpicture}  \enspace.
\end{align*}
We observe
\begin{align*}
&D_{(u,\bbullet,\gdiamond,P_{k-\ell+1},...,P_k)}(x,y) \\
&= \left [ L_{23}(yt^{\gamma+1}) L_{13}(xt^{\beta+1}) R_{12}(y/xt^\alpha) \right ]_u \cdot \textcolor{brown}{xt^\beta(y/xt^\alpha)} \textcolor{green}{yt^\gamma} \cdot D_{(P_{k-\ell+1},...,P_k)}(x,y) \\
&= \left [ L_{23}(yt^{\gamma+1}) L_{13}(xt^{\beta+1}) R_{12}(y/xt^\alpha) \right ]_u \cdot \textcolor{green}{yt^\gamma} \textcolor{brown}{xt^\beta(y/xt^\alpha)} \cdot D_{(P_{k-\ell+1},...,P_k)}(x,y) = D_{(u,\gdiamond,\bbullet,P_{k-\ell+1},...,P_k)}(x,y) \\
&= \left [ L_{23}(yt^{\gamma+1}) L_{13}(xt^{\beta+1}) R_{13}(y/xt^\alpha) \right ]_u \cdot \textcolor{brightpink}{yt^\gamma xt^\beta(y/xt^\alpha)} \cdot D_{(P_{k-\ell+1},...,P_k)}(x,y) = D_{(u,\pbullet,P_{k-\ell+1},...,P_k)}(x,y)
\end{align*}
and
\begin{align*}
&D_{(u,\gbullet,\bbullet,P_{k-\ell+1},...,P_k)}(x,y) \\
&= \left [ L_{23}(yt^{\gamma+1}) L_{13}(xt^{\beta+2}) R_{12}(y/xt^{\alpha+1}) \right ]_u \cdot \textcolor{green}{xt^{\beta+1}(1-y/xt^\alpha)} \textcolor{brown}{xt^\beta(y/xt^\alpha)} \cdot D_{(P_{k-\ell+1},...,P_k)}(x,y) \\
&= t \left [ L_{23}(yt^{\gamma+1}) L_{13}(xt^{\beta+2}) R_{12}(y/xt^{\alpha+1}) \right ]_u \cdot \textcolor{brown}{xt^{\beta+1}(y/xt^{\alpha+1})} \textcolor{green}{xt^\beta(1-y/xt^\alpha)} \cdot D_{(P_{k-\ell+1},...,P_k)}(x,y) \\
&= tD_{(u,\bbullet,\gbullet,P_{k-\ell+1},...,P_k)}(x,y).
\end{align*}
We have 
\begin{align*}
&D_{(u,1,2,P_{k-\ell+1},...,P_k)}(x,y) \\
&= D_{(u,\gbullet,\bbullet,P_{k-\ell+1},...,P_k)}(x,y) + D_{(u,\gdiamond,\bbullet,P_{k-\ell+1},...,P_k)}(x,y) \\
&= tD_{(u,\bbullet,\gbullet,P_{k-\ell+1},...,P_k)}(x,y) + D_{(u,\pbullet,P_{k-\ell+1},...,P_k)}(x,y) \\ 
&= t [D_{(u,\bbullet,\gbullet,P_{k-\ell+1},...,P_k)}(x,y) + D_{(u,\bbullet,\gdiamond,P_{k-\ell+1},...,P_k)}(x,y)] -t D_{(u,\bbullet,\gdiamond,P_{k-\ell+1},...,P_k)}(x,y) + D_{(u,\pbullet,P_{k-\ell+1},...,P_k)}(x,y) \\
&= t [D_{(u,\bbullet,\gbullet,P_{k-\ell+1},...,P_k)}(x,y) + D_{(u,\bbullet,\gdiamond,P_{k-\ell+1},...,P_k)}(x,y)] + (1-t) D_{(u,\pbullet,P_{k-\ell+1},...,P_k)}(x,y) \\
&= t D_{(u,2,1,P_{k-\ell+1},...,P_k)}(x,y) + (1-t) D_{(u,3,P_{k-\ell+1},...,P_k)}(x,y).
\end{align*}

\indent We will now do the case $M = G$.  Let $\ell = |v|$. Choose paths $P_{k-\ell+1},...,P_k$ on the gauche lattice so that the resulting word is $v$.  Let $\alpha,\beta,\gamma$ be the number of special paths among $P_{k-\ell+1},...,P_k$ for the 12, 13, 23 vertices respectively.  Note that there are two tuples of paths on the gauche lattice so that the resulting word is $(2,1)$, namely
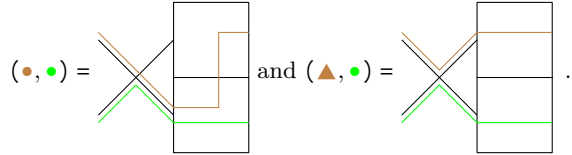
\begin{align*}
(\bbullet,\gbullet) = 
\begin{tikzpicture} \gauche \draw[brown] (0,1.6) -- (1,0.6) -- (1.6,0.6) -- (1.6,1.6) -- (2,1.6); \draw[green] (0,0.4) -- (0.5,0.9) -- (1,0.4) -- (2,0.4);  \end{tikzpicture}
 \textrm{ and }
(\bdiamond,\gbullet) = 
\begin{tikzpicture} \gauche \draw[brown] (0,1.6) -- (0.5,1.1) -- (1,1.6) -- (2,1.6); \draw[green] (0,0.4) -- (0.5,0.9) -- (1,0.4) -- (2,0.4);  \end{tikzpicture} \enspace. 
\end{align*}
Also note that there is only one path on the gauche lattice that gives the letter 3, namely
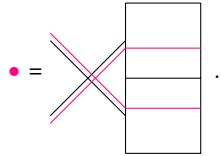
\begin{align*}
\pbullet = 
\begin{tikzpicture} \gauche \draw[brightpink] (0,0.4) -- (1,1.4) -- (2,1.4); \draw[brightpink] (0,1.6) -- (1,0.6) -- (2,0.6);  \end{tikzpicture} \enspace. 
\end{align*}
We observe
\begin{align*}
&G_{(u,\gbullet,\bdiamond,P_{k-\ell+1},...,P_k)}(x,y) \\
&= \left [ R_{12}(y/xt^\alpha) L_{13}(xt^{\beta+1}) L_{23}(yt^{\gamma+1}) \right ]_u \cdot \textcolor{green}{xt^\beta} \textcolor{brown}{(y/xt^\alpha)yt^\gamma} \cdot G_{(P_{k-\ell+1},...,P_k)}(x,y) \\
&= \left [ R_{12}(y/xt^\alpha) L_{13}(xt^{\beta+1}) L_{23}(yt^{\gamma+1}) \right ]_u \cdot \textcolor{brown}{(y/xt^\alpha)yt^\gamma} \textcolor{green}{xt^\beta} \cdot G_{(P_{k-\ell+1},...,P_k)}(x,y) = G_{(u,\bdiamond,\gbullet,P_{k-\ell+1},...,P_k)}(x,y) \\
&= \left [ R_{12}(y/xt^\alpha) L_{13}(xt^{\beta+1}) L_{23}(yt^{\gamma+1}) \right ]_u \cdot \textcolor{brightpink}{(y/xt^\alpha)xt^\beta yt^\gamma} \cdot G_{(P_1,...,P_k)}(x,y) = G_{(u,\pbullet,P_{k-\ell+1},...,P_k)}(x,y)
\end{align*}
and
\begin{align*}
&G_{(u,\gbullet,\bbullet,P_{k-\ell+1},...,P_k)}(x,y) \\
&= \left [ R_{12}(y/xt^{\alpha+1}) L_{13}(xt^{\beta+2}) L_{23}(yt^{\gamma+1}) \right ]_u \cdot \textcolor{green}{xt^{\beta+1}} \textcolor{brown}{(1-y/xt^\alpha)yt^\gamma} \cdot G_{(P_{k-\ell+1},...,P_k)}(x,y) \\
&= t \left [ R_{12}(y/xt^{\alpha+1}) L_{13}(xt^{\beta+2}) L_{23}(yt^{\gamma+1}) \right ]_u \cdot \textcolor{brown}{(1-y/xt^\alpha)yt^\gamma} \textcolor{green}{xt^\beta} \cdot G_{(P_{k-\ell+1},...,P_k)}(x,y) \\
&= tG_{(u,\bbullet,\gbullet,P_{k-\ell+1},...,P_k)}(x,y).
\end{align*}
We have 
\begin{align*}
&G_{(u,1,2,P_{k-\ell+1},...,P_k)}(x,y) \\
&= G_{(u,\gbullet,\bbullet,P_{k-\ell+1},...,P_k)}(x,y) + G_{(u,\gbullet,\bdiamond,P_{k-\ell+1},...,P_k)}(x,y) \\
&= tG_{(u,\bbullet,\gbullet,P_{k-\ell+1},...,P_k)}(x,y) + G_{(u,\pbullet,P_{k-\ell+1},...,P_k)}(x,y) \\ 
&= t [G_{(u,\bbullet,\gbullet,P_{k-\ell+1},...,P_k)}(x,y) + G_{(u,\bdiamond,\gbullet,P_{k-\ell+1},...,P_k)}(x,y)] - t G_{(u,\bdiamond,\gbullet,P_{k-\ell+1},...,P_k)}(x,y) + G_{(u,\pbullet,P_{k-\ell+1},...,P_k)}(x,y) \\
&= t [G_{(u,\bbullet,\gbullet,P_{k-\ell+1},...,P_k)}(x,y) + G_{(u,\bdiamond,\gbullet,P_{k-\ell+1},...,P_k)}(x,y)] + (1-t) G_{(u,\pbullet,P_{k-\ell+1},...,P_k)}(x,y) \\
&= t G_{(u,2,1,P_{k-\ell+1},...,P_k)}(x,y) + (1-t) G_{(u,3,P_{k-\ell+1},...,P_k)}(x,y).
\end{align*}
\end{proof}

\indent The following corollary follows from Lemmas \ref{YBE14} and \ref{YBE15}.

\begin{cor} To prove Equation \eqref{YBE(matrices)}, it is enough to show that 
\[ G_w(x,y) = D_w(x,y)\] 
for $w = 1^k$ and $w = 2^k$. \end{cor}

\subsection{Checking $w = 1^k$ and $w = 2^k$}

\indent We will now show that $G_w(x,y) = D_w(x,y)$ for $w = 1^k$ and $w = 2^k$ by computing them explicitly.

\begin{lem} $M_{1^k}(x,y) = x^{k} t^{\binom{k}{2}}$ \end{lem}

\begin{proof}
For the gauche lattice, there is only one path corresponding to the letter 1, namely
\begin{align*} 
 \begin{tikzpicture} \gauche \draw[brightpink] (0,0.4) -- (0.5,0.9) -- (1,0.4) -- (2,0.4);  \end{tikzpicture}  \enspace . \\
\end{align*}
Therefore
\begin{align*}
& G_{1^k}(x,y) = 
 \begin{tikzpicture} \gauche \draw[brightpink] (0,0.4) -- (0.5,0.9) -- (1,0.4) -- (2,0.4); \node[below] at (0.5,0.5) {$\textcolor{brightpink}{k}$};  \end{tikzpicture}  
= x^{2k}t^{k \choose 2}. \\
\end{align*}

For the droite lattice, there are two paths corresponding to the letter 1, which we call type I and type II.

\begin{center}
type I:
 \begin{tikzpicture} \droite \draw[brightpink] (0.6,0) -- (0.6,1.6) -- (1,1.6) -- (2,0.6);  \end{tikzpicture}  and type II:
 \begin{tikzpicture} \droite \draw[brightpink] (0.6,0) -- (0.6,0.4) -- (1,0.4) -- (1.5,0.9) -- (2,0.4);  \end{tikzpicture} 
\end{center}

\indent 

\noindent Let $\ell$ be the number of paths of type I, and suppose that the paths of type I
appear in colors $\{s_1,\ldots ,s_\ell\}$.  Then the paths of type I will contribute $x^\ell t^{\ell \choose 2} \prod_{i=0}^{\ell-1} \left(1-\frac{y}{xt^i}\right)$ and the paths of type II will contribute
$
y^{k-\ell} t^{{k-\ell \choose 2}+\sum_{i=1}^\ell s_i-{\ell+1\choose 2}}$.  Therefore
\begin{align*}
 D_{1^k}(x,y) &= \sum_{l=0}^k \sum_{\begin{subarray}{c}
S\subseteq[k]\\ 
|S|=\ell\end{subarray}}
x^\ell t^{\ell \choose 2} \prod_{i=0}^{\ell-1} \left(1-\frac{y}{xt^i}\right)\cdot y^{k-\ell} t^{{k-\ell \choose 2}+\sum_{i=1}^\ell s_i-{\ell+1\choose 2}} \\
&= \sum_{\ell=0}^k x^\ell t^{\ell \choose 2} \prod_{i=0}^{\ell-1} \left(1-\frac{y}{xt^i}\right) \cdot y^{k-\ell} t^{k-\ell \choose 2} \qbin{k}{\ell}{t}.
\end{align*}

We now prove by induction that $D_{1^k}(x,y) = x^k t^{k\choose 2}$.
For $n=0$ this is true. Now let $k>0$.
We use the fact that if $0<\ell<k$ we have 
$$
\qbin{k}{\ell}{t} =q^{k-\ell}\qbin{k-1}{\ell-1}{t} +\qbin{k-1}{\ell}{t}. 
$$
We have $D_{1^k}=A_k+B_k$
with
$$A_k=\sum_{\ell=1}^{k} x^\ell t^{\ell \choose 2} \prod_{i=0}^{\ell-1} \left(1-\frac{y}{xt^i}\right) \cdot y^{k-\ell} t^{n-\ell+1 \choose 2} \qbin{k-1}{\ell-1}{t}.$$
Now if $\ell>0$, 
$$x^\ell t^{\ell \choose 2}\prod_{i=0}^{\ell-1} \left(1-\frac{y}{xt^i}\right)=x^\ell t^{\ell \choose 2}\prod_{i=0}^{\ell-2} \left(1-\frac{y}{xt^i}\right)-
x^{\ell-1} t^{\ell-1 \choose 2}\prod_{i=0}^{\ell-2} \left(1-\frac{y}{xt^i}\right).
$$

So $A_k=A^{(1)}_k-A^{(2)}_k$ with
\begin{eqnarray*}
A^{(1)}_k&=&\sum_{\ell=1}^k x^\ell t^{\ell \choose 2}\prod_{i=0}^{\ell-2} \left(1-\frac{y}{xt^i}\right)\cdot y^{k-\ell} t^{k-\ell+1 \choose 2} \qbin{k-1}{\ell-1}{t}.
\end{eqnarray*}

One can easily check that $B_k=A^{(2)}_k$ and that $A^{(1)}_k=xt^{k-1}D_{1^{k-1}}$.
As $D_{1^{k}}=A^{(1)}_k-A^{(2)}_k+B_k$, we have $D_{1^{k}}=xt^{k-1}D_{1^{k-1}}$ for all $k>0$.
The proof is complete as $D_{1^{0}}=1$.
\end{proof}

\begin{lem} $M_{2^k}(x,y) = x^k y^k t^{k \choose 2}$ \end{lem}

\begin{proof}

The proof follows from the previous lemma as one can easily check that
$
D_{2^k}(x,y)=\frac{y^k}{x^k}G_{1^k}(x,y)$
and 
$G_{2^k}(x,y)=\frac{y^k}{x^k}D_{1^k}(x,y).$
We just have to rotate the left and right lattices by 180 degrees.
\end{proof}

\section{Proof of Lemmas \ref{lem:BoxComp}, \ref{lem:Comp} \label{lemma-proofs}}

Recall that for both Lemma \ref{lem:BoxComp} and Lemma \ref{lem:Comp}, we have a bijection $\Phi$ between sets of configurations and wish to prove that $\coinv(T)-\coinv(\Phi(T))$ is independent of the choice of $T$. For Lemma \ref{lem:BoxComp}, this bijection rotates 180 degrees and reverses the colors, whereas for Lemma \ref{lem:Comp}, this bijection is an extension of the classical bijection between SSYT of a shape and those of its complement shape.

In either case, we note that the set of path configurations of $\bm{\lambda}$ is connected under flipping corners. It thus suffices to show that if $T$ is a configuration of $\bm{\lambda}$ and $f(T)$ is a configuration of $\bm{\lambda}$ resulting in flipping a single corner of $T$, then 
\begin{equation}
    \coinv T - \coinv f(T) = \coinv \Phi(T) - \coinv \Phi(f(T)).
    \label{coinv-flip}
\end{equation}
Moreover, as coinversions are counted pairwise between colors, and in both cases $\Phi$ interchanges down-flippable corners with up-flippable corners, without loss of generality it suffices to prove (\ref{coinv-flip}) in the case when $\bm{\lambda} = (\lambda^{(1)}, \lambda^{(2)})$ and $f(T)$ results from $T$ by flipping a corner up (or down). We thus restrict out attention to this case in both proofs below.

\begin{proof}[Proof of Lemma \ref{lem:BoxComp}]
There are ten possible local configurations of paths where a corner of $T$ might be down-flippable. The table below shows each local configuration of $T$ and the corresponding corner in $\Phi(T)$.
\begin{center}
\resizebox{0.8\textwidth}{!}{
\begin{tabular}{|
|c|c|c|c|c|c|c|c|c|c|}
\hline
Case: 
& 1 & 2 & 3 & 4 & 5/6 & 7 & 8 & 9 & 10 \\
\hline \hline 
& & & & & & & & & \\
$T$  
&
\begin{tikzpicture}[baseline=(current bounding box.center)]
\draw (0,-0.5) rectangle ++(1,1);
\draw[blue] (0,0.1)--(1,0.1);
\draw[red] (0.6,-0.5)--(0.6,-0.1)--(1,-0.1);
\end{tikzpicture}
&
\begin{tikzpicture}[baseline=(current bounding box.center)]
\draw (0,-0.5) rectangle ++(1,1);
\draw[blue] (0.4,-0.5)--(0.4,0.1)--(1,0.1);
\draw[red] (0.6,-0.5)--(0.6,0.5);
\end{tikzpicture}
&
\begin{tikzpicture}[baseline=(current bounding box.center)]
\draw (0,-0.5) rectangle ++(1,1);
\draw[blue] (0.4,-0.5)--(0.4,0.1)--(1,0.1);
\draw[red] (0,-0.1)--(1,-0.1);
\end{tikzpicture}
&
\begin{tikzpicture}[baseline=(current bounding box.center)]
\draw (0,-0.5) rectangle ++(1,1);
\draw[blue] (0.4,-0.5)--(0.4,0.1)--(1,0.1);
\draw[red] (0,-0.1)--(0.6,-0.1)--(0.6,0.5);
\end{tikzpicture}
&
\begin{tikzpicture}[baseline=(current bounding box.center)]
\draw (0,-0.5) rectangle ++(1,1);
\draw[blue] (0.4,-0.5)--(0.4,0.1)--(1,0.1);
\draw[red] (0.6,-0.5)--(0.6,-0.1)--(1,-0.1);
\end{tikzpicture}
&
\begin{tikzpicture}[baseline=(current bounding box.center)]
\draw (0,-0.5) rectangle ++(1,1);
\draw[blue] (0,0.1)--(0.4,0.1)--(0.4,0.5);
\draw[red] (0.6,-0.5)--(0.6,-0.1)--(1,-0.1);
\end{tikzpicture}
&
\begin{tikzpicture}[baseline=(current bounding box.center)]
\draw (0,-0.5) rectangle ++(1,1);
\draw[blue] (0.4,-0.5)--(0.4,0.1)--(1,0.1);
\end{tikzpicture}
&
\begin{tikzpicture}[baseline=(current bounding box.center)]
\draw (0,-0.5) rectangle ++(1,1);
\draw[red] (0.6,-0.5)--(0.6,-0.1)--(1,-0.1);
\end{tikzpicture}
& 
\begin{tikzpicture}[baseline=(current bounding box.center)]
\draw (0,-0.5) rectangle ++(1,1);
\draw[red] (0.6,-0.5)--(0.6,-0.1)--(1,-0.1);
\draw[blue] (0.4,-0.5)--(0.4,0.5);
\end{tikzpicture}
 \\
& & & & & & & & &
 \\
 \hline  
 & & & & & & & & &
 \\
$\Phi(T)$ 
& 
\begin{tikzpicture}[baseline=(current bounding box.center)]
\draw (0,-0.5) rectangle ++(1,1);
\draw[blue] (0,0.1)--(0.4,0.1)--(0.4,0.5);
\draw[red] (0,-0.1)--(1,-0.1);
\end{tikzpicture}
& 
\begin{tikzpicture}[baseline=(current bounding box.center)]
\draw (0,-0.5) rectangle ++(1,1);
\draw[blue] (0.4,-0.5)--(0.4,0.5);
\draw[red] (0,-0.1)--(0.6,-0.1)--(0.6,0.5);
\end{tikzpicture}
& 
\begin{tikzpicture}[baseline=(current bounding box.center)]
\draw (0,-0.5) rectangle ++(1,1);
\draw[blue] (0,0.1)--(1,0.1);
\draw[red] (0,-0.1)--(0.6,-0.1)--(0.6,0.5);
\end{tikzpicture}
& 
\begin{tikzpicture}[baseline=(current bounding box.center)]
\draw (0,-0.5) rectangle ++(1,1);
\draw[blue] (0.4,-0.5)--(0.4,0.1)--(1,0.1);
\draw[red] (0,-0.1)--(0.6,-0.1)--(0.6,0.5);
\end{tikzpicture}
& 
\begin{tikzpicture}[baseline=(current bounding box.center)]
\draw (0,-0.5) rectangle ++(1,1);
\draw[blue] (0,0.1)--(0.4,0.1)--(0.4,0.5);
\draw[red] (0,-0.1)--(0.6,-0.1)--(0.6,0.5);
\end{tikzpicture}
& 
\begin{tikzpicture}[baseline=(current bounding box.center)]
\draw (0,-0.5) rectangle ++(1,1);
\draw[blue] (0,0.1)--(0.4,0.1)--(0.4,0.5);
\draw[red] (0.6,-0.5)--(0.6,-0.1)--(1,-0.1);
\end{tikzpicture} 
&
\begin{tikzpicture}[baseline=(current bounding box.center)]
\draw (0,-0.5) rectangle ++(1,1);
\draw[red] (0,-0.1)--(0.6,-0.1)--(0.6,0.5);
\end{tikzpicture}
&
\begin{tikzpicture}[baseline=(current bounding box.center)]
\draw (0,-0.5) rectangle ++(1,1);
\draw[blue] (0,0.1)--(0.4,0.1)--(0.4,0.5);
\end{tikzpicture}
& 
\begin{tikzpicture}[baseline=(current bounding box.center)]
\draw (0,-0.5) rectangle ++(1,1);
\draw[red] (0.6,-0.5)--(0.6,0.5);
\draw[blue] (0,0.1)--(0.4,0.1)--(0.4,0.5);
\end{tikzpicture}
\\ 
& & & & & & & & &
\\ \hline
\end{tabular}
}
\end{center}

\noindent We will do case 1 in detail. The rest can be done similarly.

In case 1 the original local configuration of $T$ contributes a power of $t$ as the blue path exits right and a red path is present. Suppose that flipping the red corner down causes $T$ to lose a power of $t$. Marking the original face with a $*$, we must be in the situation 
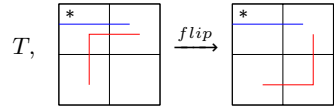
\[
T,\;\;
\resizebox{1.5cm}{!}{
\begin{tikzpicture}[baseline=(current bounding box.center)]
\draw (0,0) rectangle (2,2); \draw[step=1] (0,0) grid (2,2);\node[below right] at (0,2) {$*$};
\draw[blue] (0,1.6)--(1.4,1.6); \draw[red] (0.6,0.4)--(0.6,1.4)--(1.6,1.4);
\end{tikzpicture}
}
\xrightarrow{flip}
\resizebox{1.5cm}{!}{
\begin{tikzpicture}[baseline=(current bounding box.center)]
\draw (0,0) rectangle (2,2); \draw[step=1] (0,0) grid (2,2);\node[below right] at (0,2) {$*$};
\draw[blue] (0,1.6)--(1.4,1.6); \draw[red] (0.6,0.4)--(1.6,0.4)--(1.6,1.4);
\end{tikzpicture}
}
\]
where no blue path is present in the bottom-right face (a blue path exiting right would contribute a power of $t$, and the presence of a blue path in the top-right face prevents a blue path from exiting up). In $\Phi(T)$ this corresponds to the situation
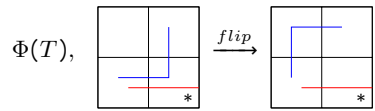
\[
\Phi(T),\;\;
\resizebox{1.5cm}{!}{
\begin{tikzpicture}[baseline=(current bounding box.center)]
\draw (0,0) rectangle (2,2); \draw[step=1] (0,0) grid (2,2);\node[above left] at (2,0) {$*$};
\draw[blue] (0.4,0.6)--(1.4,0.6)--(1.4,1.6); \draw[red] (0.6,0.4)--(2,0.4);
\end{tikzpicture}
}
\xrightarrow{flip}
\resizebox{1.5cm}{!}{
\begin{tikzpicture}[baseline=(current bounding box.center)]
\draw (0,0) rectangle (2,2); \draw[step=1] (0,0) grid (2,2);\node[above left] at (2,0) {$*$};
\draw[blue] (0.4,0.6)--(0.4,1.6)--(1.4,1.6); \draw[red] (0.6,0.4)--(2,0.4);
\end{tikzpicture}
}
\]
where we flip the blue path up. Note that here the bottom-left face originally contributed a power of $t$, which is lost after the flip. So we lose a power of $t$ in both $T$ and $\Phi(T)$.

Now suppose there was no change in the power of $t$ after flipping the corner in $T$. Then we must be in the situation
\[
T,\;\;
\resizebox{1.5cm}{!}{
\begin{tikzpicture}[baseline=(current bounding box.center)]
\draw (0,0) rectangle (2,2); \draw[step=1] (0,0) grid (2,2);\node[below right] at (0,2) {$*$};
\draw[blue] (0,1.6)--(1.4,1.6); \draw[red] (0.6,0.4)--(0.6,1.4)--(1.6,1.4);
\draw[blue] (1.4,0.6)--(2,0.6);
\end{tikzpicture}
}
\xrightarrow{flip}
\resizebox{1.5cm}{!}{
\begin{tikzpicture}[baseline=(current bounding box.center)]
\draw (0,0) rectangle (2,2); \draw[step=1] (0,0) grid (2,2);\node[below right] at (0,2) {$*$};
\draw[blue] (0,1.6)--(1.4,1.6); \draw[red] (0.6,0.4)--(1.6,0.4)--(1.6,1.4);
\draw[blue] (1.4,0.6)--(2,0.6);
\end{tikzpicture}
}
\]
where a blue paths exits right in the bottom-right face contributing a power of $t$ to make up for the one lost from the top-right face. In $\Phi(T)$ this corresponds to
\[
\Phi(T),\;\;
\resizebox{1.5cm}{!}{
\begin{tikzpicture}[baseline=(current bounding box.center)]
\draw (0,0) rectangle (2,2); \draw[step=1] (0,0) grid (2,2);\node[above left] at (2,0) {$*$};
\draw[blue] (0.4,0.6)--(1.4,0.6)--(1.4,1.6); \draw[red] (0.6,0.4)--(2,0.4);
\draw[red] (0,1.4)--(0.6,1.4);
\end{tikzpicture}
}
\xrightarrow{flip}
\resizebox{1.5cm}{!}{
\begin{tikzpicture}[baseline=(current bounding box.center)]
\draw (0,0) rectangle (2,2); \draw[step=1] (0,0) grid (2,2);\node[above left] at (2,0) {$*$};
\draw[blue] (0.4,0.6)--(0.4,1.6)--(1.4,1.6); \draw[red] (0.6,0.4)--(2,0.4);
\draw[red] (0,1.4)--(0.6,1.4);
\end{tikzpicture}
}
\]
where now the top-left face contributes a power of $t$ to make up for the one lost from the bottom-left face. So there is no net change in the power of $t$ for either $T$ or $\Phi(T)$.

From this we see that $\coinv(T)-\coinv \Phi(T)$ is independent of the choice of $T$. It is left to compute the quantity for some $T$. We choose $T$ such that all the paths stay as high as possible (no up-flips are available). This choice of configuration $T$ maps to $\Phi(T)$ in which the paths are as low as possible (no down-flips available). For example,
\[
\begin{tikzpicture}[baseline=(current bounding box.center)]
\draw (0,0) -- (0,2); \draw (1,0) -- (1,2); \draw (2,0) -- (2,2);  \draw (3,0) -- (3,2); \draw (4,0) -- (4,2); \draw (5,0) -- (5,2);
\draw (0,0) -- (5,0); \draw (0,1) -- (5,1); \draw (0,2) -- (5,2);

\draw[blue] (1.4,0)--(1.4,1.6)--(3.4,1.6)--(3.4,2);
\draw[blue] (3.4,0)--(3.4,0.6)--(4.4,0.6)--(4.4,2);
\draw[red] (0.6,0)--(0.6,1.4)--(3.6,1.4)--(3.6,2);
\draw[red] (2.6,0)--(2.6,0.4)--(4.6,0.4)--(4.6,2);

\node[above] at (2.5,2) {$T$};
\end{tikzpicture}
\quad \longrightarrow \quad
\begin{tikzpicture}[xscale=-1,yscale=-1,baseline=(current bounding box.center)]
\draw (0,0) -- (0,2); \draw (1,0) -- (1,2); \draw (2,0) -- (2,2);  \draw (3,0) -- (3,2); \draw (4,0) -- (4,2); \draw (5,0) -- (5,2);
\draw (0,0) -- (5,0); \draw (0,1) -- (5,1); \draw (0,2) -- (5,2);

\draw[red] (1.4,0)--(1.4,1.6)--(3.4,1.6)--(3.4,2);
\draw[red] (3.4,0)--(3.4,0.6)--(4.4,0.6)--(4.4,2);
\draw[blue] (0.6,0)--(0.6,1.4)--(3.6,1.4)--(3.6,2);
\draw[blue] (2.6,0)--(2.6,0.4)--(4.6,0.4)--(4.6,2);

\node[above] at (2.5,0) {$\Phi(T)$};
\end{tikzpicture}
\]
One can check that the powers of $t$ for each configuration are given by
\[
\coinv(T) = \sum_{i=1}^n \min(M-n-\lambda_i^{(1)},M-n-\lambda_i^{(2)}) + \sum_{i=1}^n \#\{j<i \;|\; \lambda_i^{(2)}-i\ge \lambda_j^{(1)}-j\}
\]
\[
\begin{aligned}
\coinv \Phi(T) = & \; \sum_{i=1}^n \min((\lambda^{(2)})^c_i,(\lambda^{(1)})^c_i) + \sum_{i=1}^n \#\{j\ge i \;|\; (\lambda^{(2)})^c_j - j > (\lambda^{(1)})^c_i - i\} \\
= & \; \sum_{i=1}^n \min(M-n-\lambda_i^{(1)},M-n-\lambda_i^{(2)}) + \sum_{i=1}^n \#\{j\ge i \; | \; \lambda^{(2)}_i - i < \lambda^{(1)}_j - j\}.
\end{aligned}
\]
where in each, the first sum counts powers of $t$ coming from 
\begin{tikzpicture}[baseline=(current bounding box.center)]
\draw (0,0)--(1,0)--(1,1)--(0,1)--(0,0); \draw[blue] (0.4,0.6)--(1,0.6); \draw[red] (0.6,0.4)--(1,0.4);
\end{tikzpicture} 
and the second sum counts those coming from
\begin{tikzpicture}[baseline=(current bounding box.center)]
\draw (0,0)--(1,0)--(1,1)--(0,1)--(0,0); \draw[blue] (0.4,0.6)--(1,0.6); \draw[red] (0.6,0.4)--(0.6,1);
\end{tikzpicture}.

We have 
\[
\begin{aligned}
d(\bm{\lambda}) = & \; \coinv(T)-\coinv \Phi(T) \\
= & \; \sum_{i=1}^n \#\{j<i\; | \; \lambda_i^{(2)}-i\ge \lambda_j^{(1)}-j\} - \sum_{i=1}^n \#\{j\ge i \; |\; \lambda^{(2)}_i - i < \lambda^{(1)}_j - j\} \\
= & \; \binom{n}{2} - \#\{i,j\; |\; \lambda^{(2)}_i - i < \lambda^{(1)}_j - j\}.
\end{aligned}
\]
To extend this to $k$ colors, we sum over all pairs of colors $a<b$ and get
\[
d(\bm{\lambda}) = \binom{n}{2}\binom{k}{2} - \sum_{a<b} \#\{i,j \;|\; \lambda^{(b)}_i - i < \lambda^{(a)}_j - j\}.
\]
\end{proof}

\begin{proof}[Proof of Lemma \ref{lem:Comp}]
We consider the layout below, in which blue, red paths denote parts of configurations for $\lambda^{(1)}, \lambda^{(2)}$, respectively, and we flip a red corner up from row $a$ to row $a+1$. On the tableaux, this flip is precisely swapping an entry $a \in T^{(2)}$ with $a+1 \in \Phi(T^{(2)})$, and so $\Phi(f(T))$ is gotten from $\Phi(T)$ by flipping a red corner down. 
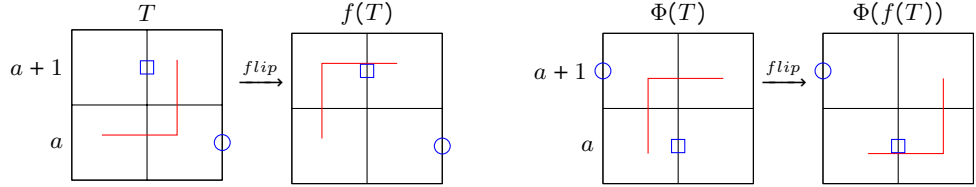
\[
\begin{tikzpicture}[baseline=(current bounding box.center)]
\draw (0,0) rectangle (2,2); \draw[step=1] (0,0) grid (2,2);
\draw[red] (0.4,0.6)--(1.4,0.6)--(1.4,1.6);
\draw[blue, xshift=-2.5pt, yshift=-2.5pt] (1,1.5) rectangle ++(5pt,5pt) ;
\draw[blue] (2,0.5) circle (3pt) ;
\node[above] at (1,2) {$T$};
\node[left] at (0,0.5) {$a$}; \node[left] at (0,1.5) {$a+1$};
\end{tikzpicture}
\xrightarrow{flip}
\begin{tikzpicture}[baseline=(current bounding box.center)]
\draw (0,0) rectangle (2,2); \draw[step=1] (0,0) grid (2,2);
\draw[red] (0.4,0.6)--(0.4,1.6)--(1.4,1.6);
\draw[blue, xshift=-2.5pt, yshift=-2.5pt] (1,1.5) rectangle ++(5pt,5pt) ;
\draw[blue] (2,0.5) circle (3pt) ;
\node[above] at (1,2) {$f(T)$};
\end{tikzpicture}
\hspace{1cm}
\begin{tikzpicture}[baseline=(current bounding box.center)]
\draw (0,0) rectangle (2,2); \draw[step=1] (0,0) grid (2,2);
\draw[red] (0.6,0.4)--(0.6,1.4)--(1.6,1.4);
\draw[blue, xshift=-2.5pt, yshift=-2.5pt] (1,0.5) rectangle ++(5pt,5pt) ;
\draw[blue] (0,1.5) circle (3pt) ;
\node[above] at (1,2) {$\Phi(T)$};
\node[left] at (0,0.5) {$a$}; \node[left] at (-0.1,1.5) {$a+1$};
\end{tikzpicture}
\xrightarrow{flip}
\begin{tikzpicture}[baseline=(current bounding box.center)]
\draw (0,0) rectangle (2,2); \draw[step=1] (0,0) grid (2,2);
\draw[red] (0.6,0.4)--(1.6,0.4)--(1.6,1.4);
\draw[blue, xshift=-2.5pt, yshift=-2.5pt] (1,0.5) rectangle ++(5pt,5pt) ;
\draw[blue] (0,1.5) circle (3pt) ;
\node[above] at (1,2) {$\Phi(f(T))$};
\end{tikzpicture}
\]
Keeping in mind that blue denotes the smaller color on the LHS and the larger color on the RHS, note that $\coinv T \neq \coinv f(T)$ precisely when there is a blue path through one of the square or the circle, but not both, and likewise for $\Phi(T)$. We claim that this holds in $T$ iff it holds in $\Phi(T)$. It will be helpful to note that paths exiting in the same column of the configuration $T$ are exactly entries in the same content line of the filling $T$. 

Indeed, if there is a blue path through the square in $T$, then there is an entry $a+1 \in T^{(1)}$ on the same content line as the entry $a \in T^{(2)}$. If there is no blue path through the circle, then the entry directly below this $a+1$ cannot be $a$. Hence, the filling in $\Phi(T^{(1)})$ contains an $a$ but not an $a+1$ in the corresponding column.The case is pictured below, with a concrete example on the left. In the left picture, $a$ is the circled green 2, and in both the gray cells denote cells in the complement shape. Likewise, if there is a blue path through the square but not the circle in $\Phi(T)$, then we are in the case with the red circles.
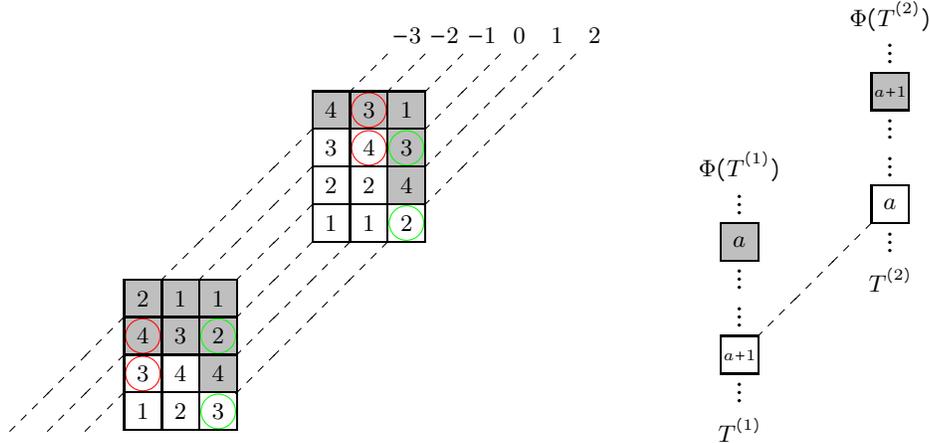
\ytableausetup{nosmalltableaux}
\ytableausetup{nobaseline}
\begin{center}
   \begin{ytableau}
		\none & \none & \none & \none & \none & \none & \none & \none & \none & \none & \none[-3] & \none[-2] & \none[-1] & \none[0] & \none[1] & \none[2] \\
		\none & \none & \none & \none & \none & \none & \none & \none & \none &\none[\diag] &\none[\diag] &\none[\diag] &\none[\diag] & \none[\diag] & \none[\diag] \\
		\none & \none & \none & \none & \none & \none & \none & \none & *(lightgray) 4 & *(lightgray) \circled{3}{red} & *(lightgray) 1 & \none[\diag] & \none[\diag] & \none[\diag] \\
		\none & \none & \none & \none & \none & \none & \none & \none[\diag] & 3 & \circled{4}{red} & *(lightgray) \circled{3}{green} & \none[\diag] & \none[\diag] \\
		\none & \none & \none & \none & \none & \none & \none[\diag] & \none[\diag] & 2 & 2 & *(lightgray) 4 & \none[\diag] \\
		\none & \none & \none & \none & \none & \none[\diag] & \none[\diag] & \none[\diag] & 1 & 1 & \circled{2}{green} \\
		\none & \none & \none & \none & \none[\diag] & \none[\diag] & \none[\diag] & \none[\diag] & \none[\diag] & \none[\diag] \\
		\none & \none & \none & *(lightgray) 2 & *(lightgray) 1 & *(lightgray) 1 & \none[\diag] & \none[\diag] & \none[\diag] \\
		\none & \none & \none[\diag] & *(lightgray) \circled{4}{red} & *(lightgray) 3 & *(lightgray) \circled{2}{green} & \none[\diag] & \none[\diag] \\
		\none & \none[\diag] & \none[\diag] & \circled{3}{red} & 4 & *(lightgray) 4 & \none[\diag]  \\
		\none[\diag] & \none[\diag] & \none[\diag] & 1 & 2 & \circled{3}{green}  \\
	\end{ytableau}
    \qquad \qquad
	\begin{ytableau}
	\none & \none & \none & \none & \none[\Phi(T^{(2)})] \\
	\none & \none & \none & \none & \none[\vdots] \\
	\none & \none & \none & \none & *(lightgray) \scriptstyle a+1 \\
	\none & \none & \none & \none & \none[\vdots] \\
	\none[\Phi(T^{(1)})] & \none & \none & \none & \none[\vdots]  \\
	\none[\vdots] & \none & \none & \none & a \\
	*(lightgray) a & \none & \none & \none[\diag] & \none[\vdots] \\
	\none[\vdots] & \none & \none[\diag] & \none & \none[T^{(2)}] \\
	\none[\vdots] & \none[\diag] \\
	\scriptstyle a+1 \\
	\none[\vdots] \\
	\none[T^{(1)}]
	\end{ytableau}	
\end{center}

\ytableausetup{smalltableaux, boxsize=1.25em}
It remains to see that this $\ytableaushort[*(lightgray)]{a} \in \Phi(T^{(1)})$ is along the same content line as the $\begin{ytableau} *(lightgray) \scriptstyle a+1 \end{ytableau} \in \Phi(T^{(2)})$. Let $c_1, c_2$ denote the columns above. Note that due to column strictness, every number in the interval $[a+2, n]$ in both $c_1$ and $c_2$ must lie between the cells $a$ and $a+1$. Thus, the number of cells between $a$ and $a+1$ is the same, and since we assume that $\begin{ytableau} \scriptstyle a+1 \end{ytableau} \in T^{(1)}$ and $\ytableaushort{a} \in T^{(2)}$ are on the same content line, then so too must $\ytableaushort[*(lightgray)]{a} \in \Phi(T^{(1)})$ and $\begin{ytableau} *(lightgray) \scriptstyle a+1 \end{ytableau} \in \Phi(T^{(2)})$.

A similar argument holds for when there is a blue path through the circle and not the square.

We now explicitly calculate $\coinv T - \coinv \Phi(T)$ by fixing a particular configuration $T$. We again choose the configuration in which the paths are as low as possible, so that no down-flips are available. On the tuple of tableaux, this corresponds to the superstandard filling of the shape $\lambda$, i.e. the filling in which the $i^{th}$ row is filled with only $i$'s. Hence, $\Phi(T)$ is the configuration in which the paths are as high as possible, so that no up-flips are available.
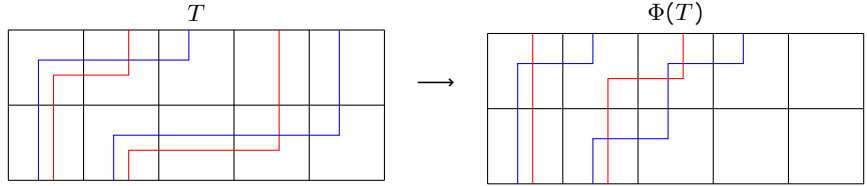
\[
\begin{tikzpicture}[baseline=(current bounding box.center)]
\draw (0,0) -- (0,2); \draw (1,0) -- (1,2); \draw (2,0) -- (2,2);  \draw (3,0) -- (3,2); \draw (4,0) -- (4,2); \draw (5,0) -- (5,2);
\draw (0,0) -- (5,0); \draw (0,1) -- (5,1); \draw (0,2) -- (5,2);

\draw[red] (0.6,0)--(0.6,1.4)--(1.6,1.4)--(1.6,2);
\draw[red] (1.6,0)--(1.6,0.4)--(3.6,0.4)--(3.6,2);
\draw[blue] (0.4,0)--(0.4,1.6)--(2.4,1.6)--(2.4,2);
\draw[blue] (1.4,0)--(1.4,0.6)--(4.4,0.6)--(4.4,2);

\node[above] at (2.5,2) {$T$};
\end{tikzpicture}
\quad \longrightarrow \quad
\begin{tikzpicture}[baseline=(current bounding box.center)]
\draw (0,0) -- (0,2); \draw (1,0) -- (1,2); \draw (2,0) -- (2,2);  \draw (3,0) -- (3,2); \draw (4,0) -- (4,2); \draw (5,0) -- (5,2);
\draw (0,0) -- (5,0); \draw (0,1) -- (5,1); \draw (0,2) -- (5,2);

\draw[blue] (0.4,0)--(0.4,1.6)--(1.4,1.6)--(1.4,2);
\draw[blue] (1.4,0)--(1.4,0.6)--(2.4,0.6)--(2.4,1.6)--(3.4,1.6)--(3.4,2);
\draw[red] (0.6,0)--(0.6,2);
\draw[red] (1.6,0)--(1.6,1.4)--(2.6,1.4)--(2.6,2);

\node[above] at (2.5,2) {$\Phi(T)$};
\end{tikzpicture}
\]
As in the proof of Lemma \ref{lem:BoxComp} above, one can check that the powers of $t$ in $T$ and $\Phi(T)$ are given by
\[
\coinv(T) = \sum_{i=1}^n \min(\lambda^{(1)}_i,\lambda^{(2)}_i) + \sum_{i=1}^n \#\{j\ge i \;|\; \lambda^{(1)}_j - j > \lambda^{(2)}_i - i\}
\]
and
\begin{align*}
\coinv \Phi(T)  &= \sum_{i} \min (\lambda^{(1), c}_i, \lambda^{(2), c}_i) + \# \{ i \leq j \mid \lambda^{(1), c}_i - i < \lambda^{(2), c}_j - j \} \\
&= \sum_i \min (M-n - \lambda^{(1)}_{M-n-i}, M-n-\lambda^{(2)}_{M-n-i}) + \# \{ i \leq j \mid -\lambda^{(1)}_{M-n-i} - i < -\lambda^{(2), c}_{M-n-j} - j \} \\
&= (M-n)n - \sum_i \max(\lambda^{(1)}_{i}, \lambda^{(2)}_{i}) + \# \{ i \leq j \mid \lambda^{(1)}_{M-n-i} +i > \lambda^{(2)}_{M-n-j} + j \} \\
&= (M-n)n - \sum_i \max(\lambda^{(1)}_{i}, \lambda^{(2)}_{i}) + \# \{ i \geq j \mid \lambda^{(1)}_{i} -i > \lambda^{(2)}_{j} -j \}
\end{align*}
Hence,
\begin{align*}
\coinv T - \coinv \Phi(T) &= \sum_{i=1}^n \min(\lambda^{(1)}_i,\lambda^{(2)}_i) - (M-n)n + \sum_i \max(\lambda^{(1)}_{i}, \lambda^{(2)}_{i}) \\
&= |\lambda^{(1)}| + |\lambda^{(2)}| - (M-n)n
\end{align*} 
To extend this to $k$ colors, we sum over all pairs $a < b$, giving
\[ \widetilde{d}(\bm{\lambda}) =  (k-1)|\bm{\lambda}| - n(M-n)\binom{k}{2} \]
\end{proof}

\section{Examples \label{AppendixExamples}}

\tikzstyle{every picture}=[baseline=(current bounding box.south)]

Throughout this section we order our colors such that blue is color 1 and red is color 2.
\subsection{Two-color weights}
The weights of $L_x$ in the case of two colors are given by
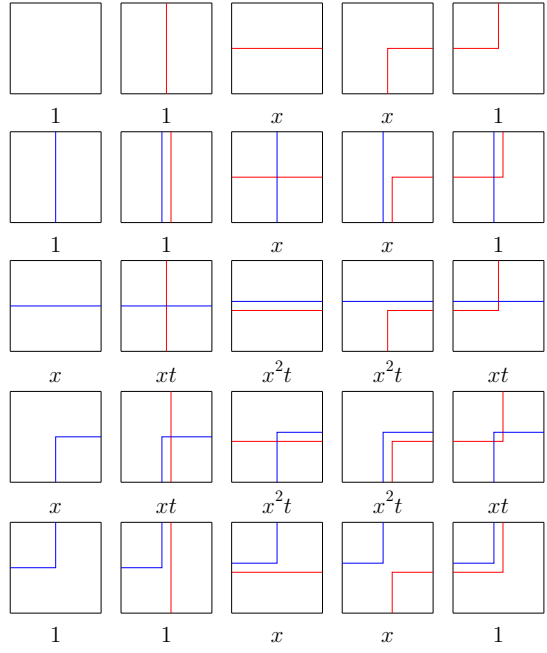
\begin{center}
{\Large
\resizebox{0.5\textwidth}{!}{
\begin{tabular}{ccccc}
\begin{tikzpicture}
\draw (-1,-1) -- (1,-1); \draw (-1,-1) -- (-1,1);
\draw (1,1) -- (1,-1); \draw (1,1) -- (-1,1);
\end{tikzpicture}
&
\begin{tikzpicture}
\draw (-1,-1) -- (1,-1); \draw (-1,-1) -- (-1,1);
\draw (1,1) -- (1,-1); \draw (1,1) -- (-1,1);
\draw[red] (0,-1) -- (0,1);
\end{tikzpicture}
&
\begin{tikzpicture}
\draw (-1,-1) -- (1,-1); \draw (-1,-1) -- (-1,1);
\draw (1,1) -- (1,-1); \draw (1,1) -- (-1,1);
\draw[red] (-1,0) -- (1,0);
\end{tikzpicture}
&
\begin{tikzpicture}
\draw (-1,-1) -- (1,-1); \draw (-1,-1) -- (-1,1);
\draw (1,1) -- (1,-1); \draw (1,1) -- (-1,1);
\draw[red] (0,-1) -- (0,0) -- (1,0);
\end{tikzpicture}
&
\begin{tikzpicture}
\draw (-1,-1) -- (1,-1); \draw (-1,-1) -- (-1,1);
\draw (1,1) -- (1,-1); \draw (1,1) -- (-1,1);
\draw[red] (-1,0) -- (0,0) -- (0,1);
\end{tikzpicture}
\\
$1$ & $1 $ & $x$ & $x$ & $1$
\\
\begin{tikzpicture}
\draw (-1,-1) -- (1,-1); \draw (-1,-1) -- (-1,1);
\draw (1,1) -- (1,-1); \draw (1,1) -- (-1,1);
\draw[blue] (0,-1) -- (0,1);
\end{tikzpicture}
&
\begin{tikzpicture}
\draw (-1,-1) -- (1,-1); \draw (-1,-1) -- (-1,1);
\draw (1,1) -- (1,-1); \draw (1,1) -- (-1,1);
\draw[red] (0.1,-1) -- (0.1,1);
\draw[blue] (-0.1,-1) -- (-0.1,1);
\end{tikzpicture}
&
\begin{tikzpicture}
\draw (-1,-1) -- (1,-1); \draw (-1,-1) -- (-1,1);
\draw (1,1) -- (1,-1); \draw (1,1) -- (-1,1);
\draw[red] (-1,0) -- (1,0);
\draw[blue] (0,-1) -- (0,1);
\end{tikzpicture}
&
\begin{tikzpicture}
\draw (-1,-1) -- (1,-1); \draw (-1,-1) -- (-1,1);
\draw (1,1) -- (1,-1); \draw (1,1) -- (-1,1);
\draw[red] (0.1,-1) -- (0.1,0) -- (1,0);
\draw[blue] (-0.1,-1) -- (-0.1,1);
\end{tikzpicture}
&
\begin{tikzpicture}
\draw (-1,-1) -- (1,-1); \draw (-1,-1) -- (-1,1);
\draw (1,1) -- (1,-1); \draw (1,1) -- (-1,1);
\draw[red] (-1,0) -- (0.1,0) -- (0.1,1);
\draw[blue] (-0.1,-1) -- (-0.1,1);
\end{tikzpicture}
\\
$1$ & $1 $ & $x$ & $x$ & $1$
\\
\begin{tikzpicture}
\draw (-1,-1) -- (1,-1); \draw (-1,-1) -- (-1,1);
\draw (1,1) -- (1,-1); \draw (1,1) -- (-1,1);
\draw[blue] (-1,0) -- (1,0);
\end{tikzpicture}
&
\begin{tikzpicture}
\draw (-1,-1) -- (1,-1); \draw (-1,-1) -- (-1,1);
\draw (1,1) -- (1,-1); \draw (1,1) -- (-1,1);
\draw[red] (0,-1) -- (0,1);
\draw[blue] (-1,0) -- (1,0);
\end{tikzpicture}
&
\begin{tikzpicture}
\draw (-1,-1) -- (1,-1); \draw (-1,-1) -- (-1,1);
\draw (1,1) -- (1,-1); \draw (1,1) -- (-1,1);
\draw[red] (-1,-0.1) -- (1,-0.1);
\draw[blue] (-1,0.1) -- (1,0.1);
\end{tikzpicture}
&
\begin{tikzpicture}
\draw (-1,-1) -- (1,-1); \draw (-1,-1) -- (-1,1);
\draw (1,1) -- (1,-1); \draw (1,1) -- (-1,1);
\draw[red] (0,-1) -- (0,-0.1) -- (1,-0.1);
\draw[blue] (-1,0.1) -- (1,0.1);
\end{tikzpicture}
&
\begin{tikzpicture}
\draw (-1,-1) -- (1,-1); \draw (-1,-1) -- (-1,1);
\draw (1,1) -- (1,-1); \draw (1,1) -- (-1,1);
\draw[red] (-1,-0.1) -- (0,-0.1) -- (0,1);
\draw[blue] (-1,0.1) -- (1,0.1);
\end{tikzpicture}
\\
$x$ & $xt$ & $x^2t$ & $x^2t$ & $xt$
\\
\begin{tikzpicture}
\draw (-1,-1) -- (1,-1); \draw (-1,-1) -- (-1,1);
\draw (1,1) -- (1,-1); \draw (1,1) -- (-1,1);
\draw[blue] (0,-1) -- (0,0) -- (1,0);
\end{tikzpicture}
&
\begin{tikzpicture}
\draw (-1,-1) -- (1,-1); \draw (-1,-1) -- (-1,1);
\draw (1,1) -- (1,-1); \draw (1,1) -- (-1,1);
\draw[red] (0.1,-1) -- (0.1,1);
\draw[blue] (-0.1,-1) -- (-0.1,0) -- (1,0);
\end{tikzpicture}
&
\begin{tikzpicture}
\draw (-1,-1) -- (1,-1); \draw (-1,-1) -- (-1,1);
\draw (1,1) -- (1,-1); \draw (1,1) -- (-1,1);
\draw[red] (-1,-0.1) -- (1,-0.1);
\draw[blue] (0,-1) -- (0,0.1) -- (1,0.1);
\end{tikzpicture}
&
\begin{tikzpicture}
\draw (-1,-1) -- (1,-1); \draw (-1,-1) -- (-1,1);
\draw (1,1) -- (1,-1); \draw (1,1) -- (-1,1);
\draw[red] (0.1,-1) -- (0.1,-0.1) -- (1,-0.1);
\draw[blue] (-0.1,-1) -- (-0.1,0.1) -- (1,0.1);
\end{tikzpicture}
&
\begin{tikzpicture}
\draw (-1,-1) -- (1,-1); \draw (-1,-1) -- (-1,1);
\draw (1,1) -- (1,-1); \draw (1,1) -- (-1,1);
\draw[red] (-1,-0.1) -- (0.1,-0.1) -- (0.1,1);
\draw[blue] (-0.1,-1) -- (-0.1,0.1) -- (1,0.1);
\end{tikzpicture}
\\
$x$ & $xt$ & $x^2t$ & $x^2t$ & $xt$
\\
\begin{tikzpicture}
\draw (-1,-1) -- (1,-1); \draw (-1,-1) -- (-1,1);
\draw (1,1) -- (1,-1); \draw (1,1) -- (-1,1);
\draw[blue] (-1,0) -- (0,0) -- (0,1);
\end{tikzpicture}
&
\begin{tikzpicture}
\draw (-1,-1) -- (1,-1); \draw (-1,-1) -- (-1,1);
\draw (1,1) -- (1,-1); \draw (1,1) -- (-1,1);
\draw[red] (0.1,-1) -- (0.1,1);
\draw[blue] (-1,0) -- (-0.1,0) -- (-0.1,1);
\end{tikzpicture}
&
\begin{tikzpicture}
\draw (-1,-1) -- (1,-1); \draw (-1,-1) -- (-1,1);
\draw (1,1) -- (1,-1); \draw (1,1) -- (-1,1);
\draw[red] (-1,-0.1) -- (1,-0.1);
\draw[blue] (-1,0.1) -- (0,0.1) -- (0,1);
\end{tikzpicture}
&
\begin{tikzpicture}
\draw (-1,-1) -- (1,-1); \draw (-1,-1) -- (-1,1);
\draw (1,1) -- (1,-1); \draw (1,1) -- (-1,1);
\draw[red] (0.1,-1) -- (0.1,-0.1) -- (1,-0.1);
\draw[blue] (-1,0.1) -- (-0.1,0.1) -- (-0.1,1);
\end{tikzpicture}
&
\begin{tikzpicture}
\draw (-1,-1) -- (1,-1); \draw (-1,-1) -- (-1,1);
\draw (1,1) -- (1,-1); \draw (1,1) -- (-1,1);
\draw[red] (-1,-0.1) -- (0.1,-0.1) -- (0.1,1);
\draw[blue] (-1,0.1) -- (-0.1,0.1) -- (-0.1,1);
\end{tikzpicture}
\\
$1$ & $1$ & $x$ & $x$ & $1$
\end{tabular}
}
}
\end{center}

The weights of $R_{y/x}$ in the case of two colors are given by
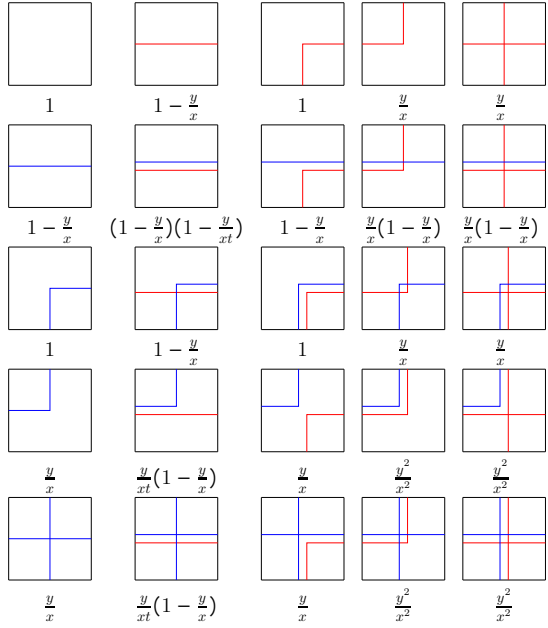
\begin{center}
{\Large
\resizebox{0.5\textwidth}{!}{
\begin{tabular}{cccccc}
\begin{tikzpicture}
\draw (-1,-1) -- (1,-1); \draw (-1,-1) -- (-1,1);
\draw (1,1) -- (1,-1); \draw (1,1) -- (-1,1);
\end{tikzpicture}
&
\begin{tikzpicture}
\draw (-1,-1) -- (1,-1); \draw (-1,-1) -- (-1,1);
\draw (1,1) -- (1,-1); \draw (1,1) -- (-1,1);
\draw[red] (-1,0) -- (1,0);
\end{tikzpicture}
&
\begin{tikzpicture}
\draw (-1,-1) -- (1,-1); \draw (-1,-1) -- (-1,1);
\draw (1,1) -- (1,-1); \draw (1,1) -- (-1,1);
\draw[red] (0,-1) -- (0,0) -- (1,0);
\end{tikzpicture}
&
\begin{tikzpicture}
\draw (-1,-1) -- (1,-1); \draw (-1,-1) -- (-1,1);
\draw (1,1) -- (1,-1); \draw (1,1) -- (-1,1);
\draw[red] (-1,0) -- (0,0) -- (0,1);
\end{tikzpicture}
&
\begin{tikzpicture}
\draw (-1,-1) -- (1,-1); \draw (-1,-1) -- (-1,1);
\draw (1,1) -- (1,-1); \draw (1,1) -- (-1,1);
\draw[red] (-1,0) -- (1,0); \draw[red] (0,-1) -- (0,1);
\end{tikzpicture}
\\
$1$ & $1-\frac{y}{x}$ & $1$ & $\frac{y}{x}$ & $\frac{y}{x}$
\vspace{1mm}
\\
\begin{tikzpicture}
\draw (-1,-1) -- (1,-1); \draw (-1,-1) -- (-1,1);
\draw (1,1) -- (1,-1); \draw (1,1) -- (-1,1);
\draw[blue] (-1,0) -- (1,0);
\end{tikzpicture}
&
\begin{tikzpicture}
\draw (-1,-1) -- (1,-1); \draw (-1,-1) -- (-1,1);
\draw (1,1) -- (1,-1); \draw (1,1) -- (-1,1);
\draw[red] (-1,-0.1) -- (1,-0.1);
\draw[blue] (-1,0.1) -- (1,0.1);
\end{tikzpicture}
&
\begin{tikzpicture}
\draw (-1,-1) -- (1,-1); \draw (-1,-1) -- (-1,1);
\draw (1,1) -- (1,-1); \draw (1,1) -- (-1,1);
\draw[red] (0,-1) -- (0,-0.1) -- (1,-0.1);
\draw[blue] (-1,0.1) -- (1,0.1);
\end{tikzpicture}
&
\begin{tikzpicture}
\draw (-1,-1) -- (1,-1); \draw (-1,-1) -- (-1,1);
\draw (1,1) -- (1,-1); \draw (1,1) -- (-1,1);
\draw[red] (-1,-0.1) -- (0,-0.1) -- (0,1);
\draw[blue] (-1,0.1) -- (1,0.1);
\end{tikzpicture}
&
\begin{tikzpicture}
\draw (-1,-1) -- (1,-1); \draw (-1,-1) -- (-1,1);
\draw (1,1) -- (1,-1); \draw (1,1) -- (-1,1);
\draw[red] (-1,-0.1) -- (1,-0.1); \draw[red] (0,-1) -- (0,1);
\draw[blue] (-1,0.1) -- (1,0.1);
\end{tikzpicture}
\\
$1-\frac{y}{x}$ & $(1-\frac{y}{x})(1-\frac{y}{xt})$ & $1-\frac{y}{x}$ & $\frac{y}{x}(1-\frac{y}{x})$ & $\frac{y}{x}(1-\frac{y}{x})$
\vspace{1mm}
\\
\begin{tikzpicture}
\draw (-1,-1) -- (1,-1); \draw (-1,-1) -- (-1,1);
\draw (1,1) -- (1,-1); \draw (1,1) -- (-1,1);
\draw[blue] (0,-1) -- (0,0) -- (1,0);
\end{tikzpicture}
&
\begin{tikzpicture}
\draw (-1,-1) -- (1,-1); \draw (-1,-1) -- (-1,1);
\draw (1,1) -- (1,-1); \draw (1,1) -- (-1,1);
\draw[red] (-1,-0.1) -- (1,-0.1);
\draw[blue] (0,-1) -- (0,0.1) -- (1,0.1);
\end{tikzpicture}
&
\begin{tikzpicture}
\draw (-1,-1) -- (1,-1); \draw (-1,-1) -- (-1,1);
\draw (1,1) -- (1,-1); \draw (1,1) -- (-1,1);
\draw[red] (0.1,-1) -- (0.1,-0.1) -- (1,-0.1);
\draw[blue] (-0.1,-1) -- (-0.1,0.1) -- (1,0.1);
\end{tikzpicture}
&
\begin{tikzpicture}
\draw (-1,-1) -- (1,-1); \draw (-1,-1) -- (-1,1);
\draw (1,1) -- (1,-1); \draw (1,1) -- (-1,1);
\draw[red] (-1,-0.1) -- (0.1,-0.1) -- (0.1,1);
\draw[blue] (-0.1,-1) -- (-0.1,0.1) -- (1,0.1);
\end{tikzpicture}
&
\begin{tikzpicture}
\draw (-1,-1) -- (1,-1); \draw (-1,-1) -- (-1,1);
\draw (1,1) -- (1,-1); \draw (1,1) -- (-1,1);
\draw[red] (-1,-0.1) -- (1,-0.1); \draw[red] (0.1,-1) -- (0.1,1);
\draw[blue] (-0.1,-1) -- (-0.1,0.1) -- (1,0.1);
\end{tikzpicture}
\\
$1$ & $1-\frac{y}{x}$ & $1$ & $\frac{y}{x}$ & $\frac{y}{x}$
\vspace{1mm}
\\
\begin{tikzpicture}
\draw (-1,-1) -- (1,-1); \draw (-1,-1) -- (-1,1);
\draw (1,1) -- (1,-1); \draw (1,1) -- (-1,1);
\draw[blue] (-1,0) -- (0,0) -- (0,1);
\end{tikzpicture}
&
\begin{tikzpicture}
\draw (-1,-1) -- (1,-1); \draw (-1,-1) -- (-1,1);
\draw (1,1) -- (1,-1); \draw (1,1) -- (-1,1);
\draw[red] (-1,-0.1) -- (1,-0.1);
\draw[blue] (-1,0.1) -- (0,0.1) -- (0,1);
\end{tikzpicture}
&
\begin{tikzpicture}
\draw (-1,-1) -- (1,-1); \draw (-1,-1) -- (-1,1);
\draw (1,1) -- (1,-1); \draw (1,1) -- (-1,1);
\draw[red] (0.1,-1) -- (0.1,-0.1) -- (1,-0.1);
\draw[blue] (-1,0.1) -- (-0.1,0.1) -- (-0.1,1);
\end{tikzpicture}
&
\begin{tikzpicture}
\draw (-1,-1) -- (1,-1); \draw (-1,-1) -- (-1,1);
\draw (1,1) -- (1,-1); \draw (1,1) -- (-1,1);
\draw[red] (-1,-0.1) -- (0.1,-0.1) -- (0.1,1);
\draw[blue] (-1,0.1) -- (-0.1,0.1) -- (-0.1,1);
\end{tikzpicture}
&
\begin{tikzpicture}
\draw (-1,-1) -- (1,-1); \draw (-1,-1) -- (-1,1);
\draw (1,1) -- (1,-1); \draw (1,1) -- (-1,1);
\draw[red] (-1,-0.1) -- (1,-0.1); \draw[red] (0.1,-1) -- (0.1,1);
\draw[blue] (-1,0.1) -- (-0.1,0.1) -- (-0.1,1);
\end{tikzpicture}
\\
$\frac{y}{x}$ & $\frac{y}{xt}(1-\frac{y}{x})$ & $\frac{y}{x}$ & $\frac{y^2}{x^2}$ & $\frac{y^2}{x^2}$
\vspace{1mm}
\\
\begin{tikzpicture}
\draw (-1,-1) -- (1,-1); \draw (-1,-1) -- (-1,1);
\draw (1,1) -- (1,-1); \draw (1,1) -- (-1,1);
\draw[blue] (-1,0) -- (1,0); \draw[blue] (0,-1) -- (0,1);
\end{tikzpicture}
&
\begin{tikzpicture}
\draw (-1,-1) -- (1,-1); \draw (-1,-1) -- (-1,1);
\draw (1,1) -- (1,-1); \draw (1,1) -- (-1,1);
\draw[red] (-1,-0.1) -- (1,-0.1);
\draw[blue] (-1,0.1) -- (1,0.1); \draw[blue] (0,-1) -- (0,1);
\end{tikzpicture}
&
\begin{tikzpicture}
\draw (-1,-1) -- (1,-1); \draw (-1,-1) -- (-1,1);
\draw (1,1) -- (1,-1); \draw (1,1) -- (-1,1);
\draw[red] (0.1,-1) -- (0.1,-0.1) -- (1,-0.1);
\draw[blue] (-1,0.1) -- (1,0.1); \draw[blue] (-0.1,-1) -- (-0.1,1);
\end{tikzpicture}
&
\begin{tikzpicture}
\draw (-1,-1) -- (1,-1); \draw (-1,-1) -- (-1,1);
\draw (1,1) -- (1,-1); \draw (1,1) -- (-1,1);
\draw[red] (-1,-0.1) -- (0.1,-0.1) -- (0.1,1);
\draw[blue] (-1,0.1) -- (1,0.1); \draw[blue] (-0.1,-1) -- (-0.1,1);
\end{tikzpicture}
&
\begin{tikzpicture}
\draw (-1,-1) -- (1,-1); \draw (-1,-1) -- (-1,1);
\draw (1,1) -- (1,-1); \draw (1,1) -- (-1,1);
\draw[red] (-1,-0.1) -- (1,-0.1); \draw[red] (0.1,-1) -- (0.1,1);
\draw[blue] (-1,0.1) -- (1,0.1); \draw[blue] (-0.1,-1) -- (-0.1,1);
\end{tikzpicture}
\\
 ${\Large\frac{y}{x}}$  & $\frac{y}{xt}(1-\frac{y}{x})$ & $\frac{y}{x}$ & $\frac{y^2}{x^2}$ & $\frac{y^2}{x^2}$
\end{tabular}
} 
}
\end{center}

\subsection{Examples of $\mathcal{L}_{\bg}(X;t)$}
\indent We use Theorem \ref{LatticeModel} to compute two coinversion LLT polynomials.

First we compute $\mathcal{L}_\bg(x_1,...,x_n;t)$ where $\bg = ((3)/(0),(2)/(0))$ and $n = 2$.  There are 12 semistandard Young tableaux of shape $\bg$ with entries in $\{1,2\}$, which are shown below along with the corresponding lattices and polynomials.

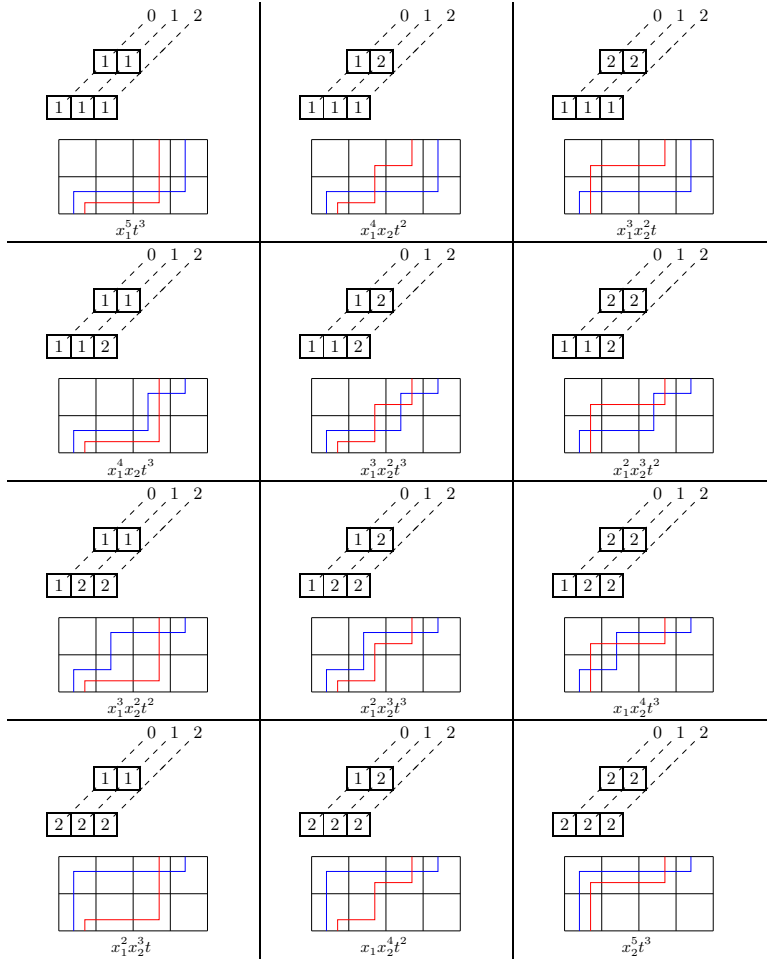
\begin{center}
\resizebox{4in}{!}{
\begin{tabular}{c|c|c}
\extwo{1}{1}{1}{1}{1}{$x_1^5t^3$} & \extwo{1}{1}{1}{1}{2}{$x_1^4x_2t^2$} & \extwo{1}{1}{1}{2}{2}{$x_1^3x_2^2t$} \\ \hline
\extwo{1}{1}{2}{1}{1}{$x_1^4x_2t^3$} & \extwo{1}{1}{2}{1}{2}{$x_1^3x_2^2t^3$} & \extwo{1}{1}{2}{2}{2}{$x_1^2x_2^3t^2$} \\ \hline
\extwo{1}{2}{2}{1}{1}{$x_1^3x_2^2t^2$} & \extwo{1}{2}{2}{1}{2}{$x_1^2x_2^3t^3$} & \extwo{1}{2}{2}{2}{2}{$x_1x_2^4t^3$} \\ \hline
\extwo{2}{2}{2}{1}{1}{$x_1^2x_2^3t$} & \extwo{2}{2}{2}{1}{2}{$x_1x_2^4t^2$} & \extwo{2}{2}{2}{2}{2}{$x_2^5t^3$}
\end{tabular}}
\end{center}

\noindent Therefore the coinversion LLT polynomial is
$$
\begin{aligned}
\mathcal{L}_{((3)/(0),(2)/(0))}(x_1,x_2;t) = \enspace &t(x_1^2x_2^3+x_1^3x_2^2) 
+ t^2(x_1x_2^4+x_1^2x_2^3+x_1^3x_2^2+x_1^4x_2) \\
&+ t^3(x_2^5+x_1x_2^4+x_1^2x_2^3+x_1^3x_2^2+x_1^4x_2+x_1^5).
\end{aligned}
$$

\indent Next, we compute $\mathcal{L}_\bg(x_1,...,x_n;t)$ where $\bg = ((3,3)/(2,1), (3,1)/(1,0))$ and $n = 2$.  Again there are 12 semistandard Young tableaux, shown below along with the corresponding lattices and polynomials.

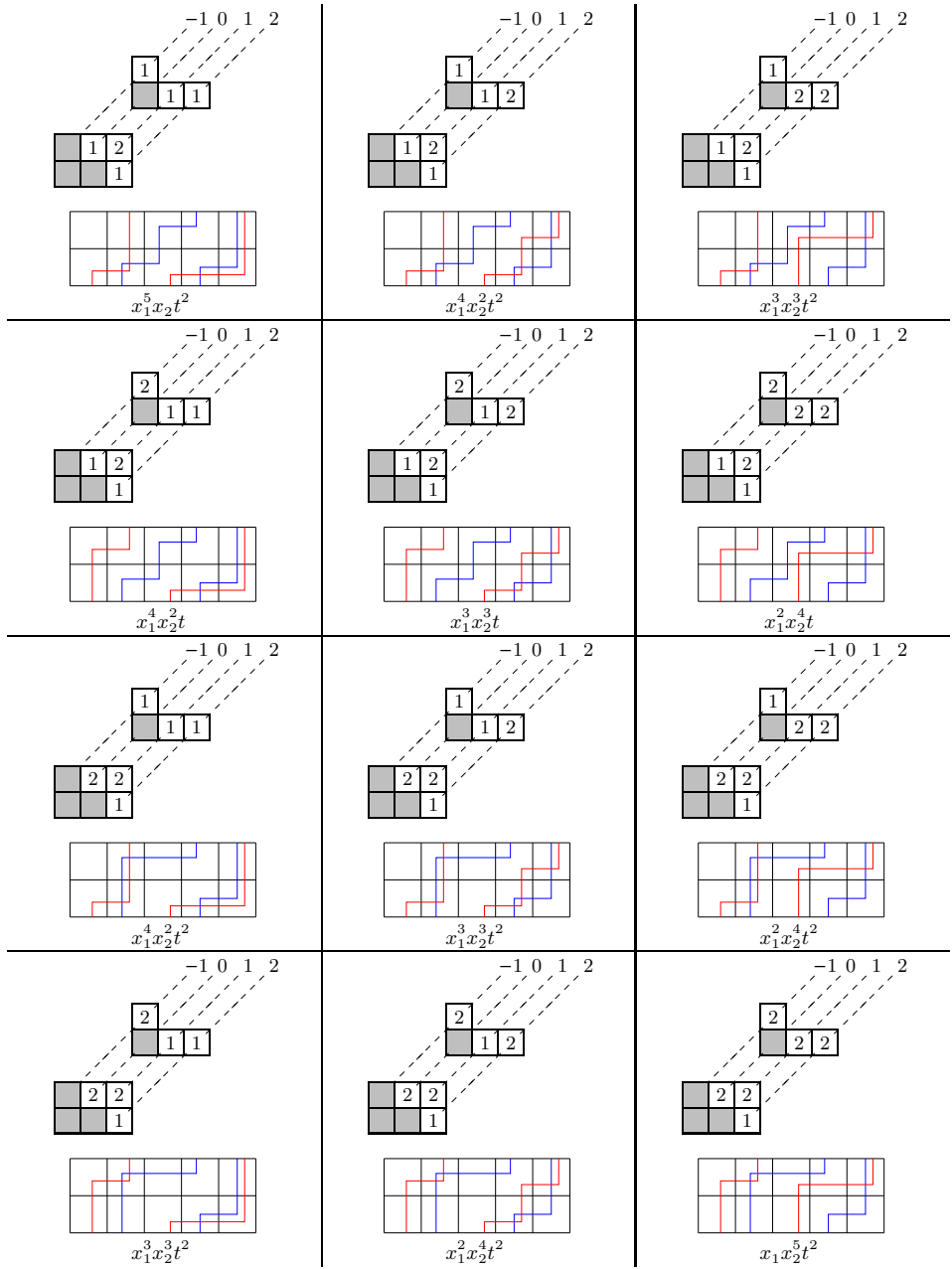
\begin{center}
\resizebox{5in}{!}{
\begin{tabular}{c|c|c}
\exone{1}{1}{2}{1}{1}{1}{$x_1^5x_2t^2$} & \exone{1}{1}{2}{1}{2}{1}{$x_1^4x_2^2t^2$} & \exone{1}{1}{2}{2}{2}{1}{$x_1^3x_2^3t^2$} \\ \hline
\exone{1}{1}{2}{1}{1}{2}{$x_1^4x_2^2t$} & \exone{1}{1}{2}{1}{2}{2}{$x_1^3x_2^3t$} & \exone{1}{1}{2}{2}{2}{2}{$x_1^2x_2^4t$} \\ \hline 
\exone{1}{2}{2}{1}{1}{1}{$x_1^4x_2^2t^2$} & \exone{1}{2}{2}{1}{2}{1}{$x_1^3x_2^3t^2$} & \exone{1}{2}{2}{2}{2}{1}{$x_1^2x_2^4t^2$} \\ \hline
\exone{1}{2}{2}{1}{1}{2}{$x_1^3x_2^3t^2$} & \exone{1}{2}{2}{1}{2}{2}{$x_1^2x_2^4t^2$} & \exone{1}{2}{2}{2}{2}{2}{$x_1x_2^5t^2$}
\end{tabular}}
\end{center}

\noindent Therefore the coinversion LLT polynomial is
$$
\begin{aligned}
\mathcal{L}_{((3,3)/(2,1), (3,1)/(1,0))}(x_1,x_2;t) = \enspace &t(x_1^2x_2^4+x_1^3x_2^3+x_1^4x_2^2) \\
&+ t^2(x_1x_2^5 + 2x_1^2x_2^4 + 3x_1^3x_2^3 + 2x_1^4x_2^2 + x_1^5x_2). 
\end{aligned}
$$

\bibliographystyle{abbrv}
\bibliography{VertexLLT}

\end{document}